\documentclass{amsart}

 \usepackage[all]{xy}
\usepackage{lmodern}
\usepackage[dvipsnames,svgnames,x11names,hyperref]{xcolor}
\usepackage{mathtools,url,graphicx,verbatim,amssymb,enumerate,stmaryrd,nicefrac}
\usepackage[pagebackref,colorlinks,citecolor=Mahogany,linkcolor=Mahogany,urlcolor=Mahogany,filecolor=Mahogany]{hyperref}
\usepackage{microtype}
\usepackage[margin=1.5in]{geometry}
\usepackage{setspace}
\usepackage{tikz, tikz-cd}
\usetikzlibrary{matrix,calc,arrows,patterns,decorations.pathreplacing}
\usepackage{caption}

\newtheorem*{corollary*}{Corollary}

\theoremstyle{definition}

\title{The homotopy type of the PL cobordism category. II}

\author{Mauricio Gomez Lopez}
\email{gomezlom@lafayette.edu}
\address{Department of Mathematics \\
	       Lafayette College, Pardee Hall\\
	       Easton, PA, 18042 \\USA}

\date{\today}

\begin{document}

\maketitle 

\begin{abstract}
In this article, we prove the PL analogue of the theorem of Galatius, Madsen, Tillmann, and Weiss which describes the homotopy type of the smooth cobordism category. More specifically, we introduce the PL
Madsen-Tillmann spectrum $\mathbf{MT}PL(d)$ and prove that there is a 
weak homotopy equivalence of the form 
$B\mathsf{Cob}^{\mathrm{PL}}_d \simeq \Omega^{\infty-1}\mathbf{MT}PL(d)$. 
We also discuss how to adjust the methods of this paper to obtain the topological version of our main result.
\end{abstract}

\tableofcontents

\section{Introduction} \label{introduction}

The goal of the present paper is to complete the project initiated by the author in \cite{GL}. In that article, we started the proof of the piecewise linear version of the main theorem of Galatius, Madsen, Tillmann, and Weiss from \cite{GMTW}. This result gives an explicit description of the homotopy type of the smooth cobordism category $\mathsf{Cob}_d$. More specifically, the authors of \cite{GMTW} proved that there is a weak homotopy equivalence
\begin{equation} \label{hpty.cob}
B\mathsf{Cob}_d \simeq \Omega^{\infty-1}\mathbf{MT}O(d),
\end{equation}
where $\mathbf{MT}O(d)$ is the Madsen-Tillmann spectrum.  

In \cite{GL}, our strategy was to adapt the proof of Galatius and Randall-Williams given in \cite{GRW} for (\ref{hpty.cob}) in the context of piecewise linear topology. To carry out this plan, we introduced a piecewise linear analogue $\mathsf{Cob}^{\mathrm{PL}}_d$ of the smooth cobordism category. This construction required introducing a PL version $\Psi_d(\mathbb{R}^N)$ of the space of smooth manifolds defined by Galatius in \cite{Ga}. As in the smooth case, our spaces of PL manifolds assemble into a spectrum $\Psi_d$, and the main result from \cite{GL} is the following. 

\theoremstyle{plain} \newtheorem{thma}{Theorem}[section]
\renewcommand{\thethma}{\Alph{thma}}

\begin{thma} \label{thma}
There is a weak homotopy equivalence $B\mathsf{Cob}^{\mathrm{PL}}_d \simeq \Omega^{\infty-1}\Psi_d$. 
\end{thma}

Our main theorem in the present paper is the following. 

\theoremstyle{plain} \newtheorem{thmb}[thma]{Theorem}

\begin{thmb} \label{thmb}
There is a weak equivalence of spectra $\Psi_d \simeq \mathbf{MT}PL(d)$. 
\end{thmb}

Combining Theorems A and B, we obtain the PL analogue of the main theorem from \cite{GMTW}. 

\theoremstyle{plain} \newtheorem{thmc}[thma]{Theorem}

\begin{thmc} \label{thmc}
There is a weak homotopy equivalence $B\mathsf{Cob}^{\mathrm{PL}}_d \simeq \Omega^{\infty-1}\mathbf{MT}PL(d)$. 
\end{thmc}

In the last two statements, $\mathbf{MT}PL(d)$ denotes the \textit{PL Madsen-Tillmann spectrum}, which we will introduce in Section \S \ref{section.grass}. The main results from PL topology that we will use as input for the proof of 
Theorem \ref{thmb} are the following: 

\begin{itemize}

\item[$\cdot$] Any PL submersion is a microfibration. Moreover, we can guarantee the existence of microlifts which are piecewise linear. 

\item[$\cdot$] The Bundle Representation Theorem of Kuiper and Lashof, which asserts that any 
$d$-dimensional PL microbundle contains a PL $\mathbb{R}^d$-bundle. 

\end{itemize}

Both of these results, along with other background material, will be reviewed in the appendices. 

Let us give an outline of the proof for Theorem \ref{thmb}. Roughly, the strategy is to make a PL adaptation of the scanning argument sketched by Hatcher in \cite{Hat}, which goes as follows: 

\begin{itemize}

\item[(i)] Elements of the space of smooth manifolds which are far away from the origin $\mathbf{0} \in \mathbb{R}^N$ get pushed to infinity.

\item[(ii)] On the other hand, for an element $W$ sufficiently close to $\mathbf{0}$, one chooses a point $x$ in $W$ and an open ball $B(x,\epsilon)$ centered at $x$. The intersection of $B(x,\epsilon)$ with $W$ must be diffeomorphic to a disk of dimension $d$.

\item[(iii)] Then, scanning the ball $B(x,\epsilon)$, the manifold $W$ is deformed into a $d$-dimensional plane. The amount of scanning performed depends on how close $W$ is to the origin. 

\end{itemize}

In Section \S \ref{first.section}, we review the definition of the simplicial set of PL manifolds 
$\Psi_d (\mathbb{R}^N)$ defined in \cite{GL} and collect some of its basic properties. We will also discuss the construction of the spectrum $\Psi_d$.  In Section \S \ref{section.marking}, we introduce a variant of the simplicial set $\Psi_d (\mathbb{R}^N)$, denoted by $\widetilde{\Psi}_d (\mathbb{R}^N)$, in which manifolds come equipped with base-points. In particular, a 0-simplex of $\widetilde{\Psi}_d (\mathbb{R}^N)$  will be a PL submanifold $W$ of $\mathbb{R}^N$ and a specified point $x\in W$. This will be the point at which we are going to scan $W$.  

In Section \S \ref{section.grass}, we define the PL analogues of the Grassmannian, the affine Grassmannian, and the Thom space $\mathrm{Th}(\gamma^{\perp}_{d,N})$, which will be denoted by $\mathrm{Gr}_d(\mathbb{R}^N)$, $A\mathrm{Gr}_d(\mathbb{R}^N)$, and
$A\mathrm{Gr}_d^+(\mathbb{R}^N)$ respectively.  
Informally, we can describe the $p$-simplices of these simplicial sets as follows: 

\begin{itemize}

\item[$\cdot$] In $\mathrm{Gr}_d(\mathbb{R}^N)$, a $p$-simplex is a family, parameterized by $\Delta^p$, of locally flat $d$-dimensional PL planes in $\mathbb{R}^N$ which intersect the origin $\mathbf{0}\in \mathbb{R}^N$.

\item[$\cdot$] A $p$-simplex in $A\mathrm{Gr}_d(\mathbb{R}^N)$ is a family over $\Delta^p$ of locally flat $d$-dimensional PL planes in $\mathbb{R}^N$. In this case, the planes are not required to intersect the origin. 

\item[$\cdot$] Finally, a $p$-simplex in $A\mathrm{Gr}_d^+(\mathbb{R}^N)$ is a family over $\Delta^p$
of PL submanifolds $W_x$ of $\mathbb{R}^N$, where either $W_x$ is empty or is a locally flat $d$-plane. 

\end{itemize}

In all of the above cases, the manifolds parameterized by $\Delta^p$ in a $p$-simplex assemble into a manifold $W$ such that the projection $W \rightarrow \Delta^p$ is a \textit{piecewise linear submersion of codimension $d$} (see Appendix \ref{sec.pl.sub}). Following the line of ideas in \cite{Hat}, we view $A\mathrm{Gr}_d^+(\mathbb{R}^N)$ as the one-point compactification of $A\mathrm{Gr}_d(\mathbb{R}^N)$. The spaces in the sequence $\{ A\mathrm{Gr}_d^+(\mathbb{R}^N) \}_{N}$ are going to be the levels of the PL Madsen-Tillmann spectrum $\mathbf{MT}PL(d)$. 

In Section \S \ref{section.grass}, we also introduce an enhancement of $\widetilde{\Psi}_d (\mathbb{R}^N)$, denoted by $\widetilde{\Psi}_d^{\hat{\circ}}(\mathbb{R}^N)$, where we specify normal data around the base-points. For a 0-simplex $(W,x)$ in $\widetilde{\Psi}_d (\mathbb{R}^N)$, this amounts to choosing an open ball around $x$. For a simplex of degree $p> 0$, the open balls of all the fibers assemble into a fiber bundle.

To prove our main result, we will show that there is a chain of weak equivalences of the form 
\begin{equation} \label{main.composite}
\xymatrix{A\mathrm{Gr}_d^+(\mathbb{R}^N) \ar[r]^{(3)} & \widetilde{\Psi}_d^{\hat{\circ}}(\mathbb{R}^N) 
\ar[r]^{\hspace{-0.1cm}(2)} &
\widetilde{\Psi}_d (\mathbb{R}^N) \ar[r]^{\hspace{-0.1cm}(1)}  & \Psi_d (\mathbb{R}^N).
}
\end{equation}

Each one of these equivalences represents a step in the sketch given by Hatcher in \cite{Hat}. In (1), we choose the base-points at which we are going to perform scanning. In (2), we fix the open balls that we are going to scan. Finally, equivalence (3) is the result of the scanning process. 

Letting $\mathcal{G}_N$ denote the composite (\ref{main.composite}), we will have that the sequence $\{\mathcal{G}_N\}_N$ defines a map of spectra $\mathcal{G}$, which by the above discussion will be a weak equivalence. 

We conclude this paper by pointing out that, with minor adjustments, our methods also work for spaces of topological manifolds. Thus, after making the necessary modifications, we obtain an alternate proof for the topological analogue of Theorem \ref{thmb}, proven by A. Kupers and the author in \cite{GLK}. 

\theoremstyle{definition}  \newtheorem*{ack}{Acknowledgments}

\begin{ack}

The author would like to thank Oscar Randal-Williams for helpful discussions regarding this work and the Departments of Mathematics 
at the University of Oregon and Lafayette College for providing support and positive working environments during the development of this project. 
The author is also grateful to the anonymous referee for their valuable comments and suggestions.
\end{ack}

\subsection{Definitions from piecewise linear topology}

In this paper, we will work mainly with objects and morphisms from the \textit{piecewise linear category}. 
The main definitions that the reader needs to know are the following. 

\theoremstyle{definition} \newtheorem{polyhedron}{Definition}[section]

\begin{polyhedron} \label{polyhedron}

Let $X$ be a topological space. A \textit{piecewise linear chart} (or PL chart for short) for $X$ is a continuous embedding $h: |K| \rightarrow X$, where $|K|$ is the geometric realization of a finite simplicial complex $K$. Two piecewise linear charts 
$h_1: |K| \rightarrow X$ and $h_2: |K'| \rightarrow X$  are said to be \textit{compatible} if either $h_1(|K|)$ and $h_2(|K'|)$ are disjoint,  or there exists subdivisions of $K$ and $K'$ which triangulate $h_1^{-1}(h_1(|K|)\cap h_2(|K'|))$ and $h_2^{-1}(h_1(|K|)\cap h_2(|K'|))$ respectively and such that the composition
\[
\xymatrix{ 
h_1^{-1}(h_1(|K|)\cap h_2(|K'|)) \ar[r]^{\hspace{0.35cm}h_1} & h_1(|K|)\cap h_2(|K'|) \ar[r]^{\hspace{-0.35cm}h^{-1}_2} & h_2^{-1}(h_1(|K|)\cap h_2(|K'|))}
\]
is linear on each simplex of the subdivision of $K$ contained in $h_1^{-1}(h_1(|K|)\cap h_2(|K'|))$. 
A \textit{piecewise linear space} (or PL space for short) is a second countable Hausdorff space $X$ together with a collection $\Lambda$ of PL charts satisfying the following properties: 

\begin{itemize}

\item[(i)] Any two charts in $\Lambda$ are compatible. 

\item[(ii)] For any point $x \in X$, there is a PL chart $h: |K| \rightarrow X$ in $\Lambda$ such that $h(|K|)$ is a neighborhood of $x$. 

\item[(iii)] The collection $\Lambda$ is maximal. That is, if $h: |K| \rightarrow X$ is a PL chart which is compatible with every chart in $\Lambda$, then $h: |K| \rightarrow X$ must also belong to $\Lambda$. 

\end{itemize}

A collection of charts satisfying conditions (i), (ii), (iii) is called a \textit{PL structure for $X$}. 
\end{polyhedron}

\theoremstyle{definition} \newtheorem{subpolyhedron}[polyhedron]{Definition}

\begin{subpolyhedron} \label{subpolyhedron}

Consider a PL space $X$ with PL structure 
$\Lambda$, and let $X_0 \subseteq X$ be a subspace of $X$. We say that $X_0$ is a \textit{PL subspace of $X$} if, for any point 
$x \in X_0$, we can find a chart $h: |K| \rightarrow X$ in the PL structure $\Lambda$
whose image is contained in $X_0$ and has the property that
$\mathrm{Im}\hspace{0.05cm}h$ is a neighborhood of $x$ in $X_0$. 
Note that $X_0$ is itself a PL space. 
The PL structure on $X_0$ is the subcollection $\Lambda_0$ of $\Lambda$ consisting of all charts $h: |K| \rightarrow X$ such that 
$\mathrm{Im}\hspace{0.05cm}h \subseteq X_0$. 
\end{subpolyhedron}

\theoremstyle{definition} \newtheorem{pl.examples}[polyhedron]{Example}

\begin{pl.examples} \label{pl.examples}

If $K$ is a locally finite simplicial complex, then $|K|$ admits a canonical piecewise linear structure. Namely, we take the collection of all charts which are compatible with the inclusions $|L| \hookrightarrow |K|$, where $L$ is any finite subcomplex of $K$. 
Also, for any Euclidean space $\mathbb{R}^N$, we can define a natural piecewise linear structure 
on $\mathbb{R}^N$ by taking all charts which are compatible with 
inclusions of the form $|K| \hookrightarrow \mathbb{R}^N$, 
where $K$ is any simplicial complex of linear simplices in $\mathbb{R}^N$. 
\end{pl.examples}

\theoremstyle{definition} \newtheorem{pl.map}[polyhedron]{Definition}

\begin{pl.map} \label{pl.map}
A continuous map $f: P \rightarrow Q$ between two PL spaces $(P, \Lambda_P)$ and $(Q, \Lambda_Q)$ is a \textit{PL map} if, for any $x$ in $P$, we can find piecewise linear charts $h_0:|K| \rightarrow P$ and $h_1: |L| \rightarrow Q$ near $x$ and $f(x)$ respectively which satisfy the following: 

\begin{itemize} 

\item[(i)] $ f(\mathrm{Im}\hspace{0.05cm}h_0) \subseteq \mathrm{Im}\hspace{0.05cm}h_1$, and 
$ \mathrm{Im}\hspace{0.05cm}h_0$, $\mathrm{Im}\hspace{0.05cm}h_1$ are neighborhoods of $x$ and $f(x)$ respectively.

\item[(ii)] The composite $h_1^{-1} \circ f \circ h_0$ maps each simplex of $|K|$ linearly to a simplex of $|L|$. 
\end{itemize}
The map $f$ is a \textit{PL homeomorphism} if its both a PL map and a homeomorphism between the underlying topological spaces $P$ and $Q$. 
\end{pl.map}

\theoremstyle{definition} \newtheorem{PL.cat}[polyhedron]{Definition}

\begin{PL.cat} \label{PL.cat}
The category $\mathbf{PL}$ is the category whose objects are PL spaces and morphisms are PL maps. 
\end{PL.cat}

Using the notion of PL homeomorphism, we can formulate the following definition. 

\theoremstyle{definition} \newtheorem{pl.manifold}[polyhedron]{Definition}

\begin{pl.manifold} \label{pl.manifold}
A PL space $(M, \Lambda)$  is a \textit{d-dimensional PL manifold} if every point $x \in M$ has an open neighborhood which is PL homeomorphic to an open subspace of
the Euclidean space $\mathbb{R}^d$ (equipped with the standard PL structure defined in Example \ref{pl.examples}). Moreover, we say that $(M, \Lambda)$ is a \textit{d-dimensional PL manifold with boundary} 
if each $x\in M$ has a neighborhood which is PL homeomorphic to an open subspace of the half-space $\mathbb{R}^d_{\geq 0}$.  
\end{pl.manifold}

In Appendix \ref{appendix.pl}, we will give more detailed explanations of other notions from piecewise linear topology that we will use in this article. Specifically, in Appendix \ref{appendix.pl}, we discuss PL submersions, PL microbundles, and regular neighborhoods.  

\section{Fundamentals of spaces of PL manifolds} \label{first.section}

\subsection{Spaces of manifolds and quasi-PL spaces} In this section, we will review some of the basic constructions and results from \cite{GL} that we will use in this paper. We start by defining the \textit{space of PL manifolds} $\Psi_d(U)$, which was the main object studied in \cite{GL}. 
As we did in \cite{GL}, instead of just defining $\Psi_d(U)$ as a simplicial set, we shall define it as a \textit{PL set.} 
We introduce this notion in the next definition. 

\theoremstyle{definition} \newtheorem{PLset}{Definition}[section]

\begin{PLset} \label{PLset}
A \textit{PL set} is a functor of the form $\mathcal{F}: \mathbf{PL}^{op} \rightarrow \mathbf{Sets}$.
The function $\mathcal{F}(f): \mathcal{F}(P) \rightarrow \mathcal{F}(Q)$
induced by a PL map $f: Q \rightarrow P$ shall be denoted simply by $f^*$. 
Moreover, we say that $\mathcal{F}': \mathbf{PL}^{op} \rightarrow \mathbf{Sets}$ is a 
\textit{PL subset of} $\mathcal{F}$ if $\mathcal{F}'(P) \subseteq \mathcal{F}(P)$ 
for all PL spaces $P$ and, for any PL map
$f: Q \rightarrow P$, the induced function 
$f^*: \mathcal{F}'(P) \rightarrow \mathcal{F}'(Q)$ is the restriction of the 
function $f^*: \mathcal{F}(P) \rightarrow \mathcal{F}(Q)$ on $\mathcal{F}'(P)$.  
\end{PLset}

In \cite{GL}, we defined $\Psi_d(U)$ assuming that $U$ was an open subspace
of some $\mathbb{R}^N$. However, in this paper, we shall make this construction slightly more general. 

\theoremstyle{definition} \newtheorem{spaceman}[PLset]{Definition}

\begin{spaceman} \label{spaceman}   
Fix an $N$-dimensional PL manifold $M$ without boundary
and a non-negative integer $d \leq N$. For an open set $U$ of $M$, 
the \textit{space of d-dimensional PL manifolds in} 
$U$ is the PL set $\Psi_d(U):  \mathbf{PL}^{op} \rightarrow \mathbf{Sets}$ defined as follows: 

\begin{itemize}
\item[$\cdot$] For a PL space $P$, $\Psi_d(U)(P)$ is the set of all closed PL subspaces $W$ of 
$P\times U$ with the property that the restriction of the standard projection $P\times U \rightarrow P$ on 
$W$ is a PL submersion of codimension $d$. 

\item[$\cdot$] The function $f^*: \Psi_d(U)(P) \rightarrow \Psi_d(U)(Q)$ induced by a PL map $f: Q \rightarrow P$ sends an element $W$ 
to the \textit{pull-back} $f^*W$. That is, 
$f^*W$ is the PL subspace of  
$Q\times U$ given by $\{(\lambda,x) \hspace{0.05cm}| \hspace{0.05cm} (f(\lambda) ,x) \in W\}$.  
\end{itemize}
\end{spaceman}

 In Appendix \ref{sec.pl.sub}, we review the definition of PL submersion and collect some basic facts about this type of map.  
 
 \theoremstyle{definition} \newtheorem{locally.flat}[PLset]{Remark}

\begin{locally.flat} \label{locally.flat}
Throughout this paper, whenever we consider a PL set of the form $\Psi_d(U)$, we shall always assume that 
$d \geq 1$ and $N \geq 2d + 2$. The condition $N \geq 2d + 2$ will allow us to 
apply certain general position results in some of our arguments (e.g., Proposition \ref{germ.pathconn}).
On the other hand, the inequalities  $d \geq 1$ and $N \geq 2d + 2$ combined imply 
that $N - d \geq 3$.  By having $N - d \geq 3$, 
we can guarantee that every $d$-dimensional PL submanifold $W$ of $U$ is \textit{locally flat}. That is, for any $x$ in $W$, we can find an open PL 
embedding $h: \mathbb{R}^N \rightarrow U$ which maps the origin $\mathbf{0} \in \mathbb{R}^N$ to $x$, 
and maps the subspace $\mathbb{R}^d \subset \mathbb{R}^N$ to $W$.  
The fact that PL submanifolds are locally flat in codimension at least 3 is an immediate consequence of 
Theorem 1 of \cite{Zeeman}, which states that any proper pair of PL balls $(B^{N}, B^{d})$ with $N - d \geq 3$ 
is PL homeomorphic to the standard ball pair $(I^N, I^d)$, where $I^N$ (resp. $I^d$) is the $N$-fold 
(resp. $d$-fold) cartesian product of $[-1,1]$ and where we identify $I^d$ with the PL subspace of $I^N$
defined by setting the last $N-d$ coordinates equal to $0$. 
Furthermore, given any element $W$ of a set $\Psi_d(U)(P)$,
the condition $N - d \geq 3$ guarantees that the pair $(\pi, \pi|_W)$ consisting of the standard projection
$\pi: P\times U \rightarrow P$ and the restriction $\pi|_W: W \rightarrow P$ is a \textit{relative PL submersion}. The definition of relative PL submersion is given in Appendix \ref{sec.pl.sub} 
(see Definition \ref{pl.sub.relative}). 
\end{locally.flat}

\theoremstyle{definition} \newtheorem{rest.base}[PLset]{Remark}

\begin{rest.base} \label{rest.base}
Let $\mathcal{F}: \mathbf{PL}^{op} \rightarrow \mathbf{Sets}$ be a PL set and
$P$ a PL space. If $W$ is an element of $\mathcal{F}(P)$ and $Q$ is a PL subspace of $P$, 
we can obtain an element $W_Q$ of 
$\mathcal{F}(Q)$ by \textit{restricting} $W$ over $Q$. In other words, $W_{Q}$ is the  pull-back of $W$ 
along the obvious inclusion $Q \hookrightarrow P$. 
Throughout this paper, we shall call $W_Q$ \textit{the restriction of $W$ over $Q$}. If $\mathcal{F}$ is a PL set of the form $\Psi_d(U)$, 
the element $W_Q$ has a more concrete geometric interpretation: 
If $\pi: W \rightarrow P$ is the restriction of the standard projection $P\times U \rightarrow P$ on $W$, 
then $W_Q$ is the element of $\Psi_d(U)(Q)$ obtained by taking the fibers of
$\pi: W \rightarrow P$ over $Q$. In particular, if $Q$ is a one-point subspace $\{x\}$ of $P$, the restriction 
$W_{\{x\}}$ is the fiber of $W$ over $x$. As we did in \cite{GL}, we shall denote the fiber $W_{\{x\}}$
simply by $W_x$. 
\end{rest.base}

If $V \subseteq U$ are open sets of the manifold $M$, we can define a \textit{restriction map} $r_{U,V}: \Psi_d(U) \rightarrow \Psi_d(V)$ 
between PL sets (i.e., a natural transformation)
which sends an element $W$ of $\Psi_d(U)(P)$ to the intersection of $W$ with $P \times V$. Therefore, the assignments  
$U \mapsto \Psi_d(U)$ and $(V\subseteq U) \mapsto r_{U,V}$
define a functor $\mathcal{O}(M)^{op} \rightarrow \mathbf{PLsets}$, where $\mathcal{O}(M)$ denotes the poset of open sets of $M$
and $\mathbf{PLsets}$ is the category of PL sets.
In fact, since the property defining a PL submersion is local, 
we can easily verify the following result, which we state without proof.   

\theoremstyle{plain} \newtheorem{space.sheaf}[PLset]{Proposition}

\begin{space.sheaf} \label{space.sheaf}
The correspondences $U \mapsto \Psi_d(U)$ and $(V\subseteq U) \mapsto r_{U,V}$ define a sheaf of PL sets on $M$. 
That is, for any open set $U$ of $M$ and any open cover $\{ U_i \}_i$ of $U$, we have an equalizer diagram of PL sets
of the form 
\[
\xymatrix{ \Psi_d(U) \ar[r] & \prod_i \Psi_d(U_i) \ar@<0.7ex>[r] \ar@<-0.7ex>[r]  & \prod_{i,j} \Psi_d(U_i \cap U_j).
}
\]
\end{space.sheaf}

As explained in \cite{GL}, any PL set of the form $\Psi_d(U)$ is in fact a
\textit{quasi-PL space.} We will review this notion in the next definition. 

\theoremstyle{definition} \newtheorem{quasiPL}[PLset]{Definition}

\begin{quasiPL} \label{quasiPL}
A \textit{quasi-PL space} is a PL set $\mathcal{F}: \mathbf{PL}^{op} \rightarrow \mathbf{Sets}$ which satisfies the following gluing condition: Given any PL space $P$ and any locally finite collection of closed PL subspaces $\{ Q_i \}_i$ which covers $P$, the diagram of restriction maps
\begin{equation} \label{gluing.condition}
\xymatrix{ \mathcal{F}(P) \ar[r] & \prod_i  \mathcal{F}(Q_i)  \ar@<0.7ex>[r] \ar@<-0.7ex>[r]  & 
\prod_{i,j}  \mathcal{F}(Q_i \cap Q_j)
}
\end{equation}
is an equalizer diagram of sets.  
\end{quasiPL}

Definition \ref{quasiPL} should be compared with that of a \textit{quasi-topological space} used by Gromov in \cite{Gr}. In Definition \ref{quasiPL}, 
we restrict to the category $\mathbf{PL}$, and the gluing condition is required to hold for arbitrary locally finite families of closed PL subspaces. 
On the other hand, for a quasi-topological space, the corresponding gluing condition is required only for 
finite collections of closed subsets. Moreover, we point out that any quasi-PL space $\mathcal{F}$ 
is automatically a sheaf on PL spaces, i.e., the gluing condition described in (\ref{gluing.condition}) implies the
analogous gluing condition for arbitrary open covers. This fact can be proven using
standard triangulation techniques from PL topology and Lemma 2.1.6 from \cite{Jo}.    

Next, we shall review how an arbitrary PL set $\mathcal{F}$ induces naturally a simplicial set $\mathcal{F}_{\bullet}$. 
Before reviewing this construction, we need to introduce some conventions. 

\theoremstyle{definition} \newtheorem{first.conv}[PLset]{Convention}

\begin{first.conv} \label{first.conv}
For any non-negative integer $p$, we will denote by $[p]$ the set $\{0, 1, \ldots, p \}$. As it is normally done in the literature, we will denote by $\Delta$ the category whose objects are the sets $[p]$ and morphisms are non-decreasing functions. In this paper, the \textit{standard geometric p-simplex} $\Delta^p$ will be the convex hull of the vectors of the standard basis in $\mathbb{R}^{p+1}$. Note that any morphism $\eta: [p] \rightarrow [q]$ in the category $\Delta$ induces a canonical linear map $\Delta^p \rightarrow \Delta^q$ between geometric simplices. We will denote this linear map by $\widetilde{\eta}$. Also, notice that the correspondences 
$[p] \mapsto \Delta^p$ and $\eta \mapsto \widetilde{\eta}$ assemble into an inclusion functor of the form 
$\mathcal{I}: \Delta \hookrightarrow \mathbf{PL}$.
Thus, we can regard $\Delta$ as a (non-full) subcategory of $\mathbf{PL}$.   
\end{first.conv}

\theoremstyle{definition} \newtheorem{underlying}[PLset]{Definition}

\begin{underlying} \label{underlying}

For any PL set $\mathcal{F}: \mathbf{PL}^{op} \rightarrow \mathbf{Sets}$,  
we define \textit{the underlying simplicial set of} $\mathcal{F}$ to be 
the simplicial set $\mathcal{F}_{\bullet}$ obtained by taking the composition 
\[
\xymatrix{ \Delta^{op} \hspace{0.1cm} \ar@{^{(}->}[r]^{\mathcal{I}^{op}} &  \mathbf{PL}^{op} \ar[r]^{\mathcal{F}} &  \mathbf{Sets},
}
\]
where $\mathcal{I}: \Delta \hookrightarrow \mathbf{PL}$ is the functor defined in Convention \ref{first.conv}. 
\end{underlying}

\theoremstyle{definition} \newtheorem{second.conv}[PLset]{Convention}

\begin{second.conv} \label{second.conv}
Consider a PL set of the form $\Psi_d(U)$ and let $\Psi_d(U)_{\bullet}$
be its underlying simplicial set. According to the previous definition, the structure map
of $\Psi_d(U)_{\bullet}$ corresponding to a morphism $\eta \in \Delta([q],[p])$ is the function
$\widetilde{\eta}^*: \Psi_d(U)_p \rightarrow \Psi_d(U)_q$ which sends an element $W \in \Psi_d(U)_p = \Psi_d(U)(\Delta^p)$
to its pull-back $\widetilde{\eta}^*W$ along the canonical linear map $\widetilde{\eta}^*: \Delta^q \rightarrow \Delta^p$
induced by $\eta: [q] \rightarrow [p]$.
From now on, we shall denote the structure map
$\widetilde{\eta}^*: \Psi_d(U)_p \rightarrow \Psi_d(U)_q$ simply by $\eta^*$, and we shall denote any
pull-back of the form $\widetilde{\eta}^*W$ by $\eta^*W$.  

\end{second.conv}

\theoremstyle{definition} \newtheorem{pointed}[PLset]{Remark}

\begin{pointed} \label{pointed}
 We point out that $\Psi_d(U)_{\bullet}$ is a \textit{pointed} simplicial set, 
 where the base-point is the subsimplicial set of all simplices $W$ such that $W= \varnothing$. 
 Equivalently, the base-point consists of all degeneracies of the \textit{empty} $0$-simplex.
\end{pointed}

Consider a PL set $\mathcal{F}$ and let $\mathcal{F}_{\bullet}$ be its underlying simplicial set. 
If $\mathcal{F}$ happens to be a quasi-PL space, the simplicial set $\mathcal{F}_{\bullet}$ has the following property. 

\theoremstyle{plain} \newtheorem{quasi.kan}[PLset]{Proposition}

\begin{quasi.kan} \label{quasi.kan}
The underlying simplicial set $\mathcal{F}_{\bullet}$ of a quasi-PL space $\mathcal{F}$ is Kan. 
\end{quasi.kan}

See
Theorem 2.13 in \cite{GL} for a proof of this proposition. 
As mentioned before, the PL sets $\Psi_d(U)$ are examples of quasi-PL spaces. 

\theoremstyle{plain} \newtheorem{psi.quasi}[PLset]{Proposition}

\begin{psi.quasi} \label{psi.quasi}
Fix two non-negative integers $d \leq N$. For any open set $U \subseteq M$ in an $N$-dimensional PL manifold $M$ without boundary, the functor $\Psi_d(U): \mathbf{PL}^{op} \rightarrow \mathbf{Sets}$ is a quasi-PL space.  In particular, the underlying simplicial set $\Psi_d(U)_{\bullet}$ is Kan.  
\end{psi.quasi} 

See Theorem 2.15 in \cite{GL} for a proof of this result. 
We point out that, in \cite{GL}, we proved  that $\Psi_d(U)$ 
is a quasi-PL space
in the special case when 
$U$ is an open subset of 
$\mathbb{R}^N$. However, the proof given in \cite{GL} follows through without 
much difficulty if we assume instead that $U$ is  
contained in an arbitrary PL manifold $M$ without boundary.  

One reason why PL sets (and, in particular, quasi-PL spaces) are useful is that they allow one to solve homotopy-theoretical questions geometrically. To make this assertion more precise, we need to recall the following notion introduced in \cite{GL}.

\theoremstyle{definition} \newtheorem{concordance}[PLset]{Definition}

\begin{concordance} \label{concordance} 
Consider a PL set $\mathcal{F}$ and let $P$ be an arbitrary PL space. 
Two elements $W_0$ and $W_1$ of the set $\mathcal{F}(P)$ are said to be \textit{concordant} if there exists a $\widehat{W}$ in $\mathcal{F}([0,1]\times P)$ such that, for $j =0$ and $j=1$, the element $W_j$ is equal to the pull-back of $\widehat{W}$ along the canonical inclusion 
$i_j: P\hookrightarrow [0,1]\times P$ defined by  $i_j(x) = (j, x)$. 
In this case, we say that $\widehat{W}$ is a \textit{concordance} from $W_0$ to $W_1$. 
\end{concordance}

\theoremstyle{definition} \newtheorem{remark.concordance1}[PLset]{Remark}

\begin{remark.concordance1} \label{remark.concordance1}
Fix a PL set $\mathcal{F}$, and let $W_0$ and $W_1$ be two elements of $\mathcal{F}(\Delta^p)$. In other words, $W_0$ and $W_1$ are $p$-simplices of the underlying simplicial set $\mathcal{F}_{\bullet}$. For $j=0$ and $j=1$, let $f_j: \Delta^p_{\bullet} \rightarrow \mathcal{F}_{\bullet}$ be the classifying map of $W_j$. In Proposition 2.21 of \cite{GL}, 
we described a procedure which takes as input a concordance from $W_0$ to $W_1$  and produces a homotopy between the geometric realizations $|f_0|$ and $|f_1|$.  
Therefore, if $W_0$ and $W_1$ are $p$-simplices of $\Psi_d(U)_{\bullet}$ with classifying maps $f_0: \Delta^p_{\bullet} \rightarrow \mathcal{F}_{\bullet}$
and $f_1: \Delta^p_{\bullet} \rightarrow \mathcal{F}_{\bullet}$ respectively, 
we can produce a homotopy from $|f_0|$ to $|f_1|$ by simply constructing a concordance $\widehat{W} \in \Psi_d(U)( [0,1]\times \Delta^p)$
from $W_0$ to $W_1$. We will exploit this observation in many of the arguments that we will elaborate in this paper. 
 \end{remark.concordance1}
 
Fix a PL set $\mathcal{F}: \mathbf{PL}^{op} \rightarrow \mathbf{Sets}$ and
let $\mathcal{F}'$ be a
PL subset of $\mathcal{F}$. Evidently, the underlying simplicial set 
$\mathcal{F}'_{\bullet}$ is a subsimplicial set of $\mathcal{F}_{\bullet}$. 
Using the subdivision techniques from \cite{GL}, we can prove the following.  

\theoremstyle{plain} \newtheorem{subdivision.cover}[PLset]{Proposition}

\begin{subdivision.cover} \label{subdivision.cover}

Let $\mathcal{F}: \mathbf{PL}^{op} \rightarrow \mathbf{Sets}$ be a PL set
and suppose that 
$\mathcal{F}'$ is a PL subset of $\mathcal{F}$ with the property that, 
for any PL space $P$ and any element $W\in\mathcal{F}(P)$, there exists an open cover $\mathcal{U}$ of $P$ such that, 
for each open set $U \in \mathcal{U}$,
 the restriction $W_U$ is an element of $\mathcal{F}'(U)$.   
Then, under these assumptions, 
the standard inclusion $\mathcal{F}'_{\bullet} \hookrightarrow \mathcal{F}_{\bullet}$ is a weak homotopy equivalence. 
\end{subdivision.cover} 

The proof of this proposition will be postponed until Appendix \ref{subdivision.gl}, 
where we will also review briefly the subdivision methods from \cite{GL}. 

\subsection{The spectrum of piecewise linear manifolds} \label{spectrum.sec}
 In this paper, we will mainly focus on quasi-PL spaces of the form $\Psi_d(\mathbb{R}^N)$; i.e., we will work with spaces of PL manifolds inside $\mathbb{R}^N$. 

In \cite{GL}, we explained how to assemble the sequence of spaces $\{ |\Psi_d(\mathbb{R}^N)_{\bullet}| \}_N$
into a spectrum $\Psi_d$. In this paper, we will use an alternative model for this spectrum $\Psi_d$,
one where the levels of  $\Psi_d$ are given by simplicial sets rather than topological spaces.
Before we explain how to obtain this alternative definition of $\Psi_d$, we need to introduce the following simplicial set: 
Let $\mathbb{R}\cup \{ \infty\}$ be the one-point compactification of the real line $\mathbb{R}$ 
with its usual topology. 
From now on, we will denote by $S^1$ the subsimplicial set of $\mathrm{Sing}(\mathbb{R}\cup \{ \infty\})$ 
consisting of those
continuous functions $f: \Delta^p \rightarrow \mathbb{R}\cup \{ \infty\}$ whose restriction on the complement $\Delta^p - f^{-1}(\infty)$ is piecewise linear. 
Note that $S^1$ is a pointed simplicial set; the base-point is defined by all constant functions $\Delta^p \rightarrow \{\infty\}$.  Also, recall that the natural base-point of $\Psi_d(\mathbb{R}^N)_{\bullet}$ is the subsimplicial set consisting of all degeneracies of the empty $0$-simplex $\varnothing$. As we did in \cite{GL}, we shall denote this subsimplicial set by $\varnothing_{\bullet}$.   

Let $\iota:  \Psi_d(\mathbb{R}^N)_{\bullet} \hookrightarrow \Psi_d(\mathbb{R}^{N+1})_{\bullet}$ be the inclusion of simplicial sets induced by the standard inclusion $\mathbb{R}^N \hookrightarrow \mathbb{R}^{N+1}$. 
In this discussion, we shall denote the image of a $p$-simplex $W$ of $\Psi_d(\mathbb{R}^N)_{\bullet}$ under $\iota$ 
simply by $W$. 
Now, given an arbitrary $p$-simplex $(f,W)$ of the product $S^1\times \Psi_d(\mathbb{R}^N)_{\bullet}$, we can produce a $p$-simplex in $\Psi_d(\mathbb{R}^{N+1})_{\bullet}$ by pushing the fibers of $W$ towards infinity along the $N+1$-direction using the function $f$. More rigorously, we can do the following: 

\begin{itemize}
\item[$\cdot$] Let $U$ denote the complement $\Delta^p - f^{-1}(\infty)$, and let $W_U$ be the restriction
of $W$ over $U$.  

\item[$\cdot$] Let $W + f$ be the image of $W_U$ under the PL embedding 
$U\times\mathbb{R}^N \rightarrow U\times \mathbb{R}^{N+1}$ defined by 
$(\lambda, x_1, \ldots, x_N) \mapsto (\lambda, x_1, \ldots, x_N, f(\lambda))$.  
\end{itemize}

Clearly, $W+f$ is an element of the set $\Psi_d(\mathbb{R}^{N+1})(U)$. However, since
 $f: \Delta^p \rightarrow \mathbb{R}\cup \{ \infty\}$ is a continuous map, $W + f$ is actually closed as a PL subspace of 
 $\Delta^p \times \mathbb{R}^{N+1}$. Therefore, $W+f$ is a $p$-simplex of $\Psi_d(\mathbb{R}^{N+1})_{\bullet}$ 
 which has empty fibers over $f^{-1}(\infty)$. 

Note that if $(W,f)$ is a $p$-simplex satisfying $W = \varnothing$ or $f(\Delta^p)=\{\infty\}$, then $W+f$ will be the empty $p$-simplex in $\Psi_d(\mathbb{R}^{N+1})_{\bullet}$. Thus, for any positive integer $N$, the map of simplicial sets 
\vspace{0.15cm}
\begin{equation} \label{prespectrum}
S^1\times \Psi_d(\mathbb{R}^N)_{\bullet} \rightarrow \Psi_d(\mathbb{R}^{N+1})_{\bullet} \qquad (W,f) \mapsto W +f
\end{equation}
factors through the smash product $S^1 \wedge \Psi_d(\mathbb{R}^N)_{\bullet}$. 
With this observation, we can now define the spectrum $\Psi_d$. 

\theoremstyle{definition} \newtheorem{spectrum.man}[PLset]{Definition}

\begin{spectrum.man} \label{spectrum.man}
\textit{The spectrum of piecewise linear manifolds} $\Psi_d$ is the spectrum whose $N$-th level is $\Psi_d(\mathbb{R}^N)_{\bullet}$ and whose structure maps $S^1 \wedge \Psi_d(\mathbb{R}^N)_{\bullet} \rightarrow \Psi_d(\mathbb{R}^{N+1})_{\bullet}$ are the morphisms of simplicial sets induced by (\ref{prespectrum}).  
\end{spectrum.man}

\section{Spaces of manifolds with marking functions} \label{section.marking}

In this section,  we will introduce for any PL space $P$ a variant of the set $\Psi_d(\mathbb{R}^N)(P)$ where fibers come equipped with base-points. The main ingredient for this new construction is the following definition.  

\theoremstyle{definition} \newtheorem{marking}{Definition}[section]

\begin{marking} \label{marking} Let $P$ be a PL space and let $W$ be an element of $\Psi_d(\mathbb{R}^N)(P)$. A continuous map $f:P \rightarrow \mathbb{R}^N \cup\{\infty\}$ is said to be a \textit{marking function for W} if it satisfies the following conditions: 
\begin{itemize}
\item[(i)] For any $x\in P$, $f(x) = \infty$ if and only if $W_x = \varnothing$.  

\item[(ii)] If $W_x \neq \varnothing$, then $f(x)$ must be a point in the fiber $W_x$. 

\item[(iii)] If the fiber $W_x$ contains the origin $\mathbf{0}\in \mathbb{R}^N$, then $f(x)=\mathbf{0}$. 

\item[(iv)] The restriction of $f$ on the pre-image $f^{-1}(\mathbb{R}^N)$ is a piecewise linear map. 
\end{itemize}
 \end{marking}

\theoremstyle{definition} \newtheorem{remark.marking}[marking]{Remark}

\begin{remark.marking} \label{remark.marking} 
For any marking function $f: P \rightarrow \mathbb{R}^N \cup\{\infty\}$ 
of an element $W \in \Psi_d(\mathbb{R}^N)(P)$, 
we will often denote the pre-image  $f^{-1}(\mathbb{R}^N)$ by $V_f$.  
Note that if $\pi: W \rightarrow P$ is the restriction of the standard projection 
$P\times \mathbb{R}^N \rightarrow P$ on $W$, then
$V_f \subseteq P$ is equal to the image $\pi(W)$.
Now, let $\tilde{f}: V_f \hookrightarrow W$ 
be the piecewise linear inclusion defined by $\tilde{f}(x) = (x,f(x))$. 
Evidently, $\tilde{f}$ maps $V_f$ onto the graph of $f|_{V_f}$, 
which is a closed PL subspace of $W$.  Using the fact that 
$\pi: W \rightarrow P$ is a PL submersion of codimension $d$, 
one can easily verify that the diagram
\[
\xymatrix{ V_f  \hspace{0.1cm} \ar@{^{(}->}[r]^{\tilde{f}} & W \ar[r]^{\pi} & V_f 
}
\]
is a $d$-dimensional PL microbundle (we will review the definition of microbundle
in Appendix \ref{sec.pl.micro} ). More generally, since the map of pairs
\[
(\mathrm{pr}_1, \pi): (P\times \mathbb{R}^N, W) \longrightarrow (P,P) 
\]
is a relative 
PL submersion of codimensions $N$ and $d$ (see Definition \ref{pl.sub.relative}), 
one can also verify that the diagram
\[
\xymatrix{ V_f  \hspace{0.1cm} \ar@{^{(}->}[r]^{\hspace{-0.8cm}\tilde{f}} & 
(V_f\times\mathbb{R}^N,W) \ar[r]^{\hspace{0.8cm}(\mathrm{pr}_1,\pi)} & V_f 
}
\]
is a  PL $(N,d)$-microbundle pair over $V_f$ (the definition of microbundle pair will also
be discussed in Appendix \ref{sec.pl.micro}).   
\end{remark.marking}

The following is the main definition of this section.  

\theoremstyle{definition} \newtheorem{space.marking}[marking]{Definition}

\begin{space.marking} \label{space.marking} 
Let $P$ be a PL space. We define $\widetilde{\Psi}_d(\mathbb{R}^N)(P)$ to be the set of all tuples $(W,f)$, where $W$ is an element of $\Psi_d(\mathbb{R}^N)(P)$ and $f: P \rightarrow \mathbb{R}^N\cup \{ \infty \}$ is a marking function for $W$.  
\end{space.marking} 

For a PL map 
$g:Q \rightarrow P$, let $\widetilde{g^*}: \widetilde{\Psi}_d(\mathbb{R}^N)(P) \rightarrow   \widetilde{\Psi}_d(\mathbb{R}^N)(Q) $ be the function defined by $\widetilde{g^*}\big((W,f)\big) = (g^*W, f\circ g)$. One can easily check that the correspondences $P \mapsto \widetilde{\Psi}_d(\mathbb{R}^N)(P)$ and $g \mapsto \widetilde{g^*}$ define a contravariant functor $\widetilde{\Psi}_d(\mathbb{R}^N): \mathbf{PL}^{op} \rightarrow \mathbf{Sets}$.  In fact, we have the following proposition.  

\theoremstyle{plain} \newtheorem{marking.quasi}[marking]{Proposition}

\begin{marking.quasi} \label{marking.quasi} 
The functor $\widetilde{\Psi}_d(\mathbb{R}^N)$ is a quasi-PL space. In particular, the underlying simplicial set 
$\widetilde{\Psi}_d(\mathbb{R}^N)_{\bullet}$ is Kan.  
\end{marking.quasi}

\begin{proof} 
We need to prove that $\widetilde{\Psi}_d(\mathbb{R}^N)$ satisfies the 
gluing condition stated in Definition \ref{quasiPL}. 
Let then $P$ be a PL space and let $\{ P_i \}_{i\in\Lambda}$ be a locally finite collection 
of closed PL subspaces covering $P$. 
For each $i\in \Lambda$, let $(W_i, f_i)$ be an element of $\widetilde{\Psi}_d(\mathbb{R}^N)(P_i)$. Moreover, 
suppose that for any two indices $i$ and $j$ the pairs 
$(W_i, f_i)$ and $(W_j,f_j)$ agree over the intersection $P_i\cap P_j$. 
Since $\Psi_d(\mathbb{R}^N)$ is a quasi-PL space, there is a unique element $W$ of $\Psi_d(\mathbb{R}^N)(P)$ 
with the property that $W_{P_i} = W_i$ for all $i\in \Lambda$ (i.e., the restriction of $W$ over $P_i$ agrees with $W_i$). 
Also, by gluing all the functions $f_i$, we obtain a function 
$f:P \rightarrow \mathbb{R}^N\cup \{\infty\}$ which agrees with $f_i$ when restricted to $P_i$. It is evident that the function $f$ satisfies properties (i), (ii), (iii) from Definition \ref{marking}, 
so we only need to check that $f$ is piecewise linear on $f^{-1}(\mathbb{R}^N)$. 
To verify this, we shall assume that $P$ is the geometric realization $|K|$ of a locally finite simplicial complex $K$
(it is possible to make this assumption because any PL space admits a \textit{triangulation}.
See Chapter \S 3 of \cite{plhud} for a discussion of triangulations of PL spaces).  
Now, pick any point $x$ in $f^{-1}(\mathbb{R}^N) \subseteq |K|$. After subdividing $K$ if necessary, we can assume that the point $x$ is a vertex of $K$ and that all the simplices of the star $\mathrm{st}(x,K)$ (i.e., the union of all simplices of $K$ which have $x$ as a vertex) are contained in $f^{-1}(\mathbb{R}^N)$. 
Also, by subdividing $K$ even further, 
we can assume that each simplex of $\mathrm{st}(x,K)$ is contained in some PL subspace $P_i$. All of these assumptions guarantee that $f$ is piecewise linear on each simplex of $\mathrm{st}(x,K)$.
Therefore, $f$ is piecewise linear on 
$\mathrm{st}(x,K)$, which is a neighborhood of $x$ contained in $f^{-1}(\mathbb{R}^N)$.
Thus, since $x$ was arbitrary, we have shown that $f$ is locally piecewise linear at any point of $f^{-1}(\mathbb{R}^N)$.  
Consequently, $f$ is piecewise linear on all of $f^{-1}(\mathbb{R}^N)$, and
it follows that the pair $(W,f)$ is an element of $\widetilde{\Psi}_d(\mathbb{R}^N)(P)$.  

The previous argument shows that the diagram of restriction maps
\begin{equation} \label{gluing.condition.marking}
\xymatrix{ \widetilde{\Psi}_d(\mathbb{R}^N)(P) \ar[r] & \prod_i \widetilde{\Psi}_d(\mathbb{R}^N)(P_i) 
\ar@<0.7ex>[r] \ar@<-0.7ex>[r]  & 
\prod_{i,j} \widetilde{\Psi}_d(\mathbb{R}^N)(P_i \cap P_j)
}
\end{equation}
is an equalizer diagram of sets. In other words, the functor 
$\widetilde{\Psi}_d(\mathbb{R}^N): \mathbf{PL}^{op} \rightarrow \mathbf{Sets}$ 
satisfies the gluing condition stated in Definition \ref{quasiPL}.  
\end{proof}

Notice that a 0-simplex of $\widetilde{\Psi}_d(\mathbb{R}^N)_{\bullet}$  
is a $d$-dimensional piecewise linear submanifold $W$ of $\mathbb{R}^N$, 
which is closed as a subspace,  
with a choice of base-point.   The base-point has to be equal to the origin $\mathbf{0}$ of 
$\mathbb{R}^N$ if $\mathbf{0} \in W$. From now on, any 0-simplex of  $\widetilde{\Psi}_d(\mathbb{R}^N)_{\bullet}$  
with underlying manifold $W$ shall be denoted by $(W,x)$, 
where the second component indicates the base-point. If $W= \varnothing$, 
then we must have $x= \infty$.  
 
\subsection{Constructing concordances via scanning}   \label{concordances.via.scanning}

In this section, we will explain how 
to use scanning to construct concordances between elements of $\widetilde{\Psi}_d(\mathbb{R}^N)(P)$,
assuming that $P$ is a compact PL space.  
The following definition will play a central role in our constructions.    

\theoremstyle{definition} \newtheorem{elem.scan}[marking]{Definition}

\begin{elem.scan} \label{elem.scan} 
 
Fix a value $\epsilon>0$. A map $F:[0,1)\times \mathbb{R}^N \rightarrow \mathbb{R}^N$ is an \textit{elementary $\epsilon$-scanning map}  if it satisfies the following:  

\begin{itemize}

\item[(i)] $F$ is piecewise linear.

\item[(ii)] For all $t$ in $[0,1)$, $F_t$ is a piecewise linear homeomorphism from $\mathbb{R}^N$ to itself
that preserves the origin $\mathbf{0} \in \mathbb{R}^N$, i.e., 
$F_t(\mathbf{0}) = \mathbf{0}$ for all $t$ in $[0,1).$

\item[(iii)] $F_0$ is equal to the identity map $\mathrm{Id}_{\mathbb{R}^N}$.    

\item[(iv)] Given any value $r>0$, there is a $t_0$ in $[0,1)$ such that 
$[-r,r]^N \subset F_t\big( (-\epsilon,\epsilon)^N \big)$ for all $t \geq t_0$. 

\end{itemize}
\end{elem.scan}

One can construct elementary $\epsilon$-scanning maps via an iterated application of 
the Isotopy Extension Theorem (see Chapter 4 of \cite{RS}).   For many of the arguments that we will do in this section, 
it will be necessary to modify elements $(W,f)$ of a set 
$\widetilde{\Psi}_d(\mathbb{R}^N)(P)$ by pushing some fibers of $W$
to infinity while keeping other fibers fixed. 
This will be accomplished by applying the following construction 
as well as Lemma \ref{scan.lemma}  below.  

\theoremstyle{definition} \newtheorem{control.scan}[marking]{Definition}

\begin{control.scan} \label{control.scan}
Let  $F:[0,1)\times \mathbb{R}^N \rightarrow \mathbb{R}^N$ be an elementary $\epsilon$-scanning map and
$h: P \rightarrow [0,1]$ a piecewise linear function. Also, let $U_h$ denote the pre-image 
$h^{-1}([0,1))$. 
The piecewise linear homeomorphism 
$\widehat{F}: U_h\times\mathbb{R}^N  \rightarrow U_h\times\mathbb{R}^N$ 
defined by the formula
\[
\widehat{F}(x,y) = (x, F_{h(x)}(y))
\]
will be called the \textit{$\epsilon$-scanning map induced by $F$ and $h$}. 
\end{control.scan}

It is a routine exercise to check that the map $\widehat{F}$  is actually a PL homeomorphism. For the statement of Lemma \ref{scan.lemma}, we will need to use marking functions which are not necessarily globally defined.  

\theoremstyle{definition} \newtheorem{partial.marking}[marking]{Definition}

\begin{partial.marking} \label{partial.marking}
Let $W$ be an element of $\Psi_d(\mathbb{R}^N)(P)$. A function $f: V \rightarrow \mathbb{R}^N\cup\{\infty\}$ is a \textit{partially-defined marking function for $W$} if: 
\begin{itemize}

\item[(i)] The domain $V$ is a non-empty open subset of $P$, and

\item[(ii)] $f$ is a marking function for the element $W_V \in \Psi_d(\mathbb{R}^N)(V)$ obtained
by restricting $W$ over $V$.
\end{itemize}
\end{partial.marking} 

\theoremstyle{plain} \newtheorem{scan.lemma}[marking]{Lemma}

\begin{scan.lemma} \label{scan.lemma}
Fix a compact PL space $P$, an element $W$ of $\Psi_d(\mathbb{R}^N)(P)$, a value $\epsilon > 0$, an elementary 
$\epsilon$-scanning map $F:[0,1)\times \mathbb{R}^N \rightarrow \mathbb{R}^N$, and  
a piecewise linear function $h: P \rightarrow [0,1]$.  
Moreover, let us denote 
the pre-image $h^{-1}([0,1))$ by $U_h$.  
If for every $x$ in $h^{-1}(1)$ we have that the fiber $W_x$ does not intersect the cube $[-\epsilon,\epsilon]^N$, 
then there exists an element  $W^{F,h}$ of $\Psi_d(\mathbb{R}^N)(P)$ 
which satisfies the following: 

\begin{itemize}

\item[(i)] For any $x$ in $U_h$, the fiber $W^{F,h}_x$ is equal to $F_{h(x)}(W_x)$. 

\item[(ii)] $W^{F,h}_x = \varnothing$ for all $x$ in $P-U_h$.   

\end{itemize}

Furthermore, if $f: V \rightarrow \mathbb{R}^N\cup \{\infty\}$ is a partially-defined 
marking function for $W$ and the domain $V$ contains $U_h$, 
then the function $f^{F,h}: P \rightarrow \mathbb{R}^N\cup \{\infty\}$ defined by 
\begin{equation}  \label{new.marking.function}
f^{F,h}(x)   = \left\{
    \begin{array}{rl}
      F_{h(x)}(f(x)) & \textrm{if } x\in U_h\cap f^{-1}(\mathbb{R}^N) \\
       \infty & \textrm{if } x \in (P - U_h)\cup f^{-1}(\infty)
    \end{array} \right. 
\end{equation}
is a marking function for $W^{F,h}$. 
\end{scan.lemma}

Before jumping into the proof of Lemma \ref{scan.lemma}, 
let us say a few words about our strategy
for proving this result.
 Roughly speaking, we will construct the element 
  $W^{F,h}$ of $\Psi_d(\mathbb{R}^N)(P)$ by enlarging the cube 
  $[-\epsilon, \epsilon]^N$ at each point $x \in P$. In other words, 
  $W^{F,h}$ will be obtained by performing scanning 
  (via the elementary $\epsilon$-scanning map $F$) at each point
  $x \in P$.  
  The function $h:P \rightarrow [0,1]$  indicates how 
  much scanning we will do at each 
  point of $P$. If $h(x) = 0$, then we do not do any scanning at $x$. 
  In this case, $W$ and $W^{F,h}$ 
  will have the same fiber over $x$.  
  If $0 < h(x) < 1$, then the elementary $\epsilon$-scanning 
  map $F:[0,1) \times \mathbb{R}^N \rightarrow \mathbb{R^N}$ 
  will scan any part of the fiber $W_x$ lying 
  outside of $[-\epsilon, \epsilon]^N$ away from the origin. The closer $h(x)$ 
  is to 1, the stronger the scanning. 
  Finally, if $h(x)=1$, the 
  fiber $W_x$ is pushed all the way to infinity, making the new 
  fiber $W^{F,h}_x$ empty. 
 If $f:V \rightarrow \mathbb{R}^N\cup \{\infty\}$ is a partially-defined 
 marking function for $W$ so that $U_h \subseteq V$, then the statement of 
 Lemma \ref{scan.lemma} also says that we can scan 
 the images of $f$ along with the fibers of $W$
 to produce a marking function 
 $f^{F,h}: P \rightarrow \mathbb{R}^N\cup \{\infty\}$
 for $W^{F,h}$ (this is the function defined in (\ref{new.marking.function})). 
 An interesting feature of the scanning process 
 described in Lemma \ref{scan.lemma} is that, 
 even though the initial marking function 
 $f:V \rightarrow \mathbb{R}^N\cup \{\infty\}$ might
 only be partially defined, the new marking function
  $f^{F,h}: P \rightarrow \mathbb{R}^N\cup \{\infty\}$ 
  will be \textit{globally} defined on all of $P$.  
  As a final comment, we point out that any non-empty fiber 
  $W_x$ over a point $x \in U_h$ will remain 
  non-empty throughout the scanning argument 
  described above. This is the reason why we 
  require $U_h \subseteq V$;
  with this assumption, we guarantee that any
  non-empty fiber $W^{F,h}_x$ will have a base-point
  of the form $f^{F,h}(x)$ by the end of the scanning 
  process. On the other hand,  
  since all the fibers of $W^{F,h}$ over $P - U_h$ 
  will be empty, the value $f^{F,h}(x)$ 
  for any point $x \notin U_h$ will simply be $\infty$.

 \begin{proof}[Proof of Lemma \ref{scan.lemma}]
Let $\widehat{F}: U_h\times\mathbb{R}^N  \rightarrow U_h\times\mathbb{R}^N$  
be the $\epsilon$-scanning map induced by the maps $F$ and $h$ 
(in the sense of Definition \ref{control.scan}),  
and let $W_{U_h}$ be the restriction of $W$ over $U_h$. 
We construct the element $W^{F,h}$ as follows: 
First, take the image of $W_{U_h}$ under the $\epsilon$-scanning map 
$\widehat{F}: U_h\times\mathbb{R}^N  \rightarrow U_h\times\mathbb{R}^N$, 
and then include this image in $P\times\mathbb{R}^N$ via the obvious inclusion 
$U_h\times \mathbb{R}^N \hookrightarrow P\times \mathbb{R}^N$. 
To show that $W^{F,h}$ is indeed an element of $\Psi_d(\mathbb{R}^N)(P)$,
we must verify two things:

\begin{enumerate}
\item $W^{F,h}$ is a closed PL subspace of $P\times \mathbb{R}^N$.  
\item The restriction of the standard projection 
$P \times \mathbb{R}^N \rightarrow P$ 
on $W^{F,h}$ is a PL submersion of codimension $d$.   
\end{enumerate}

Since $\widehat{F}$ is a PL homeomorphism, it follows that the image $\widehat{F}(W_{U_h})$ is a PL subspace of $U_h\times \mathbb{R}^N$, which in turn is a PL subspace of $P\times \mathbb{R}^N$. Therefore, $W^{F,h}$ is a PL subspace of $P\times \mathbb{R}^N$.
To establish that $W^{F,h}$ is a closed subspace of $P\times \mathbb{R}^N$, we will prove
the following claim:

\begin{itemize}

\item[ ] \textbf{Claim A.} \textit{For any point $(x_0, y_0) \in P\times \mathbb{R}^N - W^{F,h}$, there exists 
an open neighborhood $V_{x_0}$ of $x_0$ in $P$ and a value $R> 0$ such that 
$V_{x_0} \times B(y_0, R)$ is disjoint from  $W^{F,h}$.}

\end{itemize}

In the previous statement, $B(y_0, R)$ denotes the open cube $\prod_{j=1}^N(y_0^j - R, y_0^j + R)$, where
$y_0^j$ is the $j$-th coordinate of the point $y_0$.  We will break down the proof 
of this claim into two cases:  

\textit{\underline{Case A.1:} The point $x_0$ is in $U_h$.}  
Recall that $W^{F,h}$ is equal to the image 
 $\widehat{F}(W_{U_h})$, where $W_{U_h}$ is the restriction of $W$
 over $U_h$ and $\widehat{F}: U_h\times\mathbb{R}^N  \rightarrow U_h\times\mathbb{R}^N$  
is the $\epsilon$-scanning map induced by $F$ and $h$. 
Since $W_{U_h}$ is closed in $U_h \times \mathbb{R}^N$ and
$\widehat{F}$ is a PL homeomorphism,  it follows
that $W^{F,h}$ is closed in $U_h \times \mathbb{R}^N$. 
Consequently, we can find an open neighborhood $V_{x_0}$
of $x_0$ in $U_h$ and a value $R>0$ such that
$\big(\hspace{0.02cm}V_{x_0}\times B(y_0,R)\hspace{0.02cm} \big)  \cap W^{F,h} = \varnothing$. 
 
 \textit{\underline{Case A.2:} The point $x_0$ is in $P-U_h$.}  
Recall that $U_h$ is equal to the pre-image $h^{-1}([0,1))$. Then, 
since $x_0 \in P -U_h$, we must have that $h(x_0) = 1$.
Thus, by the assumptions given in the statement of this lemma, the fiber 
$W_{x_0}$ is disjoint from the cube $[-\epsilon, \epsilon]^N$.   
In fact, using the compactness
of $P$, we can find a value $0 < \delta < 1$ such that
$W_x \cap [-\epsilon, \epsilon]^N = \varnothing$ for all
$x \in h^{-1}((1 - \delta, 1])$. 
We can prove the existence 
of such a $\delta > 0$ as follows: Let $V_{\epsilon}$ be the 
subspace of $P$ of all points $x$ such that $W_x$ is disjoint
from $[-\epsilon, \epsilon]^N$. 
The fact that $W$ is closed and 
$[-\epsilon, \epsilon]^N$ is compact ensures that $V_{\epsilon}$
is open in $P$. Next, note that $V_{\epsilon}$ and the collection
$\{ h^{-1}([0,1 - s)) \}_{0 <  s < 1}$ form an open cover for $P$. Then, 
the compactness of $P$ implies that there is a $\delta > 0$ 
such that $V_{\epsilon}$ and $h^{-1}([0,1 - \delta))$ form an
open cover for $P$. 
Consequently, since  $h^{-1}((1 - \delta, 1])$
 and $h^{-1}([0,1 - \delta))$ are disjoint, we must have  
 that $h^{-1}((1 - \delta, 1]) \subseteq V_{\epsilon}$. 
 Therefore, for any  $x \in h^{-1}((1 - \delta,1])$, 
 the fiber $W_x$ must be disjoint from $[-\epsilon, \epsilon]^N$. 

Now, let us fix an arbitrary value $R>0$. We claim that, for this 
$R>0$, we can find an open neighborhood $V_{x_0}$ of $x_0$ 
such that $V_{x_0} \times B(y_0, R)$ is disjoint from $W^{F,h}$. 
To prove this, pick first a large enough value $r > 0$ so that the cube 
$[-r,r]^N$ contains $B(y_0, R)$.  
By condition (iv) from the definition of elementary 
$\epsilon$-scanning map, it is possible to find a value
$\delta'$ in the open interval $(0, \delta)$ such that 
$[-r, r]^N \subseteq F_{h(x)}\big( (-\epsilon, \epsilon)^N \big)$
for all $x$ in $h^{-1}((1 - \delta', 1])$.  
Then, since $W^{F,h}_x = F_{h(x)}(W_x)$ for all $x \in U_h$
and $W_x \cap [-\epsilon, \epsilon]^N = \varnothing$
for all $x \in h^{-1}((1 - \delta', 1])$, it follows that
\begin{equation} \label{partition1}
W^{F,h}_x \cap [-r, r]^N = \varnothing \text{ for all } x \in U_h \cap h^{-1}((1 - \delta', 1]). 
\end{equation}
On the other hand, since $W^{F,h}$ is empty over $P- U_h$, we also have that  
\begin{equation} \label{partition2}
W^{F,h}_x \cap [-r, r]^N = \varnothing \text{ for all } x \in (P-U_h) \cap h^{-1}((1 - \delta', 1]). 
\end{equation}
Thus, if we set $V_{x_0} := h^{-1}((1 - \delta', 1])$, the statements given in 
(\ref{partition1}) and (\ref{partition2}) imply that 
$W^{F,h}_x \cap [-r, r]^N = \varnothing$ for all 
$x \in V_{x_0}$, which is equivalent to saying that 
$V_{x_0}\times [-r, r]^N$ does not intersect  $W^{F,h}$. Since 
$[-r,r]^N$ contains $B(y_0, R)$, the product 
$V_{x_0}\times B(y_0,R)$ is also disjoint from
$W^{F,h}$, and we have thus concluded the proof of Claim A.  

To prove that the restriction of the standard projection 
$P\times \mathbb{R}^N\rightarrow P$ on $W^{F,h}$
 is a PL submersion of codimension $d$, one just 
 needs to observe that this map can be factored as 
 \[
 W^{F,h} \rightarrow W_{U_h} \hookrightarrow W \rightarrow P,
 \]
 where the first map is the restriction of $\widehat{F}^{-1}$ on $W^{F,h}$, 
 which is a PL homeomorphism, the second map is 
 the inclusion of $W_{U_h}$ into $W$, and the third 
 map is the restriction of the standard projection 
 $P\times\mathbb{R}^N\rightarrow P$ on $W$.
 Since the first two maps in this composition are PL submersions of codimension 0 
 and the third one is a PL submersion of codimension $d$,
we can conclude that the restriction of the 
standard projection $P\times \mathbb{R}^N \rightarrow P$ on $W^{F,h}$ is a PL submersion of codimension $d$.
 Thus, we have proven that $W^{F,h}$ is an element of $\Psi_d(\mathbb{R}^N)(P)$.  
Also, by the way we constructed $W^{F,h}$, 
it is evident that this element satisfies conditions (i) and (ii) given in the statement of this lemma.  
 
 Now suppose that $f: V \rightarrow \mathbb{R}^N\cup \{\infty\}$ 
 is a partially-defined marking function for $W$ such that 
 $V$ is an open set containing $U_h$,
 and consider the function 
 $f^{F,h}: P \rightarrow \mathbb{R}^N\cup \{\infty\}$
 defined in (\ref{new.marking.function}). 
Evidently, by the definition given in 
(\ref{new.marking.function}),
the pre-image $(f^{F,h})^{-1}(\infty)$ 
is equal to $(P - U_{h})\cup f^{-1}(\infty)$. 
Also, note that a fiber $W^{F,h}_x$ of $W^{F,h}$ is empty 
if and only if $x$ is a point in 
$(P - U_{h})\cup f^{-1}(\infty)$. 
In particular, the function $f^{F,h}$ 
satisfies condition (i) from the definition of 
marking function, given in Definition \ref{marking}. 
On the other hand, for any $x$ in the open set 
$U_{h}\cap f^{-1}(\mathbb{R}^N)$ 
(which is the pre-image
of $\mathbb{R}^N$ under $f^{F,h}$),  
condition (ii) from the definition of marking function 
guarantees that $f(x) \in W_x$, which evidently 
implies that 
$F_{h(x)}(f(x)) \in F_{h(x)}(W_x)$.  Then, since 
$f^{F,h}(x) = F_{h(x)}(f(x))$ and $W^{F,h}_x = F_{h(x)}(W_x)$, we can 
conclude  that $f^{F,h}(x)$ is a point in $W^{F,h}_x$ 
whenever $x \in U_{h}\cap f^{-1}(\mathbb{R}^N) = (f^{F,h})^{-1}(\mathbb{R}^N)$. 
Therefore, $f^{F,h}$ also satisfies condition (ii) from  Definition \ref{marking}. 
Furthermore, by condition (ii) from the definition
of elementary $\epsilon$-scanning map (Definition \ref{elem.scan}), the PL homeomorphism
$\widehat{F}$ maps the product $U_h \times \{\mathbf{0}\}$
to itself, where $\mathbf{0}$ once again denotes the origin 
of $\mathbb{R}^N$.
From this last observation,
it follows that $f^{F,h}$ 
also satisfies condition (iii) 
from the definition of marking function.   

It remains to show that $f^{F,h}$ is piecewise linear on 
$U_h\cap f^{-1}(\mathbb{R}^N)$ and continuous on all of $P$. 
To show that $f^{F,h}$ is piecewise linear on 
$U_h\cap f^{-1}(\mathbb{R}^N)$, 
it is enough to observe that the formula 
given in (\ref{new.marking.function}) asserts  
that $f^{F,h}$ is equal to a composition of 
PL maps when restricted to $U_h\cap f^{-1}(\mathbb{R}^N)$.  
In particular, since any PL map is continuous, the function $f^{F,h}$ is 
continuous on $U_h\cap f^{-1}(\mathbb{R}^N)$. 
Thus, it only remains to show that $f^{F,h}$ is continuous on 
$(P - U_h)\cup f^{-1}(\infty)$. 
Fix then any point 
$x_0$  in $(P - U_h)\cup f^{-1}(\infty)$. Note that the 
fiber $W^{F,h}_{x_0}$ is empty. In particular, $f^{F,h}(x_0) = \infty$. 
Also, consider the
following family of open sets in $\mathbb{R}^N \cup \{ \infty \}$: 
\begin{equation} \label{fund.system}
\Big\{ \big(\mathbb{R}^N \cup \{ \infty \}\big) - [-r, r]^N \Big\}_{r>0}.
\end{equation}
Since $W^{F,h}_{x_0} = \varnothing$, we can repeat the argument we did to 
establish (\ref{partition1}) and (\ref{partition2}) in the proof of Claim A to
show that, for any $r>0$, there exists an open neighborhood $V_r$ of $x_0$ in $P$
such that $W^{F,h}_x\cap [-r,r]^N = \varnothing$  for any $x \in V_r$.
This then  implies that $f^{F,h}(V_r) \subseteq \big(\mathbb{R}^N \cup \{ \infty \}\big) - [-r, r]^N$.
Since the family given in (\ref{fund.system}) is 
a fundamental system of neighborhoods of the point $f^{F,h}(x_0) = \infty$ in $\mathbb{R}^N \cup \{ \infty \}$, 
we can conclude that $f^{F,h}$ is continuous at $x_0$.  
\end{proof}

We shall adopt the following terminology throughout the
rest of this section. 

\theoremstyle{definition} \newtheorem{scanning.result}[marking]{Definition}

\begin{scanning.result} \label{scanning.result}
Let $P$, $W$, $\epsilon$, $F$, $h$, and
$f$ be as in
the statement of Lemma \ref{scan.lemma}.
In particular, we assume that $f$
is defined on an open set $V \subseteq P$
containing $U_h := h^{-1}([0,1))$. 
If $W^{F,h}$ is the element of $\Psi_d(\mathbb{R}^N)(P)$  
produced via the procedure described in Lemma \ref{scan.lemma},
we shall say that $W^{F,h}$ is 
\textit{the result of performing $(\epsilon,F,h)$-scanning on $W$}. 
Similarly, if $f^{F,h}: P  \rightarrow \mathbb{R}^N\cup\{\infty\}$ is the marking
function for $W^{F,h}$ obtained from $f$ via Lemma \ref{scan.lemma},
then we say that the pair $(W^{F,h}, f^{F,h})$ is \textit{the result of performing 
$(\epsilon,F,h)$-scanning on $(W,f)$.}  
\end{scanning.result} 

\theoremstyle{definition} \newtheorem{remark.scanning.result}[marking]{Remark}

\begin{remark.scanning.result} \label{remark.scanning.result}
It is worth pointing out again that we can perform the type of scanning introduced in Lemma \ref{scan.lemma} 
on a pair $(W,f)$ even if $f$ is just a partially-defined marking function. 
Once we perform $(\epsilon,F,h)$-scanning on both $W$ and $f$, 
the resulting marking function $f^{F,h}$ will be defined on all of $P$.   
\end{remark.scanning.result}

Consider a compact PL space $P$. In the next proposition, 
we will show that the process of modifying
an element $(W,f) \in \widetilde{\Psi}_d(\mathbb{R}^N)(P)$ 
via an $(\epsilon, F, h)$-scanning can be realized as a concordance
(see Definition \ref{concordance}). 
Note that, since $(W,f) \in \widetilde{\Psi}_d(\mathbb{R}^N)(P)$,
we are implicitly assuming that the marking function $f$ is 
defined on all of $P$.  

\theoremstyle{plain} \newtheorem{scanning.concordance}[marking]{Proposition}

\begin{scanning.concordance} \label{scanning.concordance}
Let $P$ be a compact PL space and
$(W,f)$ an element of the set $\widetilde{\Psi}_d(\mathbb{R}^N)(P)$. 
If $(W^{F,h}, f^{F,h})$ is the result of performing 
$(\epsilon,F,h)$-scanning on $(W,f)$, then there 
exists an element $(\widehat{W},\widehat{f})$ of 
$\widetilde{\Psi}_d(\mathbb{R}^N)([0,1]\times P)$ 
with the following properties: 

\begin{itemize}
\item[(i)] $(\widehat{W},\widehat{f})$ is a concordance from $(W,f)$ to $(W^{F,h}, f^{F,h})$. 

\item[(ii)] If the fiber $W_x$ of $W$ over a point $x\in P$ is empty, then $\widehat{W}$ is empty over the product $[0,1]\times \{x\}$.  In other words, any empty fiber of $W$ will remain empty throughout the scanning process. 
\end{itemize}
\end{scanning.concordance}

\begin{proof}
Let  $(W^{F,h}, f^{F,h})$ be the result of performing 
$(\epsilon,F,h)$-scanning on $(W,f)$. Recall that $F$ and $h$ represent 
the following: 

\begin{itemize}
\item[($\ast$)] $F: [0,1)\times \mathbb{R}^N \rightarrow \mathbb{R}^N$ is an elementary $\epsilon$-scanning map.  

\item[($\ast\ast$)] $h:P \rightarrow [0,1]$ is a PL function such that, for every $x$ in the pre-image $h^{-1}(1)$, the fiber $W_x$ does not intersect the cube $[-\epsilon, \epsilon]^N$.
Since $f$ is defined on all of $P$, we trivially have that the set $U_h := h^{-1}([0,1))$
is contained in the domain of $f$.   
\end{itemize}

Besides these maps $F$ and $h$, we shall also use the following objects in this proof:  

\begin{itemize}
\item[$\cdot$] The constant concordance  $[0,1]\times W$ from $W$ to itself. That is, $[0,1]\times W$
is the element of the set $\Psi_d(\mathbb{R}^N)([0,1]\times P)$ whose fiber over a point 
$(t,x) \in [0,1]\times P$ is equal to $W_x$. 

\item[$\cdot$] The marking function $\bar{f}: [0,1]\times P \rightarrow \mathbb{R}^N\cup\{\infty\}$ 
 for $[0,1]\times W$ defined by 
 \[
 \bar{f}(t,x) = f(x).
\]

\item[$\cdot$]  The constant map  $c_0: P \rightarrow [0,1]$ 
which maps every point of $P$ to $0$.  
\end{itemize}

Using standard methods from PL topology, it is possible to construct a 
PL homotopy $H: [0,1]\times P \rightarrow [0,1]$ from $c_0$ to $h$
(i.e., $H_0 = c_0$ and $H_1 = h$) so that 
$H(t,x) < 1$ whenever $t < 1$. 
This last property of $H$ ensures that
the concordance $[0,1]\times W$ satisfies the following condition:  
\begin{itemize}
\vspace{0.15cm}
\item[($\ast\ast\ast$)] If $H(t,x) =1$, then the fiber of $[0,1]\times W$ over the point $(t,x)$ is disjoint
from the cube  $[-\epsilon, \epsilon]^N$.  
\vspace{0.15cm}
\end{itemize}
Note that $H(t,x) = 1$ necessarily implies that $t =1$. By virtue of the condition given in ($\ast\ast\ast$), 
we can perform $(\epsilon,F,H)$-scanning (where $F$
is the same map appearing in ($\ast$)) on the element 
$([0,1]\times W,\bar{f})$ to obtain a new element in 
$\widetilde{\Psi}_d(\mathbb{R}^N)([0,1]\times P)$.
This new element, which we will denote by 
$(\widehat{W}, \widehat{f})$, is our desired concordance from $(W,f)$ to 
$(W^{F,h}, f^{F,h})$.  
Indeed, since $H_0 = c_0$ and $H_1= h$, 
we have that  $(\widehat{W}, \widehat{f})$ is a 
concordance between the elements obtained 
by performing respectively $(\epsilon,F,c_0)$-scanning  
and $(\epsilon,F,h)$-scanning on $(W,f)$. In other words, 
$(\widehat{W}, \widehat{f})$ is a concordance from 
$(W,f)$ to $(W^{F,h}, f^{F,h})$.  
Finally, condition (ii) given in the statement of this proposition  
just follows from the fact that any empty fiber of 
$[0,1]\times W$ will remain empty after we
perform $(\epsilon,F,H)$-scanning.
\end{proof}

The following definition will be essential for
most of the scanning arguments that we will do in this section. 

\theoremstyle{definition} \newtheorem{zero.set}[marking]{Definition}

\begin{zero.set} \label{zero.set}  Let $P$ be a PL space. For an element $W$ of $\Psi_d(\mathbb{R}^N)(P)$, \textit{the zero-set of $W$}, denoted by $Z(W)$, is the subspace of $P$ consisting of all points $x$ such that the fiber $W_x$ contains the origin $\mathbf{0}$ of $\mathbb{R}^N$.      
\end{zero.set} 

It is straightforward to prove that $Z(W)$ is a closed PL subspace of $P$.  We will use this definition in the following example, which illustrates the main situation in which we will apply the scanning construction introduced in Lemma \ref{scan.lemma}. 
In this example, we will change some of our notational conventions. Specifically, 
if $P$ is a PL space and $X$ is an arbitrary subspace of $P$, then we shall denote the closure
of $X$ by $\mathrm{cl}(X)$ instead of $\overline{X}$. 

\theoremstyle{definition} \newtheorem{nested.regular}[marking]{Example}

\begin{nested.regular} \label{nested.regular}

Fix the following data: 

\begin{itemize}
\item[$\cdot$] A compact PL manifold $M$ with non-empty boundary 
and an element $W$ of $\Psi_d(\mathbb{R}^N)(M)$ such that 
$Z(W) \neq \varnothing$ and $Z(W) \subseteq M - \partial M$. 
Note that the compactness of $M$ implies that the zero-set $Z(W)$ is also compact. 

\item[$\cdot$] A partially-defined marking function 
$f: V \rightarrow \mathbb{R}^N\cup\{\infty\}$ for $W$ such that $Z(W) \subseteq V$
and $V \subseteq M - \partial M$.

\item[$\cdot$] Two regular neighborhoods  $R$ and $R'$ of $Z(W)$ in $V$ such that 
$R \subseteq \mathrm{Int}\hspace{0.05cm}R'$.  
See Appendix \ref{sec.regular} for the definition of regular neighborhood. 

\item[$\cdot$] A PL homeomorphism $k: \mathrm{cl}(R' - R) \rightarrow \partial R\times [0,1] $ which maps $\partial R$ identically to $\partial R\times\{0\}$ and maps $\partial R'$ onto $\partial R\times\{1\}$.
The existence of such a PL homeomorphism $k$ is guaranteed by Proposition \ref{nested.regular.thm} in 
Appendix \ref{sec.regular}.  
\end{itemize}  

Since $R'$ is a regular neighborhood of $Z(W)$, we have in particular that 
$Z(W) \subseteq \mathrm{Int}\hspace{0.05cm}R'$, which is 
equivalent to saying that $Z(W)$ is disjoint from $\mathrm{cl}(M-R')$. 
Then, since $W$ is a closed subspace of $M\times \mathbb{R}^N$, there must exist
a value $\epsilon>0$ such that 
$W_x \cap [-\epsilon,\epsilon]^N = \varnothing$
for all fibers $W_x$ over $\mathrm{cl}(M-R')$. 
Fix then such a value $\epsilon$ and choose an elementary 
$\epsilon$-scanning map $F:[0,1)\times \mathbb{R}^N \rightarrow \mathbb{R}^N$ 
corresponding to this $\epsilon$.  
Additionally, let $h: M \rightarrow [0,1] $ be the piecewise linear 
function defined by 
\begin{equation} \label{bump.function}
  h(x) = \left\{
    \begin{array}{rl}
      0 & \text{if } x\in R,\\
      p_2(k(x)) & \text{if } \mathrm{cl}(R' - R),\\
      1 & \text{if } \mathrm{cl}(M - R'),
    \end{array} \right.
\end{equation}
where $p_2$ is the standard projection $\partial R\times [0,1]\rightarrow [0,1]$ 
onto the second factor.   
A function $h(x)$ as defined in (\ref{bump.function}) 
will sometimes be called a \textit{bump function relative to the pair $R$, $R'$}.  
Since the set $U_h := h^{-1}([0,1))$ is contained in $V$ 
and all the fibers of $W$ over $h^{-1}(1)$ do not intersect the cube 
$[-\epsilon, \epsilon]^N$, we can perform $(\epsilon,F,h)$-scanning 
on the pair $(W,f)$ to produce an element $(W^{F,h},f^{F,h})$ of 
$\widetilde{\Psi}_d(\mathbb{R}^N)(M)$.   
It is straightforward to check that the pair $(W^{F,h},f^{F,h})$ has the following properties: 

\begin{itemize}
\item[$\cdot$] $(W^{F,h},f^{F,h}) = (W,f)$ over the regular neighborhood $R$.   

\item[$\cdot$] Over $\mathrm{cl}(M - R')$, the pair $(W^{F,h},f^{F,h})$ 
agrees with the \textit{trivial element} $(\varnothing,\infty)$ of 
$\widetilde{\Psi}_d(\mathbb{R}^N)(M)$, i.e., $(\varnothing,\infty)$ is
the unique element of $\widetilde{\Psi}_d(\mathbb{R}^N)(M)$
whose underlying PL space is empty and whose marking
 function is the constant map which sends all points to $\infty$. 
\end{itemize}

We will use the construction given in this example when we compare the simplicial sets
$\Psi_d(\mathbb{R}^N)_{\bullet}$ and $\widetilde{\Psi}_d(\mathbb{R}^N)_{\bullet}$. 

\end{nested.regular}

In the case when $Z(W) = \varnothing$, Proposition  \ref{scanning.concordance} 
yields the following important result. 

\theoremstyle{plain} \newtheorem{scan.trivial}[marking]{Corollary}

\begin{scan.trivial} \label{scan.trivial}
If $P$ is a compact PL space and $(W,f)$ is an element of $\widetilde{\Psi}_d(\mathbb{R}^N)(P)$ with $Z(W) = \varnothing$, then there is a concordance $(\widehat{W}, \widehat{f})$ from $(W,f)$ to the trivial element 
$(\varnothing, \infty)$. Moreover, if the fiber of $W$ over a point $x\in P$ is empty, then $\widehat{W}$ 
can be chosen to be empty over the product $[0,1]\times\{x\}$.   
\end{scan.trivial}

\begin{proof}
Since $Z(W) = \varnothing$, the compactness of $P$ (and the fact that $W$ is closed in $P\times \mathbb{R}^N$) 
ensures that there is an $\epsilon >0$ such that every fiber $W_x$ of $W$ is disjoint from the cube $[-\epsilon,\epsilon]^N$. Fix then an elementary $\epsilon$-scanning map $F:[0,1) \times \mathbb{R}^N \rightarrow \mathbb{R}^N$ corresponding to this $\epsilon$. 
Also, let $c_1: P \rightarrow [0,1]$ be the constant function which maps every point of $P$ to $1$. Since the intersection $W_x\cap[-\epsilon,\epsilon]^N$ is empty for all $x$ in $P$, we can perform 
$(\epsilon, F, c_1)$-scanning on the pair $(W,f)$. 
The result of this scanning will be the trivial element $(\varnothing, \infty)$ of $\widetilde{\Psi}_d(\mathbb{R}^N)(P)$. By Proposition \ref{scanning.concordance}, there is a concordance $(\widehat{W}, \widehat{f})$ from $(W,f)$ to $(\varnothing, \infty)$, which we can assume to be equal to the trivial element of $\widetilde{\Psi}_d(\mathbb{R}^N)([0,1]\times P)$ over the product $[0,1]\times f^{-1}(\infty)$.
 \end{proof}

We can use Corollary \ref{scan.trivial} to prove 
the following property of the simplicial set $\widetilde{\Psi}_d(\mathbb{R}^N)_{\bullet}$. 

\theoremstyle{plain} \newtheorem{marking.connected}[marking]{Corollary}

\begin{marking.connected} \label{marking.connected}
The simplicial set $\widetilde{\Psi}_d(\mathbb{R}^N)_{\bullet}$ is path-connected.    
\end{marking.connected}
   
\begin{proof} 
Recall that a $0$-simplex of $\widetilde{\Psi}_d(\mathbb{R}^N)_{\bullet}$ 
is simply a pair $(W,x)$, where $W$ is a $d$-dimensional piecewise linear 
submanifold of $\mathbb{R}^N$ (which is closed as a subspace) and $x$ 
is a point in $W$. If $W = \varnothing$, then $x = \infty$, and if $W$ 
intersects the origin $\mathbf{0}$ of $\mathbb{R}^N$, then we must have 
$x = \mathbf{0}$.  

Since $\widetilde{\Psi}_d(\mathbb{R}^N)_{\bullet}$ is Kan, it is enough 
to show that any $0$-simplex $(W,x)$ is concordant to the trivial element 
$(\varnothing,\infty)$ of $\widetilde{\Psi}_d(\mathbb{R}^N)_{0}$. 
Fix then an arbitrary element $(W,x)$ of $\widetilde{\Psi}_d(\mathbb{R}^N)_{0}$.
We consider two cases:

\noindent \textit{\underline{Case 1:}} If the submanifold 
$W$ does not intersect the origin $\mathbf{0}$ of $\mathbb{R}^N$, 
then Corollary \ref{scan.trivial} ensures that there is a concordance 
$(\widehat{W}, \widehat{f})$ between $(W,x)$ and  $(\varnothing,\infty)$.

\noindent \textit{\underline{Case 2:}} On the other hand, suppose that $W$ 
contains the origin $\mathbf{0}$. 
In particular, we must have that $x = \mathbf{0}$. 
To prove that $(W,x)$ is concordant to $(\varnothing,\infty)$
in this case, we will first show 
that there exists a point $P$ in $\mathbb{R}^N$ such
that the line segment $\overline{\mathbf{0}\hspace{0.05cm}P}$ only intersects 
$W$ at $\mathbf{0}$.  
Indeed, since $W$ is a closed PL subspace of 
$\mathbb{R}^N$, we can find a simplicial complex $K$ which triangulates 
$\mathbb{R}^N$ and contains a subcomplex $L$ triangulating $W$ 
(the existence of such a simplicial complex $K$ follows from 
Theorem 3.6 of \cite{plhud}). Moreover, 
by subdividing $K$ if necessary, we can assume that $\mathbf{0}$ is a 
vertex of $K$. Now, consider the star $\mathrm{st}(\mathbf{0}, K)$,
i.e., this is the subcomplex of $K$ obtained by taking the union of all simplices that 
contain $\mathbf{0}$ as a vertex and all faces of such simplices. 
Evidently, the star $\mathrm{st}(\mathbf{0}, L)$
is a subcomplex of $\mathrm{st}(\mathbf{0}, K)$. By subdividing $K$ further, 
we can assume that $\mathrm{st}(\mathbf{0}, L)$ is \textit{full} in
$\mathrm{st}(\mathbf{0}, K)$, i.e., we can ensure that 
$\mathrm{st}(\mathbf{0}, L)$ has the following property: 
If all the vertices of a simplex
$\sigma \in \mathrm{st}(\mathbf{0}, K)$ are in $\mathrm{st}(\mathbf{0}, L)$, 
then $\sigma$ itself must be a simplex of $\mathrm{st}(\mathbf{0}, L)$
(see Section \S 3.2 of  \cite{RS}). Now, consider an $N$-simplex 
$\sigma$ of $\mathrm{st}(\mathbf{0}, K)$. Evidently, $\sigma$ cannot
belong to $\mathrm{st}(\mathbf{0}, L)$ and, by the fullness 
of $\mathrm{st}(\mathbf{0}, L)$, there must be a vertex $P$
of $\mathrm{st}(\mathbf{0}, K)$ which does not belong to 
$\mathrm{st}(\mathbf{0}, L)$. Then, the line segment 
$\overline{\mathbf{0}\hspace{0.05cm}P}$ must be a 1-simplex 
of $K$ which only intersects $|L| = W$ at $\mathbf{0}$. 

Now, let $\vec{v}$ be the vector in $\mathbb{R}^N$ with tail at $P$ and head 
at $\mathbf{0}$, and consider the PL ambient isotopy 
$H: [0,1]\times \mathbb{R}^N \rightarrow \mathbb{R}^N$
defined by $H_t(x) = x - t\vec{v}$.  
Since $H$ is a PL ambient 
isotopy, the collection of PL manifolds $\{H_t(W)\}_{t\in [0,1]}$ 
defines an element $\widehat{W} \in \Psi_d(\mathbb{R}^N)([0,1])$, i.e.,
$\widehat{W}$ is the element of $\Psi_d(\mathbb{R}^N)([0,1])$
with the property that $\widehat{W}_t = H_t(W)$ for each $t \in [0,1]$. 
In particular, 
$\widehat{W}$ is a concordance from $W$ to the element 
$W' := H_1(W) \in \Psi_d(\mathbb{R}^N)_0$.  
Moreover, since the line $\overline{\mathbf{0}\hspace{0.05cm}P}$ 
only intersects $W$ at $\mathbf{0}$, any PL manifold $H_t(W)$ with 
$t \in (0,1]$ will not contain $\mathbf{0}$.  Therefore, if 
$\widehat{f}: [0,1] \rightarrow \mathbb{R}^N$ is the PL function defined
by $\widehat{f}(t) = H_t(\mathbf{0})$ and $x' := H_1(\mathbf{0})$, 
then the pair $(\widehat{W}, \widehat{f})$ is an element of
$\widetilde{\Psi}_d(\mathbb{R}^N)([0,1])$ which 
is a concordance from 
 $(W, \mathbf{0})$ to $(W', x')$.
 Since $x' \neq \mathbf{0}$, Corollary \ref{scan.trivial}
 ensures that 
 there is a concordance $(\widehat{W}', \widehat{f}')$ from 
$(W',x')$ to $(\varnothing, \infty)$. By concatenating 
$(\widehat{W}, \widehat{f})$ and $(\widehat{W}', \widehat{f'})$, 
we produce a concordance from $(W,\mathbf{0})$ to  $(\varnothing, \infty)$.  
\end{proof}

\subsection{Comparing spaces of manifolds}  The simplicial sets in the sequence $\{ \widetilde{\Psi}_d(\mathbb{R}^N)_{\bullet}\}_N$ can also be assembled into a spectrum.
For each positive integer $N$, the base-point that we will consider for $\widetilde{\Psi}_d(\mathbb{R}^N)_{\bullet}$ will be the subsimplicial set generated by the 0-simplex $(\varnothing, \infty)$. Using the same scanning process that we used to define the spectrum $\Psi_d$, we can define structure maps
 
\begin{equation} \label{structure.map}
\widetilde{\sigma}_N:S^1\wedge \widetilde{\Psi}_d(\mathbb{R}^N)_{\bullet} \rightarrow \widetilde{\Psi}_d(\mathbb{R}^{N+1})_{\bullet}
\end{equation}
for all positive integers $N$. 
For any simplex $(W,f)$ of  $\widetilde{\Psi}_d(\mathbb{R}^N)_{\bullet}$, the map $\widetilde{\sigma}_N$ will push all the values of the function $f$ towards $\infty$ while simultaneously pushing all the fibers of $W$ towards the empty manifold.   These maps turn the sequence $\{ \widetilde{\Psi}_d(\mathbb{R}^N)_{\bullet} \}_N$ 
into a spectrum, which from now on we will denote by $\widetilde{\Psi}_d$. 

The reason for introducing $\widetilde{\Psi}_d$ is that this spectrum will serve as a bridge between the spectrum of PL manifolds $\Psi_d$ and the Madsen-Tillmann spectrum $\mathbf{MT}PL(d)$, which we will introduce in Section \S\ref{section.grass}. 
More precisely, notice that for each positive integer $N$ there is an obvious forgetful map 
\begin{equation}  \label{forgetful.map}
\mathcal{F}_N:  \widetilde{\Psi}_d(\mathbb{R}^N)_{\bullet} \rightarrow \Psi_d(\mathbb{R}^N)_{\bullet}
\end{equation}
which maps any tuple $(W,f)$ to $W$.  
All of these forgetful maps assemble into a map of spectra
\begin{equation} \label{forgetful.spectra.map}
\mathcal{F}: \widetilde{\Psi}_d \rightarrow \Psi_d,
\end{equation}
and the main goal of this section is to prove the following theorem. 

\theoremstyle{plain} \newtheorem{marking.equivalence}[marking]{Theorem}

\begin{marking.equivalence} \label{marking.equivalence}
The map of spectra $\mathcal{F}: \widetilde{\Psi}_d \rightarrow \Psi_d$ defined in (\ref{forgetful.spectra.map}) is a weak equivalence.
\end{marking.equivalence}

Later, in \S\ref{section.grass}, we will show that there is a 
weak equivalence $\mathbf{MT}PL(d) \rightarrow \widetilde{\Psi}_d$, 
which would complete the proof of our main theorem.

\theoremstyle{definition} \newtheorem{zero.set.remark}[marking]{Remark}

\begin{zero.set.remark} \label{zero.set.remark} 
For any $W \in \Psi_d(\mathbb{R}^N)(P)$, 
the restriction $W_{Z(W)}$ over the zero-set $Z(W)$
admits only one marking function. Namely, the function 
$f_{\mathbf{0}}: Z(W) \rightarrow \mathbb{R}^N$  which maps
all points of $Z(W)$ to the origin $\mathbf{0} \in \mathbb{R}^N$. 
\end{zero.set.remark}

Let us give a brief outline of the proof of Theorem \ref{marking.equivalence}. 
To prove that the map (\ref{forgetful.map}) induces a surjection between 
homotopy groups, we will show that, for any element $W$ 
of $\Psi_d(\mathbb{R}^N)(\Delta^p)$ which is empty over 
$\partial\Delta^p$, there exists an element
of $\Psi_d(\mathbb{R}^N)(\Delta^p)$ which admits a marking function
and is concordant to $W$ \textit{relative to} $\partial\Delta^p$, i.e., 
the concordance we will construct will be empty over $[0,1]\times \partial\Delta^p$.  
This will be done by first producing a partially-defined  
marking function for $W$ defined on a small open neighborhood of $Z(W)$ 
(which will be accomplished using Lemma \ref{marking.germ}) and then 
use the scanning techniques from the previous section to turn this 
partially-defined marking function into one which is globally defined. 
On the other hand, the key step to prove that $\mathcal{F}_N$ induces an injection 
between homotopy groups is to show that, at least when we restrict to 
a sufficiently small open neighborhood of $Z(W)$, any two marking 
functions for a $p$-simplex $W$ are homotopic. 
This will be proven in Lemma \ref{marking.unique}.

\theoremstyle{plain} \newtheorem{marking.germ}[marking]{Lemma}

\begin{marking.germ} \label{marking.germ} 
Let $M$ be a compact PL manifold with $\partial M \neq \varnothing$. If $W$ is an element of 
$\Psi_d(\mathbb{R}^N)(M)$ such that $Z(W) \neq \varnothing$ and $Z(W) \subseteq M - \partial M$,  
then there exists a partially-defined marking function $f$ for $W$ defined 
on an open neighborhood $V$ of $Z(W)$ in $M - \partial M$.  
Moreover, the marking function $f$ will only take values in $\mathbb{R}^N$.  
\end{marking.germ}

\begin{proof}
Before we begin with the proof of this lemma, 
there are a few preliminary constructions we need to discuss.
First, let $\pi:W \rightarrow M$ be the restriction of the standard projection $M\times \mathbb{R}^N \rightarrow M$ on $W$.
As observed at the end of Definition \ref{pl.sub} in the appendix, 
the image of any PL submersion is open. Therefore,
we can find a regular neighborhood of $Z(W)$ which is contained in the image $\pi(W)$. 
Fix then a regular neighborhood $R$ of $Z(W)$ such that $R \subset \pi(W)$.  
According to Proposition \ref{map.cylinder} in Appendix \ref{sec.regular},
there exists a piecewise linear map $h: \partial R \rightarrow Z(W)$ such that $R$ is homeomorphic to the mapping cylinder $\mathrm{Cyl}(h)$. Moreover, if $i$ denotes the obvious inclusion $Z(W)\hookrightarrow \mathrm{Cyl}(h)$, then  we can choose a homeomorphism $H: \mathrm{Cyl}(h) \rightarrow R$ 
so that the composition 
\[
\xymatrix{ Z(W) \hspace{0.1cm} \ar@{^{(}->}[r]^{\hspace{-0.1cm}i} & \mathrm{Cyl}(h) \ar[r]^{\hspace{0.4cm}H} & R 
}
\]
maps $Z(W)$ identically to itself. 

As usual, we denote the restriction of $W$ over $R$ by $W_R$. Also, for any $\epsilon$ in the interval $(0,1)$, we will denote by 
$R_{\epsilon}$ the image of $\partial R \times [0,\epsilon)$ under the composite
$\partial R \times [0,1] \twoheadrightarrow \mathrm{Cyl}(h) \stackrel{H}{\longrightarrow} R$, where the first map
is the quotient map from $\partial R \times [0,1]$ onto $\mathrm{Cyl}(h)$. 
Evidently, $R_{\epsilon}$ is an open neighborhood of $Z(W)$ in $M - \partial M$. 

With all these preliminaries sorted out, we can now start with the details of the proof.
Consider then the commutative diagram
\begin{equation} \label{diag.marking.gem}
\xymatrix{ \partial R \times \{0\} \vspace{0.1cm} \ar@{^{(}->}[d] \ar[r] & W_R \ar[d]^{\pi|_{W_R}} \\
\partial R \times [0,1] \ar[r] &  R,
}
\end{equation}  
where the top and bottom horizontal maps are respectively the composites 
\[
\partial R \times \{0\} \stackrel{\mathrm{pr}_1}{\longrightarrow} \partial R 
\stackrel{h}{\longrightarrow} Z(W) \stackrel{\tilde{c}_{\mathbf{0}}}{\longrightarrow} W_R
\]
\[
\partial R \times [0,1] \twoheadrightarrow \mathrm{Cyl}(h) \stackrel{H}{\longrightarrow} R.
\]
The right-most map in the first composite is the inclusion defined by $\tilde{c}_{\mathbf{0}}(x) = (x, \mathbf{0})$,
where $\mathbf{0}$ denotes once again the origin of $\mathbb{R}^N$. 
Since any submersion is a microfibration (see Appendix \ref{sec.pl.sub}), 
we can find an $\epsilon \in (0,1)$ so that the bottom map in 
(\ref{diag.marking.gem}) admits a microlift with respect to $\pi|_{W_R}$ defined on 
$\partial R \times [0,\epsilon)$, as illustrated in the following figure: 
\begin{equation} \label{diag.marking.gem2}
\xymatrix{ & \partial R \times \{0\} \vspace{0.1cm} \ar@{^{(}->}[d] \ar[r] & W_R \ar[d]^{\pi|_{W_R}} \\
\partial R \times [0,\epsilon)  \ar@{-->}[urr]  \ar@{^{(}->}[r] & \partial R \times [0,1] \ar[r] &  R.
}
\end{equation}
It is readily seen that the microlift in (\ref{diag.marking.gem2}) factors through $R_{\epsilon}$, and we thus obtain 
a commutative diagram of the form
\begin{equation}\label{diag.marking.gem3}
\xymatrix{ & W_R \ar[d]^{\pi|_{W_R}} \\
R_{\epsilon} \ar@{-->}[ur]  \ar@{^{(}->}[r]  & R,
} 
\end{equation}
where the bottom map is the obvious inclusion. From now on, we will denote the lift  
$R_{\epsilon} \rightarrow W_R$ appearing in (\ref{diag.marking.gem3}) by $\tilde{f}$.
Now, fix a regular neighborhood 
$R'$ of $Z(W)$ inside $R_{\epsilon}$ (which exists because $R_{\epsilon}$ is open), and
let $f: R' \rightarrow \mathbb{R}^N$ be the map obtained by taking the composite 
\begin{equation} \label{diag.marking.gem4}
\xymatrix{
R' \ar[r]^{\hspace{-0.15cm}\tilde{f}|_{R'}} & W_R \hspace{0.1cm} \ar@{^{(}->}[r] & 
R\times \mathbb{R}^N \ar[r]^{\hspace{0.3cm}\mathrm{pr}_2} &   
 \mathbb{R}^N, 
}
\end{equation}
where the second map is the obvious inclusion and the third map is 
the standard projection from $R\times \mathbb{R}^N$ onto $\mathbb{R}^N$.  
By construction, $f$ is an $\mathbb{R}^N$-valued function, and  it is straightforward to verify that $f$ 
satisfies conditions (ii) and (iii) given in the definition of marking function.  
Furthermore, by Remark \ref{prop.microfib.rem} in Appendix \ref{sec.pl.sub}, we can assume that the map $\tilde{f}|_{R'}$ is piecewise linear, which would then imply that $f$ is also piecewise linear. Therefore, $f: R' \rightarrow \mathbb{R}^N$ is an $\mathbb{R}^N$-valued marking function for $W_{R'}$.
By restricting $f$ to $\mathrm{Int}\hspace{0.05cm}R'$, we obtain a partially-defined marking function for $W$. 
\end{proof}

\theoremstyle{plain} \newtheorem{marking.unique}[marking]{Lemma}

\begin{marking.unique} \label{marking.unique} 
Fix a PL space $P$. Let $W$ be an element of  $\Psi_d(\mathbb{R}^N)(P)$ with $Z(W)$ non-empty and compact, and let $f, g:V \rightarrow \mathbb{R}^N \cup \{\infty\}$ be two partially-defined 
marking functions for $W$ defined on an open neighborhood $V$ of $Z(W)$. Then, after possibly shrinking $V$, there exists a homotopy $F: [0,1]\times V \rightarrow  \mathbb{R}^N \cup \{\infty\}$  satisfying the following: 

\begin{itemize}
\item[(i)] $F_0 = f$ and $F_1 = g$.
\item[(ii)] $F$ is a partially-defined marking function for the constant concordance $[0,1]\times W.$
\end{itemize}
\end{marking.unique}

\begin{proof}
First of all, after shrinking the open neighborhood $V$, we can assume that the point $\infty$ is not in the image of $f$ or $g$.  
Let $\pi: W \rightarrow P$ denote the restriction of the standard projection $P\times \mathbb{R}^N \rightarrow P$ on $W$, and let $\tilde{f}, \tilde{g}: V \hookrightarrow W$ denote the piecewise linear inclusions defined respectively by $\tilde{f}(x)= (x, f(x))$ and $\tilde{g}(x)= (x, g(x))$. As pointed out in Remark \ref{remark.marking}, the fact that $\pi$ is  a piecewise linear submersion of codimension $d$ implies that $\xi = (V, W_V, \tilde{f}, \pi|_{W_V})$ is a $d$-dimensional PL microbundle. 
Then, by the Representation Theorem of Kuiper and Lashof (see Theorem \ref{kuiper.lashof} in Appendix \ref{sec.pl.micro}), there exists an open neighborhood 
$\widetilde{V}$ of the section $\tilde{f}(V)$ in $W_V$ such that $\eta = (V, \widetilde{V}, \tilde{f}, \pi|_{\widetilde{V}})$ is a piecewise linear $\mathbb{R}^d$-bundle whose total space is $\widetilde{V}$ (see Appendix \ref{sec.pl.micro} for the definition of piecewise linear $\mathbb{R}^d$-bundle). 

By shrinking the open neighborhood $V$ if necessary, we can assume that the image of $\tilde{g}$ lies entirely in 
$\widetilde{V}$.Thus, with this assumption, both $\tilde{f}$ and $\tilde{g}$ are sections of the bundle $\eta$, and since any two sections of an $\mathbb{R}^d$-bundle are fiberwise homotopic, we can find a PL homotopy 
$\widetilde{F}: [0,1]\times V \rightarrow \widetilde{V}$ such that $\widetilde{F}_0 = \tilde{f}$, $\widetilde{F}_1 = \tilde{g}$, and  $\pi(\widetilde{F}(t,x))=x$ for any $(t,x)$ in $[0,1]\times V$.  
Finally, by composing $\widetilde{F}$ with the obvious inclusions
 $\widetilde{V} \hookrightarrow W_V \hookrightarrow V\times\mathbb{R}^N$ 
 and the obvious projection $V\times\mathbb{R}^N \rightarrow \mathbb{R}^N$, we obtain a PL homotopy
$F:[0,1]\times V \rightarrow \mathbb{R}^N$ from $f$ to $g$. 
Moreover, it is straightforward to verify that $F$ is indeed a marking function for the 
constant concordance $[0,1]\times W$ over the product $[0,1]\times V$.  
\end{proof}

We can now proceed to give the proof of Theorem \ref{marking.equivalence}.  

\begin{proof}[Proof of Theorem \ref{marking.equivalence}]
It suffices to prove that, for each positive integer $N$, the map  $\mathcal{F}_N:  \widetilde{\Psi}_d(\mathbb{R}^N)_{\bullet} \rightarrow \Psi_d(\mathbb{R}^N)_{\bullet}$ defined in (\ref{forgetful.map}) is a weak homotopy equivalence. In this proof, the base-points that we will consider for $\Psi_d(\mathbb{R}^N)_{\bullet}$ and $ \widetilde{\Psi}_d(\mathbb{R}^N)_{\bullet}$ will be the subsimplicial sets generated by the 0-simplices $\varnothing$ and $(\varnothing, \infty)$ respectively. 

Consider an element $W$ in $\Psi_d(\mathbb{R}^N)(\Delta^p)$ which is empty over the boundary of $\Delta^p$. In other words, $W$ 
represents an element of the homotopy group $\pi_p\big(\Psi_d(\mathbb{R}^N)_{\bullet}\big)$. 
To prove that $\mathcal{F}_N$ induces a surjection between homotopy groups, we need to show that there exists an element $\widehat{W}$ of $\Psi_d(\mathbb{R}^N)([0,1]\times\Delta^p)$ which is empty over $[0,1]\times\partial\Delta^p$ and is a concordance between $W$  
and an element of $\Psi_d(\mathbb{R}^N)(\Delta^p)$ which admits a marking function. 
We will obtain such a concordance $\widehat{W}$ as follows. 
First, by Lemma \ref{marking.germ}, there is a sufficiently small open neighborhood $V$ of $Z(W)$ on which we can define an $\mathbb{R}^N$-valued marking function $f:V \rightarrow \mathbb{R}^N$  for $W_V$. 
Now choose two regular neighborhoods $R$ and $R'$ of $Z(W)$ inside $V$ such that $R\subseteq \mathrm{Int}\hspace{0.1cm}R'$. As we did in Example \ref{nested.regular}, we shall fix the following data: a value $\epsilon > 0$ such that each fiber of $W$ over $\mathrm{cl}(\Delta^p - R')$ is disjoint from the cube $[-\epsilon, \epsilon]^N$, an elementary $\epsilon$-scanning map $F: [0,1) \times\mathbb{R}^N \rightarrow \mathbb{R}^N$ corresponding to this $\epsilon$, and a bump function $h: \Delta^p \rightarrow [0,1]$ relative to the pair $(R,R')$ (see the function (\ref{bump.function}) defined in Example \ref{nested.regular}). 
By performing $(\epsilon, F, h)$-scanning on $W$ and $f$, we obtain an element $(W^{F,h},f^{F,h})$ of $\widetilde{\Psi}_d(\mathbb{R}^N)(\Delta^p)$ which agrees with $(W_R, f|_R)$ over $R$ and with 
$(\varnothing, \infty)$ outside of $R'$. 
Then, by 
Proposition \ref{scanning.concordance}, there exists a concordance $\widehat{W}$ 
between $W$ and $W^{F,h}$ which is empty over the product $[0,1]\times \partial\Delta^p$. 

To prove injectivity, consider two elements $(W_0,f)$ and $(W_1,g)$ of 
$\widetilde{\Psi}_d(\mathbb{R}^N)(\Delta^p)$ which agree with the tuple 
$(\varnothing,\infty)$ over $\partial\Delta^p$, and assume that the underlying 
PL spaces $W_0$ and $W_1$ represent the same element in the group 
$\pi_p(\Psi_d(\mathbb{R}^N)_{\bullet})$. Using this last assumption, 
we have to show that $(W_0,f)$ and $(W_1,g)$ represent the same class in 
$\pi_p(\widetilde{\Psi}_d(\mathbb{R}^N)_{\bullet})$.   
We will do this by constructing a concordance from 
$(W_0,f)$ to $(W_1,g)$ which agrees with the trivial element $(\varnothing, \infty)$ over $[0,1]\times \partial\Delta^p$.
To construct this concordance, we start by fixing a concordance 
$\widetilde{W} \in \Psi_d(\mathbb{R}^N)([0,1]\times \Delta^p)$  from $W_0$ to $W_1$ 
which is empty over the product $[0,1]\times\partial\Delta^p$. 
Such a concordance $\widetilde{W}$ exists because the simplicial set 
$\Psi_d(\mathbb{R}^N)_{\bullet}$ is Kan and because we are assuming that 
$W_0$ and $W_1$ represent the same element in $\pi_p(\Psi_d(\mathbb{R}^N)_{\bullet})$. 
Next, by Lemma \ref{marking.germ}, there is an open neighborhood $V$ of $Z(\widetilde{W})$ 
on which we can define an $\mathbb{R}^N$-valued marking function 
$\widetilde{f}: V \rightarrow \mathbb{R}^N$ for the restriction $\widetilde{W}_V$. 
Before we proceed any further, it is 
a good idea to introduce some notation:

\begin{itemize}
\item[$\cdot$] For $j \in \{ 0,1\}$, let $i_j: \Delta^p \hookrightarrow [0,1]\times \Delta^p$
be the inclusion given by $i_j(x) = (j,x)$. From now on, we shall denote the pre-images
$i_0^{-1}(V)$ and $i_1^{-1}(V)$ by $V_0$ and $V_1$ respectively. 

\item[$\cdot$] Also, we shall denote the compositions 
$\widetilde{f}\circ i_0|_{V_0}$ and $\widetilde{f}\circ i_1|_{V_1}$ by
$\widetilde{f}_0$ and $\widetilde{f}_1$ respectively. 
\end{itemize}

Now, there is no guarantee that  
$\widetilde{f}_0$ and  
$\widetilde{f}_1$ agree with the restrictions $f|_{V_0}$ and $g|_{V_1}$ respectively.  
However, we can fix this issue as follows. 
First, let $\widetilde{W}'$ be the element of $\Psi_d(\mathbb{R}^N)([-1,2]\times \Delta^p)$
obtained by gluing copies of the trivial concordances
$[-1,0]\times W_0$ and $[1,2]\times W_1$ to the left and right-hand sides of 
$\widetilde{W}$ respectively, and let $V'$ be the open set in the base-space
$[-1,2]\times\Delta^p$ obtained by gluing the sets 
$[-1,0]\times V_0$, $V$, and $[1,2]\times V_1$.
Note that $V'$ is an open neighborhood of the zero-set 
$Z(\widetilde{W}')$. After possibly shrinking $V_0$ and $V_1$
(which can be achieved by just shrinking $V$), Lemma \ref{marking.unique} 
guarantees the existence of partially-defined marking functions 
$F: [-1,0] \times V_0 \rightarrow \mathbb{R}^N$
and $G: [1,2] \times V_1 \rightarrow \mathbb{R}^N$  
for $[-1,0]\times W_0$ and $[1,2]\times W_1$ respectively
such that 
$F$ is a homotopy from $f|_{V_0}$ to $\widetilde{f}_0$, and
$G$ is a homotopy from $\widetilde{f}_1$ to $g|_{V_1}$.  
Then,
we can glue the PL functions $F$, $\widetilde{f}$, and $G$ to obtain
a partially-defined marking function 
$\widetilde{f}': V' \rightarrow \mathbb{R}^N$ for $\widetilde{W}'$. 
By rescaling the first factor of the base-space 
$[-1,2] \times \Delta^p$, we can regard 
$\widetilde{W}'$ as an element of $\Psi_d(\mathbb{R}^N)([0,1]\times \Delta^p)$
and $V'$ as an open neighborhood of $Z(\widetilde{W}')$ in $[0,1]\times \Delta^p$.
After this rescaling, $\widetilde{f}': V' \rightarrow \mathbb{R}^N$ becomes
a partially-defined marking function with the property that 
$\widetilde{f}'_0 = f|_{V_0}$ and $\widetilde{f}'_1 = g|_{V_1}$, 
which is exactly the condition we wanted to ensure. 
By abuse of notation, we shall relabel $\widetilde{W}'$ and 
$\widetilde{f}'$ by $\widetilde{W}$ and $\widetilde{f}$ respectively. 

Our next step is to use scanning to turn the partially-defined 
marking function $\widetilde{f}$ into one which is globally defined on $[0,1]\times\Delta^p$.  
To do this, choose two regular neighborhoods $R$ and $R'$ of 
$Z(\widetilde{W})$ in $V$ such that $R\subseteq \mathrm{Int}\hspace{0.1cm}R'$, 
and fix again the following data: a value $\epsilon > 0$ such that all the fibers of $\widetilde{W}$ outside of $R'$ do not intersect the cube $[-\epsilon,\epsilon]^N$, an elementary $\epsilon$-scanning map $F$ corresponding to this $\epsilon$, and a bump function $h:[0,1]\times\Delta^p \rightarrow [0,1]$ relative to the pair $R$ and $R'$. All of this data satisfies the conditions required in Lemma \ref{scan.lemma}, which means that we can perform 
$(\epsilon,F,h)$-scanning on the pair $(\widetilde{W}, \widetilde{f})$ to produce 
a new element of $\widetilde{\Psi}_d(\mathbb{R}^N)([0,1]\times \Delta^p)$, which we will 
denote by $(\widetilde{W}^{F,h},\widetilde{f}^{F,h})$. 
Now, if $h_0: \Delta^p \rightarrow [0,1]$ and $h_1: \Delta^p \rightarrow [0,1]$ are the PL functions defined by $h_0(x) = h(0,x)$ and $h_1(x) = h(1,x)$ respectively, then 
it is evident that $(\widetilde{W}^{F,h},\widetilde{f}^{F,h})$ is a concordance between the element $(W^{F,h_0}_0, f^{F,h_0})$ obtained by performing $(\epsilon,F,h_0)$-scanning on $(W_0,f)$ and the element $(W^{F,h_1}_1, g^{F,h_1})$ obtained by performing $(\epsilon,F,h_1)$-scanning on $(W_1,g)$.
But since Proposition \ref{scanning.concordance} guarantees that there are concordances from $(W_0,f)$ to   $(W_0^{F,h_0}, f^{F,h_0})$ and from $(W_1,g)$ to 
$(W_1^{F,h_1}, g^{F,h_1})$, we can conclude that $(W_0,f)$ and $(W_1,g)$ are concordant elements of 
$\widetilde{\Psi}_d(\mathbb{R}^N)(\Delta^p)$. Moreover, Proposition \ref{scanning.concordance} also ensures that the final concordance between $(W_0,f)$ and $(W_1,g)$ is trivial over the product $[0,1]\times\partial\Delta^p$.  
\end{proof}
 
\section{Spaces of piecewise linear planes}  \label{section.grass}

\subsection{The PL Madsen-Tillmann spectrum} \label{section.mtpld} As indicated in the introduction, in this section we will introduce the \textit{PL Madsen-Tillmann spectrum} $\mathbf{MT}PL(d)$. We will also prove that there is a weak equivalence of spectra $\mathbf{MT}PL(d) \rightarrow  \widetilde{\Psi}_d$. 

To construct the spectrum $\mathbf{MT}PL(d)$, we start by introducing the quasi-PL spaces which will play the role of the smooth grassmannian and affine smooth grassmannian in the piecewise linear category.  

\theoremstyle{definition} \newtheorem{grassm}{Convention}[section]

\begin{grassm} \label{grassm}
We will use the following notation in our definitions: 

\begin{itemize}
\item[$\cdot$] For an arbitrary PL space $P$, we will denote the standard projection $P\times\mathbb{R}^N\rightarrow P$ by $\pi$. Also, for any element $W$ of $\Psi_d(\mathbb{R}^N)(P)$, we will denote the restriction of $\pi$ on $W$ by $\pi_W$. 

\item[$\cdot$] We will continue to denote the origin of $\mathbb{R}^N$ by $\mathbf{0}$.

\item[$\cdot$] $s_{\mathbf{0}}: P \hookrightarrow P\times\mathbb{R}^N$ will be the standard inclusion defined by $s_{\mathbf{0}}(x) = (x,\mathbf{0})$.
 
\item[$\cdot$] For any $(W,f)$ in $\widetilde{\Psi}_d(\mathbb{R}^N)(P)$, $V_f$ will denote the pre-image $f^{-1}(\mathbb{R}^N)$ and $\tilde{f}: V_f \rightarrow W$ will be the map defined by $\tilde{f}(x)= (x, f(x))$. As pointed out in Remark \ref{remark.marking}, the image of $\tilde{f}$ is the graph of the marking function $f$.  
\end{itemize}
\end{grassm}

Recall that, given two PL sets $\mathcal{F}, \mathcal{F}': \mathbf{PL}^{op} \rightarrow \mathbf{Sets}$,
we say that $\mathcal{F}'$ is a \textit{PL subset} of $\mathcal{F}$ if 
$\mathcal{F}'(P) \subseteq \mathcal{F}(P)$ for all PL spaces $P$ and, for any PL map
$f: Q \rightarrow P$, the induced function 
$f^*: \mathcal{F}'(P) \rightarrow \mathcal{F}'(Q)$ is the restriction of the 
induced function $f^*: \mathcal{F}(P) \rightarrow \mathcal{F}(Q)$ on $\mathcal{F}'(P)$.
If, additionally, $\mathcal{F}$ and $\mathcal{F}'$ happen to be quasi-PL spaces,
then we say that $\mathcal{F}'$ is a \textit{quasi-PL subspace} of $\mathcal{F}$.  

\theoremstyle{definition} \newtheorem{pl.grassmannians}[grassm]{Definition}

\begin{pl.grassmannians} \label{pl.grassmannians}
$\Psi_d^{\circ}(\mathbb{R}^N)$ and $\mathrm{Gr}_d(\mathbb{R}^N)$ are the quasi-PL subspaces of 
$\Psi_d(\mathbb{R}^N)$ such that, for any PL space $P$,
the sets $\Psi_d^{\circ}(\mathbb{R}^N)(P)$ and $\mathrm{Gr}_d(\mathbb{R}^N)(P)$ are 
defined as follows: 

\begin{itemize}
\item[$\cdot$] $\Psi_d^{\circ}(\mathbb{R}^N)(P)$ is the subset of $\Psi_d(\mathbb{R}^N)(P)$ consisting
of all $W$ that contain the product $P\times\{\mathbf{0}\}$. In other words, an element $W$ of $\Psi_d(\mathbb{R}^N)(P)$ is in $\Psi_d^{\circ}(\mathbb{R}^N)(P)$ if every fiber contains the origin $\mathbf{0}$ of $\mathbb{R}^N$.  

\item[$\cdot$] $\mathrm{Gr}_d(\mathbb{R}^N)(P)$ is the subset
of $\Psi_d^{\circ}(\mathbb{R}^N)(P)$
of all elements $W$ such that the diagram
\begin{equation} \label{bundle.diagram1}
\xymatrix{
P \hspace{0.1cm} \ar@{^{(}->}[r]^{\hspace{-1cm}s_{\mathbf{0}}} & 
(P\times\mathbb{R}^N,W) \ar[r]^{\hspace{0.7cm}\pi} & P
}
\end{equation}
is a piecewise linear $(\mathbb{R}^N,\mathbb{R}^d)$-bundle. 
The reader is referred to Appendix \ref{sec.pl.micro} for the definition of 
piecewise linear $(\mathbb{R}^N,\mathbb{R}^d)$-bundle. 
\end{itemize}

We will call $\mathrm{Gr}_d(\mathbb{R}^N)$ the \textit{PL Grassmannian of $d$-planes in $\mathbb{R}^N$}. Also, we define $A\mathrm{Gr}_d^{+}(\mathbb{R}^N)$ and $A\mathrm{Gr}_d(\mathbb{R}^N)$ to be the quasi-PL subspaces of 
$\widetilde{\Psi}_d(\mathbb{R}^N)$ whose values at any PL space $P$ are the following:  

\begin{itemize}
\item[$\cdot$] $A\mathrm{Gr}_d^+(\mathbb{R}^N)(P)$ will be the set of all elements $(W,f)$ with the property that the diagram
\begin{equation}  \label{bundle.diagram2}
\xymatrix{
V_f \hspace{0.1cm} \ar@{^{(}->}[r]^{\hspace{-0.85cm} \tilde{f}} & 
(V_f \times\mathbb{R}^N,W) \ar[r]^{\hspace{0.7cm}\pi} & V_f
}
\end{equation}
is a piecewise linear $(\mathbb{R}^N,\mathbb{R}^d)$-bundle. 

\item[$\cdot$] $A\mathrm{Gr}_d(\mathbb{R}^N)(P)$ is the subset of $A\mathrm{Gr}_d^+(\mathbb{R}^N)(P)$ 
of all elements $(W,f)$ with $f^{-1}(\infty)= \varnothing$. 
\end{itemize}

The quasi-PL space  $A\mathrm{Gr}_d(\mathbb{R}^N)$ will be called the \textit{PL Grassmannian of affine $d$-planes in $\mathbb{R}^N$}. 
\end{pl.grassmannians}

\theoremstyle{definition} \newtheorem{remark.bundle}[grassm]{Remark}

\begin{remark.bundle} \label{remark.bundle}
Given any element $W$ of $\mathrm{Gr}_d(\mathbb{R}^N)(P)$, the fact that the diagram in
(\ref{bundle.diagram1}) is an $(\mathbb{R}^N,\mathbb{R}^d)$-bundle implies that each fiber $W_x$ is a 
locally flat piecewise linear $d$-plane in $\mathbb{R}^N$ which goes through the origin and is closed as a subspace. Similarly, for an element $(W,f)$ of $A\mathrm{Gr}_d(\mathbb{R}^N)(P)$, all the fibers are also locally flat piecewise linear $d$-planes, but the fibers are not required to go through the origin $\mathbf{0}$, and the marking function $f$ in a tuple $(W,f)$ measures how far away a given fiber is from $\mathbf{0}$. Finally, for an element $(W,f)$ of $A\mathrm{Gr}_d^{+}(\mathbb{R}^N)(P)$, fibers are allowed to be empty. As mentioned in the introduction, we will interpret
 $A\mathrm{Gr}_d^{+}(\mathbb{R}^N)$ as the \textit{one-point compactification} 
 of the affine Grassmannian $A\mathrm{Gr}_d(\mathbb{R}^N)$. 
\end{remark.bundle}

Let $A\mathrm{Gr}_d^{+}(\mathbb{R}^N)_{\bullet}$ be the underlying simplicial set of $A\mathrm{Gr}_d^{+}(\mathbb{R}^N)$. For each positive integer $N$, the restriction of the structure map (\ref{structure.map}) of the spectrum $\widetilde{\Psi}_d$ on $S^1\wedge A\mathrm{Gr}_d^{+}(\mathbb{R}^N)_{\bullet}$ gives a map
\begin{equation} \label{structure.map2}
\widetilde{\delta}_N:S^1\wedge A\mathrm{Gr}_d^+(\mathbb{R}^N)_{\bullet}  \rightarrow A\mathrm{Gr}_d^+(\mathbb{R}^{N+1})_{\bullet}.
\end{equation}
These are the maps that we will use to define our model of the Madsen-Tillmann spectrum. 

\theoremstyle{definition} \newtheorem{geometric.mtpl}[grassm]{Definition}

\begin{geometric.mtpl} \label{geometric.mtpl}

\textit{The PL Madsen-Tillmann spectrum} is the spectrum $\mathbf{MT}PL(d)$ whose $N$-th level is $A\mathrm{Gr}_d^{+}(\mathbb{R}^N)_{\bullet}$ and, for each $N$, the structure map $S^1\wedge A\mathrm{Gr}_d^+(\mathbb{R}^N)_{\bullet}  \rightarrow A\mathrm{Gr}_d^+(\mathbb{R}^{N+1})_{\bullet}$ is the map $\widetilde{\delta}_N$ introduced in (\ref{structure.map2}). 
\end{geometric.mtpl}

The canonical inclusions $A\mathrm{Gr}_d^+(\mathbb{R}^N)_{\bullet}\hookrightarrow \widetilde{\Psi}_d(\mathbb{R}^N)_{\bullet}$ assemble into a map of spectra
\begin{equation} \label{map.spectra1}
\mathcal{I}: \mathbf{MT}PL(d)  \rightarrow \widetilde{\Psi}_d,
\end{equation}
and in Theorem \ref{hatcher} we will show that this map is a weak equivalence.  

\subsection{Equivalence of PL Grassmannians} \label{equivalence.pl.grassmannians} 

Let $\mathrm{Gr}_d(\mathbb{R}^N)_{\bullet}$ and $\Psi_d^{\circ}(\mathbb{R}^N)_{\bullet}$ be the 
underlying simplicial sets of  $\mathrm{Gr}_d(\mathbb{R}^N)$ and $\Psi_d^{\circ}(\mathbb{R}^N)$ respectively. 
An important intermediate step to prove Theorem \ref{hatcher} is to show that the canonical inclusion
\[
\mathcal{J}_N: \mathrm{Gr}_d(\mathbb{R}^N)_{\bullet} \hookrightarrow \Psi_d^{\circ}(\mathbb{R}^N)_{\bullet}
\]
is a weak homotopy equivalence. For the proof of this weak equivalence,  
we require an alternative model for the PL Grassmannian, which we will introduce in Definition \ref{germ.grassmannian}. 
To introduce this new model of the PL Grassmannian,   
we require a preliminary construction, which will be given
in Definition \ref{equivalence.germs}. Before reading the details of this
construction, the reader is encouraged to review the definition of the PL 
set $\Psi_d(U): \mathbf{PL}^{op} \rightarrow \mathbf{Sets}$ given in 
Definition \ref{spaceman}. From now on, we shall 
typically use $V$ instead of $U$ to denote open sets in $\mathbb{R}^N$. 

\theoremstyle{definition} \newtheorem{equivalence.germs}[grassm]{Definition}

\begin{equivalence.germs} \label{equivalence.germs}
Once again,
let $\mathbf{0}$ denote the origin of $\mathbb{R}^N$. 

\begin{itemize}
\item[(i)] For any open neighborhood $V$ of $\mathbf{0}$
in $\mathbb{R}^N$, let $\Psi_d^{\circ}(V)$ be the quasi-PL
subspace of $\Psi_d(V)$ such that, for any PL space $P$, 
$\Psi_d^{\circ}(V)(P)$ is the subset of $\Psi_d(V)(P)$
of all $W$ that contain the product $P \times \{\mathbf{0}\}$.
If $V = \mathbb{R}^N$, we get the quasi-PL space 
$\Psi_d^{\circ}(\mathbb{R}^N)$ introduced in 
Definition \ref{pl.grassmannians}. 

\item[(ii)] For any PL space $P$, 
let $\Psi_d(\mathbf{0},P)$ be the disjoint union of sets
\[
\coprod_{V} \Psi_d^{\circ}(V)(P),
\]
where $V$ ranges over all open neighborhoods of $\mathbf{0}$ in $\mathbb{R}^N$. 
Two elements $W_1$ and $W_2$ of $\Psi_d(\mathbf{0},P)$ are said to be \textit{germ equivalent} (expressed in symbols as $W_1 \sim_g W_2$) if 
\[
W_1 \cap V' = W_2 \cap V'
\]
for some open neighborhood $V'$  in $P\times \mathbb{R}^N$ of the zero-section $P\times\{\mathbf{0}\}$. 
\end{itemize}
\end{equivalence.germs}

It is straightforward to verify that $\sim_g$ is an equivalence relation on $\Psi_d(\mathbf{0},P)$. From now on, the equivalence class of an element $W$ of $\Psi_d(\mathbf{0},P)$ with respect to $\sim_g$ will be  denoted by $[W]$. We shall call $[W]$ \textit{the germ of $W$ near $\mathbf{0}$}. 
Also, for our second model of the PL Grassmannian, we will adopt the following notation: If $V$ is any open neighborhood of the origin in $\mathbb{R}^N$,
then $f^*_V$ will denote the function $\Psi_d^{\circ}(V)(P) \rightarrow \Psi_d^{\circ}(V)(Q)$ induced by a PL map $f:Q \rightarrow P$.   

\theoremstyle{definition} \newtheorem{germ.grassmannian}[grassm]{Definition}

\begin{germ.grassmannian}  \label{germ.grassmannian}
The \textit{germ PL Grassmannian} is the quasi-PL space $\widetilde{\mathrm{Gr}}_d(\mathbb{R}^N): \mathbf{PL}^{op}\rightarrow \mathbf{Sets}$ defined as follows: 

\begin{itemize}
\item[(1)] For any PL space $P$, $\widetilde{\mathrm{Gr}}_d(\mathbb{R}^N)(P) = \Psi_d(0,P)/\sim_g$. In other words, $\widetilde{\mathrm{Gr}}_d(\mathbb{R}^N)(P)$
is the set of all \textit{germs} $[W]$ of elements $W \in \Psi_d(\mathbf{0}, P)$. 

\item[(2)] The function $f^*: \widetilde{\mathrm{Gr}}_d(\mathbb{R}^N)(P) \rightarrow \widetilde{\mathrm{Gr}}_d(\mathbb{R}^N)(Q)$ induced by a piecewise linear map $f:Q \rightarrow P$ is defined as follows: If $W$ is an element in   $\Psi_d^{\circ}(V)(P) $, then $f^*([W]) = [f^*_VW]$. 
\end{itemize}
 
\end{germ.grassmannian} 

It is not hard to verify that any function of the form  $f^*: \widetilde{\mathrm{Gr}}_d(\mathbb{R}^N)(P) \rightarrow \widetilde{\mathrm{Gr}}_d(\mathbb{R}^N)(Q)$  is well defined, and that the correspondences 
$P \mapsto \widetilde{\mathrm{Gr}}_d(\mathbb{R}^N)(P)$ and $f \mapsto f^*$ 
indeed define a functor of the form $\mathbf{PL}^{op}\rightarrow \mathbf{Sets}$. Also,
it is straightforward to prove that  $\widetilde{\mathrm{Gr}}_d(\mathbb{R}^N)$
satisfies all the requirements for being a quasi-PL space (see Definition \ref{quasiPL}).   

Before we proceed to prove the main results of this section, it is 
useful to first establish the following property of the underlying simplicial set 
$ \widetilde{\mathrm{Gr}}_d(\mathbb{R}^N)_{\bullet}$. 

\theoremstyle{plain} \newtheorem{germ.pathconn}[grassm]{Proposition}

\begin{germ.pathconn} \label{germ.pathconn}
The underlying simplicial set 
$\widetilde{\mathrm{Gr}}_d(\mathbb{R}^N)_{\bullet}$
is path-connected. 
\end{germ.pathconn}

\begin{proof}
Consider two arbitrary $0$-simplices $[W_0]$ and 
$[W_1]$ of $\widetilde{\mathrm{Gr}}_d(\mathbb{R}^N)_{\bullet}$.
To prove that $\widetilde{\mathrm{Gr}}_d(\mathbb{R}^N)_{\bullet}$ is 
path-connected, we shall construct a concordance 
$[\widetilde{W}] \in \widetilde{\mathrm{Gr}}_d(\mathbb{R}^N)([0,1])$
from $[W_0]$ to $[W_1]$. 
First, let $V_0$ and $V_1$ be open neighborhoods of the 
origin $\mathbf{0}$ in $\mathbb{R}^N$ such that 
$W_0 \in \Psi_d^{\circ}(V_0)_0$ and $W_1 \in \Psi_d^{\circ}(V_1)_0$. 
By shrinking $V_0$ and $V_1$ to a common open subspace 
(e.g., $V_0 \cap V_1$) we can assume that $V_0 = V_1$. From now
on, we will denote both $V_0$ and $V_1$ by $V$. Also, 
throughout this proof, we will denote the $d$-fold cartesian 
product $[-1,1]^d$ by $D^d$. Finally, to avoid any notational confusion,
we shall denote the origins of $\mathbb{R}^d$ and $\mathbb{R}^N$
by $\mathbf{0}_d$ and $\mathbf{0}_N$ respectively. 

Since $W_0$ 
and $W_1$ are PL submanifolds of $V$, we can find PL embeddings 
$f, g: D^d \hookrightarrow V$ such that 
$f(\mathbf{0}_d) = g(\mathbf{0}_d) = \mathbf{0}_N$,
$\mathrm{Im}\hspace{0.05cm}f \subseteq W_0$, and 
$\mathrm{Im}\hspace{0.05cm}g \subseteq W_1$. 
The first step in this proof is to show that the PL embeddings 
$f$ and $g$ are concordant; i.e., we will show that there 
is a PL embedding $F: [0,1] \times D^d \hookrightarrow [0,1]\times V$ satisfying the
following conditions: 

\begin{itemize}
\item[(i)] $F^{-1}\big(\{0\} \times V\big) = \{0\} \times D^d$ and 
$F^{-1}\big(\{1\} \times V\big) = \{1\} \times D^d$.

\item[(ii)] $\mathrm{pr}_2 \circ F \circ i_0 = f$ and $\mathrm{pr}_2 \circ F \circ i_1 = g$, where
$\mathrm{pr}_2: [0,1]\times V \rightarrow V$ is the obvious projection onto $V$, 
and $i_0$ (resp. $i_1$) is the inclusion $D^d \hookrightarrow [0,1] \times D^d$ defined by
$x \mapsto (0, x)$ (resp. $x \mapsto (1, x)$).
\end{itemize}

Note that we do not require 
$F$ to be level-preserving, except at $t = 0$ and $t=1$.  
A PL embedding $F$ satisfying the above conditions  
is called a \textit{concordance} from $f$ to $g$. 
Furthermore, we shall require $F$ to satisfy the following additional condition: 

\begin{itemize}

\item[(iii)] $F(t, \mathbf{0}_d) = (t, \mathbf{0}_N)$ for all $t \in [0,1]$. 

\end{itemize}

We will construct such a concordance $F: [0,1] \times D^d \hookrightarrow [0,1]\times V$
as follows:  

\begin{itemize}
\item[(1)] First, we define a PL map $G: [0,1] \times D^d \rightarrow [0,1]\times V$
(not necessarily an embedding)
satisfying  $G^{-1}\big(\{0\} \times V\big) = \{0\} \times D^d$,
$G^{-1}\big(\{1\} \times V\big) = \{1\} \times D^d$, $\mathrm{pr}_2 \circ G \circ i_0 = f$,
$\mathrm{pr}_2 \circ G \circ i_1 = g$, and 
$G(t, \mathbf{0}_d) = (t, \mathbf{0}_N)$ for all $t \in [0,1]$. In fact, we can
define $G$ so that it is level-wise constant near $t = 0$ and $t = 1$. 
To do this, start by fixing a value $\epsilon$ such that $0 < \epsilon < \frac{1}{2}$. 
We can define the PL map $G$ by first setting $G(t, x) = (t, f(x))$
for all $(t,x) \in [0,\epsilon] \times D^d$, 
$G(t, x) = (t, g(x))$ for all $(t, x) \in [1 - \epsilon, 1] \times  D^d$,
and $G(t, \mathbf{0}_d) = (t, \mathbf{0}_N)$ for all $t \in [0,1]$.
Then, we triangulate $[0, 1] \times D^d$ with 
a simplicial complex $K$ which also triangulates the PL subspace
\[
X_0 := \Big( [0,\epsilon] \times D^d\Big) \cup \Big( [1 - \epsilon, 1] \times D^d \Big) \cup \Big( [0,1] \times \{\mathbf{0}_d\}\Big),
\] 
and makes the map $G$ linear on each simplex of $K$ contained in $X_0$. 
We then extend $G$ linearly to all the remaining simplices of  $K$.  

\item[(2)] Now, recall that we are assuming that $N \geq 2d + 2$ 
(see Remark \ref{locally.flat}). This condition ensures that the PL map 
$G$ constructed in the previous step is PL homotopic to a PL
embedding $F: [0,1] \times D^d \hookrightarrow [0,1]\times V$
(see Corollary 4.4 in \cite{Br}). 
Since $G$ was already a PL embedding on the PL subspace
$X_0$, we
can take $F$ to be PL homotopic to $G$ relative to $X_0$. 
In fact, if $H_s$ is the PL homotopy from $F$ to $G$,
Corollary 4.4 in \cite{Br} tells us that we can choose 
$H_s$ so that, for all $s \in [0,1]$ and $(t,x)\in [0,1]\times D^d$, the distance between 
$G(t,x)$ and $H_s(t,x)$ is at most $\frac{\epsilon}{2}$. This last condition 
ensures that the resulting PL embedding $F: [0,1]\times D^d \hookrightarrow [0,1]\times V$
satisfies  $F^{-1}\big(\{0\} \times V\big) = \{0\} \times D^d$ and
$F^{-1}\big(\{1\} \times V\big) = \{1\} \times D^d$. Finally, since $H_s$ is a PL homotopy
relative to $X_0$,
we also have that $\mathrm{pr}_2 \circ F \circ i_0 = f$,
$\mathrm{pr}_2 \circ F \circ i_1 = g$, and 
$F(t, \mathbf{0}_d) = (t, \mathbf{0}_N)$ for all $t \in [0,1]$.
Therefore, $F$ is a concordance 
from $f$ to $g$ satisfying the additional condition (iii) listed above.  
\end{itemize}

As also indicated in Remark \ref{locally.flat}, the dimensions 
$N$ and $d$ satisfy $N - d \geq 3$. This bound on $N - d$ allows us 
to apply a result of Hudson that tells us that concordant embeddings 
are always ambient isotopic (in particular, isotopic), as long as we are working with codimension
at least 3 (see Theorem 1.1 in \cite{conhud}. However, Corollary 1.4 
of the same paper is a better fit for this proof). 
Consequently, by this result of Hudson, our PL embeddings $f, g: D^d \rightarrow V$
are isotopic. In other words,  
we can find a PL isotopy $\widetilde{F}: [0,1]\times D^d \rightarrow [0,1]\times V$ 
(i.e., a level-preserving PL embedding)
such that $\widetilde{F}_0 = f$ and $\widetilde{F}_1 = g$. 
Moreover, since we constructed the concordance $F$ so that $F(t, \mathbf{0}_d) = (t, \mathbf{0}_N)$ for any $t \in [0,1]$, 
we can also ensure that $\widetilde{F}_t(\mathbf{0}_d) = \mathbf{0}_N$ for all $t \in [0,1]$. 

Using this PL isotopy $\widetilde{F}: [0,1]\times D^d \rightarrow [0,1]\times V$, 
we can now construct the desired concordance $[\widetilde{W}]$ between 
the 0-simplices $[W_0]$ and $[W_1]$ of $\widetilde{\mathrm{Gr}}_d(\mathbb{R}^N)_{\bullet}$
that we fixed at the beginning of this proof. First,
let us denote the image $\widetilde{F}([0,1]\times D^d)$ by $\widehat{W}$. Evidently,
this $\widehat{W}$ is not an element of $\Psi_d^{\circ}(V)([0,1])$ because each fiber
of the standard projection $\pi: \widehat{W} \rightarrow [0,1]$ is a manifold with non-empty boundary. However,
we can pick a small enough open neighborhood $V'$ of $\mathbf{0}_N$ in $V$ so that 
the product $[0,1] \times V'$ is disjoint from the image
$\widetilde{F}([0,1]\times \partial \hspace{0.05cm}D^d)$. Then, by 
intersecting $\widehat{W}$ with $[0,1]\times V'$, we obtain an
element $\widetilde{W}$ of $\Psi_d^{\circ}(V')([0,1])$. 
Also, recall that $f(D^d) \subset W_0$ and $g(D^d) \subset W_1$. Therefore,
since $\widetilde{F}$ is an isotopy from $f$ to $g$, the element 
$\widetilde{W} \in \Psi_d^{\circ}(V')([0,1])$ is 
(after possibly shrinking $V'$) a concordance 
from $W_0\cap V'$ to $W_1\cap V'$, which are both elements 
of $\Psi_d^{\circ}(V')_0$. Since $W_0$ and
$W_1$ are germ equivalent to  
$W_0\cap V'$ and $W_1\cap V'$ respectively
(in the sense of Definition \ref{equivalence.germs}), 
it follows that the germ 
$[\widetilde{W}] \in \widetilde{\mathrm{Gr}}_d(\mathbb{R}^N)([0,1])$ 
of $\widetilde{W}$ is a concordance from $[W_0]$ to $[W_1]$,
which concludes the proof.
\end{proof}

Let $\mathcal{G}_N: \Psi_d^{\circ}(\mathbb{R}^N)_{\bullet} \rightarrow \widetilde{\mathrm{Gr}}_d(\mathbb{R}^N)_{\bullet}$ be the morphism of simplicial sets which sends a $p$-simplex $W$ to its germ $[W]$ near the origin. The following proposition is one of the results that we will use to establish that the inclusion map 
$\mathcal{J}_N: \mathrm{Gr}_d(\mathbb{R}^N)_{\bullet} \hookrightarrow \Psi_d^{\circ}(\mathbb{R}^N)_{\bullet}$ is a weak homotopy equivalence. 

\theoremstyle{plain} \newtheorem{germ.equivalence}[grassm]{Proposition}

\begin{germ.equivalence} \label{germ.equivalence}
The morphism of simplicial sets 
$\mathcal{G}_N: \Psi_d^{\circ}(\mathbb{R}^N)_{\bullet} \rightarrow \widetilde{\mathrm{Gr}}_d(\mathbb{R}^N)_{\bullet}$ 
is a weak homotopy equivalence.  
\end{germ.equivalence}

\begin{proof}

This proof will be broken down into the following two steps: 

\textit{\underline{Step 1}: $\mathcal{G}_N$ is a Kan fibration.} Consider any diagram of the form
\begin{equation} \label{diagram.germ.fib}
\xymatrix{ \Lambda^p_{i\bullet} \hspace{0.1cm} \ar@{^{(}->}[d] \ar[r] & 
\Psi_d^{\circ}(\mathbb{R}^N)_{\bullet} \ar[d]^{\mathcal{G}_N}  \\
\Delta^p_{\bullet} \ar[r] & \widetilde{\mathrm{Gr}}_d(\mathbb{R}^N)_{\bullet},
}
\end{equation}
where $\Lambda^p_{i\bullet}$ is the $i$-th horn of $\Delta^p_{\bullet} $. To show that $\mathcal{G}_N$ is a Kan fibration, we need to produce a lift $\Delta^p_{\bullet} \rightarrow \Psi_d^{\circ}(\mathbb{R}^N)_{\bullet}$. First, observe that (\ref{diagram.germ.fib})  can be represented with the following data: 

\begin{itemize}
\item[$\cdot$] The bottom map is defined by an element $W'$ of $\Psi_d^{\circ}(V)(\Delta^p)$, where $V$ is some suitable open neighborhood of the origin in $\mathbb{R}^N$. 

\item[$\cdot$] The top map is defined by an element $W$ of $\Psi_d^{\circ}(\mathbb{R}^N)(\Lambda^p_i)$, where $\Lambda^p_i$ is the $i$-th horn of the geometric simplex $\Delta^p$.
\end{itemize}

 By the commutativity of (\ref{diagram.germ.fib}),  we can assume (after shrinking $V$ if necessary)
 that the image of $W$ under the restriction map 
$r: \Psi_d^{\circ}(\mathbb{R}^N) \rightarrow \Psi_d^{\circ}(V)$ coincides with the element $W'_{\Lambda^p_i}$.  Now fix the following data: 

\begin{itemize}
\item[$\cdot$] An open PL embedding
\[
h: \Delta^p \times \mathbb{R}^d \rightarrow  W'
\]
which commutes with the projection onto $\Delta^p$ and preserves the zero-section. We can produce such an embedding by applying the Kuiper-Lashof Theorem (see Theorem \ref{kuiper.lashof}) to the microbundle  
\[
\xymatrix{ \Delta^p \hspace{0.1cm} \ar@{^{(}->}[r]^{\hspace{-0.1cm}s_{\mathbf{0}}} & 
W' \ar[r]^{\hspace{-0.1cm}\pi_{W'}} &  \Delta^p.
}
\]

\item[$\cdot$] A PL homeomorphism $f: \Lambda^p_i \times [0,1]\rightarrow \Delta^p$ such that $f(x,0) = x$ for any 
$x \in \Lambda^p_i$.  
\end{itemize}

Also, let $r: \Delta^p \rightarrow \Lambda^p_i$ be the PL retraction defined by $f(x,t) \mapsto x$,  let $D^d$ denote again the unit cube in $\mathbb{R}^d$, and let $h'$ be the restriction of $h$ on $\Delta^p\times D^d$.  Since $N- d\geq 3$ (see Remark \ref{locally.flat}), 
we can apply the Isotopy Extension Theorem (see \cite{Br}, Corollary 7.7) to produce an ambient isotopy
\[
H: \Delta^p \times \mathbb{R}^N \rightarrow  \Delta^p \times \mathbb{R}^N
\]
over $\Delta^p$ which preserves the zero-section  and has the following two properties: 

\begin{itemize}
\item[$\cdot$] $H$ is the identity map on $\Lambda^p_i\times \mathbb{R}^N$. 

 \item[$\cdot$] For any $(x,t)$ in $\Lambda^p_i \times [0,1]$, $H_{f(x,t)}\circ h'_{f(x,t)}$ is equal to $h'_x$, where $h'_{f(x,t)}$ and $h'_x$ are the values of the isotopy $h'$ over $f(x,t)$ and $x$ respectively. Similarly, $H_{f(x,t)}$ is the value of the ambient isotopy 
 $H$ over the point $f(x,t)$. 
\end{itemize}

Note that the pullback $r^*W$ of $W$ along the retraction $r: \Delta^p \rightarrow \Lambda^p_i$ contains the image 
$H\circ h'(\Delta^p \times D^d)$ as a PL subspace. 
Therefore, the image of $r^*W$ under the inverse $H^{-1}$, which is an element of $\Psi_d(\mathbb{R}^N)(\Delta^p)$, 
contains the image $h'(\Delta^p \times D^d)$. 
Finally, since $H^{-1}(r^*W)$ is closed in $\Delta^p\times \mathbb{R}^N$, we can find an open neighborhood $V'$ of the origin in $\mathbb{R}^N$ contained in $V$ such that the intersections of $H^{-1}(r^*W)$ and $W'$ with $\Delta^p \times V'$ are equal.
In other words, the element $H^{-1}(r^*W)$ defines a lift  $\Delta^p_{\bullet} \rightarrow \Psi_d^{\circ}(\mathbb{R}^N)_{\bullet}$ in (\ref{diagram.germ.fib}).

\textit{\underline{Step 2}: $\mathcal{G}_N$ has contractible fibers.}  Let $[\mathbb{R}^d]$ denote  
the element of $\widetilde{\mathrm{Gr}}_d(\mathbb{R}^N)_0$ induced by the image of the standard inclusion 
$\mathbb{R}^d \hookrightarrow \mathbb{R}^N$, and let $\widetilde{\mathrm{Gr}}_d(\mathbb{R}^N)_{[\mathbb{R}^d]}$ denote the subsimplicial set of $\widetilde{\mathrm{Gr}}_d(\mathbb{R}^N)_{\bullet}$ which consists of all the degeneracies of $[\mathbb{R}^d]$.
In this step, we will show that the pre-image of $\widetilde{\mathrm{Gr}}_d(\mathbb{R}^N)_{[\mathbb{R}^d]}$ under $\mathcal{G}_N$, which we will denote by $\mathcal{G}_N^{-1}([\mathbb{R}^d])_{\bullet}$, is contractible. 
Note that the contractibility of $\mathcal{G}_N^{-1}([\mathbb{R}^d])_{\bullet}$ would then imply 
that all fibers of $\mathcal{G}_N$ are contractible, given that $\mathcal{G}_N$ is a Kan fibration
by Step 1 and $\widetilde{\mathrm{Gr}}_d(\mathbb{R}^N)_{\bullet}$ is path-connected by Proposition \ref{germ.pathconn}.  
The base-point that we will consider in the pre-image $\mathcal{G}_N^{-1}([\mathbb{R}^d])_{\bullet}$ is the 0-simplex corresponding to the image of the inclusion $\mathbb{R}^d \hookrightarrow \mathbb{R}^N$. From now on, we will denote this base-point by $\ast_{\mathbb{R}^d}$. Choose now a $W$ in $\mathcal{G}_N^{-1}([\mathbb{R}^d])_p$ which represents an element of the group 
$\pi_p(\mathcal{G}_N^{-1}([\mathbb{R}^d])_{\bullet},\ast_{\mathbb{R}^d})$. In other words, $W$ is an element whose restriction $W_{\partial\Delta^p}$ over $\partial\Delta^p$ is equal to $\partial\Delta^p\times\mathbb{R}^d$ and
 whose germ near the section $\Delta^p \times \{ \mathbf{0}\}$ 
 coincides with $\Delta^p \times \mathbb{R}^d$.  
 By this last property, we can find a small enough $\epsilon >0$ so that 
 \[
 W\cap(\Delta^p\times D^N_{\epsilon}) = (\Delta^p\times\mathbb{R}^d)\cap(\Delta^p\times D^N_{\epsilon}),
 \]
where $D^N_{\epsilon}$ is the cube of radius $\epsilon$ in $\mathbb{R}^N$. By scanning the cube $D^N_{\epsilon}$, we can produce a concordance $\widehat{W}$ between $W$ and the product $\Delta^p\times \mathbb{R}^d$. Furthermore, this scanning procedure will not change any fibers of $W$ which were equal to $\mathbb{R}^d$, which implies that the restriction of $\widehat{W}$ over $\partial\Delta^p\times [0,1]$ coincides with the product 
$(\partial\Delta^p\times [0,1])\times \mathbb{R}^d$. Therefore, the element $W$ represents the trivial element in 
$\pi_p(\mathcal{G}_N^{-1}([\mathbb{R}^d])_{\bullet},\ast_{\mathbb{R}^d})$, and  it follows that 
$\mathcal{G}_N^{-1}([\mathbb{R}^d])_{\bullet}$ is contractible. 
\end{proof}

One of our goals in this section is to 
show that the underlying simplicial sets of the
PL Grassmannians  $\mathrm{Gr}_d(\mathbb{R}^N)$ and  $\widetilde{\mathrm{Gr}}_d(\mathbb{R}^N)$ are actually weak homotopy 
equivalent. To establish this, we need to give a brief discussion about spaces of PL automorphisms of $\mathbb{R}^N$. The following definitions that we will introduce can also be found in \cite{KL2}.

\theoremstyle{definition} \newtheorem{kl.convention}[grassm]{Convention}

\begin{kl.convention} \label{kl.convention}
To make it easier to formulate our definitions, we will borrow the following terminology from \cite{KL2}: Any PL automorphism $f:\Delta^p\times M \rightarrow \Delta^p\times M$ which commutes with the projection onto $\Delta^p$ will be called a \textit{PL bundle isomorphism}. Similarly, any PL embedding $g: \Delta^p\times V \rightarrow \Delta^p\times M$ commuting with the projection onto $\Delta^p$ will be called  a \textit{PL bundle monomorphism}.  

Also, in the following definitions, we will denote the image of the standard inclusion 
$\mathbb{R}^d \hookrightarrow \mathbb{R}^N$ simply by $\mathbb{R}^d$.  
\end{kl.convention}

\theoremstyle{definition} \newtheorem{automorphisms}[grassm]{Definition}

\begin{automorphisms} \label{automorphisms}
$\mathcal{H}(\mathbb{R}^N)$ is the simplicial group whose $p$-simplices are PL bundle isomorphisms $f: \Delta^p \times \mathbb{R}^N \rightarrow \Delta^p \times \mathbb{R}^N$ which preserve the zero-section. Also, $\mathcal{H}(\mathbb{R}^N,\mathbb{R}^d)$ will denote the simplicial subgroup of $\mathcal{H}(\mathbb{R}^N)$ whose $p$-simplices are the 
PL bundle isomorphisms $f:\Delta^p \times \mathbb{R}^N \rightarrow \Delta^p \times \mathbb{R}^N$ 
which preserve the subspace $\Delta^p\times \mathbb{R}^d$. 
\end{automorphisms}

We also need to introduce simplicial groups whose simplices are germs of automorphisms. To define these, we will use the following spaces of embeddings:  

\begin{itemize}
\item[$\cdot$] For any open neighborhood $V$ of the origin $\mathbf{0}$ in $\mathbb{R}^N$, $\mathcal{E}_V(\mathbb{R}^N)$ is the simplicial set where $p$-simplices are PL bundle monomorphisms $g: \Delta^p \times V \rightarrow \Delta^p\times \mathbb{R}^N$ which preserve the zero-section. 

\item[$\cdot$] $\mathcal{E}(\mathbb{R}^N)$ is the coproduct 
\[
\coprod_{V} \mathcal{E}_V(\mathbb{R}^N)
\]
where $V$ ranges over all possible open neighborhoods $V$ of the origin $\mathbf{0}$.

\item[$\cdot$]  Also,  $\mathcal{E}(\mathbb{R}^N, \mathbb{R}^d)$ will denote the subsimplicial set of 
$\mathcal{E}(\mathbb{R}^N)$  whose $p$-simplices are the embeddings $g: \Delta^p \times V \rightarrow \Delta^p\times \mathbb{R}^N$ which map $\Delta^p \times (V\cap \mathbb{R}^d)$ to $\Delta^p \times \mathbb{R}^d$.
\end{itemize}

The following simplicial group is what is commonly referred to as the structure group $\mathrm{PL}_n$ in the literature (see for example \cite{KL2}).   

\theoremstyle{definition} \newtheorem{germ.automorphisms}[grassm]{Definition}

\begin{germ.automorphisms} \label{germ.automorphisms}

$\mathrm{PL}(\mathbb{R}^N)$ is the quotient of $\mathcal{E}(\mathbb{R}^N)$ obtained by identifying two embeddings 
$f_i:\Delta^p \times V_i\rightarrow \Delta^p\times\mathbb{R}^N$, $i=1,2$, if there exists an open neighborhood $V_3$ of $\mathbf{0}$ in $V_1\cap V_2$ such that 
$f_1|_{\Delta^p\times V_3} = f_2|_{\Delta^p\times V_3}$. Similarly, $\mathrm{PL}(\mathbb{R}^N, \mathbb{R}^d)$ is the quotient of $\mathcal{E}(\mathbb{R}^N, \mathbb{R}^d)$ obtained by performing the same identifications that we did for $\mathrm{PL}(\mathbb{R}^N)$.  

\end{germ.automorphisms}

\theoremstyle{definition} \newtheorem{PL.group}[grassm]{Remark}

\begin{PL.group} \label{PL.group}
The reason why $\mathrm{PL}(\mathbb{R}^N)$ is actually a simplicial group, and not just a simplicial set, is because we are working with germs of embeddings. For any monomorphism $f$ in  $\mathcal{E}_V(\mathbb{R}^N)$, we will denote its equivalence class in $\mathrm{PL}(\mathbb{R}^N)$ by $[f]$. If we take two  $p$-simplices $f_1$ and $f_2$ in $\mathcal{E}_{V_1}(\mathbb{R}^N)$ and $\mathcal{E}_{V_2}(\mathbb{R}^N)$ respectively, then the product $[f_2]\cdot [f_1]$ is defined as the equivalence class $[f_2\circ (f_1|_{\Delta^p\times V_3})]$, where $V_3$ is an open neighborhood of the origin contained  in $V_1$, chosen to be small enough so that $f_1(\Delta^p\times V_3) \subseteq \Delta^p\times V_2$. Also, given any $p$-simplex $f$ of $\mathcal{E}_{V}(\mathbb{R}^N)$, the inverse of $[f]$ will be the equivalence class defined by $f^{-1}|_{\Delta^p \times V'}$, where $V'$ is an open neighborhood of the origin in $\mathbb{R}^N$ such that 
$\Delta^p\times V' \subseteq \mathrm{Im}\hspace{0.05cm}f$. 
\end{PL.group}

\theoremstyle{definition} \newtheorem{pre.quotient.grassmannian}[grassm]{Note}

\begin{pre.quotient.grassmannian} \label{pre.quotient.grassmannian}

Recall that we can define the smooth Grassmannian as the quotient $O(N)/O(d)\times O(N-d)$. In the next proposition, we show that our PL Grassmannians $\mathrm{Gr}_d(\mathbb{R}^N)_{\bullet}$ and $\widetilde{\mathrm{Gr}}_d(\mathbb{R}^N)_{\bullet}$ admit similar descriptions in terms of PL automorphisms. Before we state this result, we need to introduce the following maps: 

(1) \quad Let $\mathcal{F}_N: \mathcal{H}(\mathbb{R}^N) \rightarrow \mathrm{Gr}_d(\mathbb{R}^N)_{\bullet}$ be the simplicial set map which sends a $p$-simplex $f$ of $\mathcal{H}(\mathbb{R}^N)$ to the image $f(\Delta^p\times \mathbb{R}^d)$.  

(2) \quad Also, we define a simplicial set map  
$\widetilde{\mathcal{F}}_N: \mathrm{PL}(\mathbb{R}^N) \rightarrow \widetilde{\mathrm{Gr}}_d(\mathbb{R}^N)_{\bullet}$ as follows:

\begin{itemize}
\item[$\cdot$] First, we claim that any $g \in \mathcal{E}_V(\mathbb{R}^N)$ induces canonically an element in $\widetilde{\mathrm{Gr}}_d(\mathbb{R}^N)_p$.
To see this, fix a bundle monomorphism $g:\Delta^p \times V \rightarrow \Delta^p \times \mathbb{R}^N$, where $V$ is an open set in $\mathbb{R}^N$ containing $\mathbf{0}$.
By the compactness of $\Delta^p$, we can find an open neighborhood $V'$ of $\mathbf{0}$ in $\mathbb{R}^N$ such that 
$\Delta^p \times V' \subset g(\Delta^p \times V)$. 
Since $g:\Delta^p \times V \rightarrow \Delta^p \times \mathbb{R}^N$ is a PL embedding, 
it is not hard to show that the intersection of $\Delta^p \times V'$ with
\[
g\big(\Delta^p\times(V \cap \mathbb{R}^d)\big)
\]
is a $p$-simplex of $\Psi_d^{\circ}(V')_{\bullet}$, which we shall denote by $W^{g}$. Then, we obtain an element 
in  $\widetilde{\mathrm{Gr}}_d(\mathbb{R}^N)_p$ by taking the germ $[W^{g}]$. Note that this germ does not depend on the
open neighborhood $V'$ we chose.   

\item[$\cdot$] Next, we define a simplicial set map 
$\widetilde{\mathcal{F}}_N^{\mathcal{E}}: \mathcal{E}(\mathbb{R}^N) \rightarrow  \widetilde{\mathrm{Gr}}_d(\mathbb{R}^N)_{\bullet}$
by setting $\widetilde{\mathcal{F}}_N^{\mathcal{E}}(g) = [W^g]$. 
Recall that $\mathcal{E}(\mathbb{R}^N)$ is the coproduct $\coprod_{V} \mathcal{E}_V(\mathbb{R}^N)$. 

\item[$\cdot$] Finally, it is straightforward to verify that two bundle monomorphisms $g_1, g_2 \in \mathcal{E}(\mathbb{R}^N)_p$ with the same germ 
near $\Delta^p \times \{ \mathbf{0}\}$ induce the same element in  $\widetilde{\mathrm{Gr}}_d(\mathbb{R}^N)_p$, i.e., $[W^{g_1}] = [W^{g_2}]$. Therefore, 
the map $\widetilde{\mathcal{F}}_N^{\mathcal{E}}: \mathcal{E}(\mathbb{R}^N) \rightarrow  \widetilde{\mathrm{Gr}}_d(\mathbb{R}^N)_{\bullet}$ 
factors through $\mathrm{PL}(\mathbb{R}^N)$, and we define 
$\widetilde{\mathcal{F}}_N: \mathrm{PL}(\mathbb{R}^N) \rightarrow \widetilde{\mathrm{Gr}}_d(\mathbb{R}^N)_{\bullet}$
to be the map induced by $\widetilde{\mathcal{F}}_N^{\mathcal{E}}$. 
\end{itemize}
\end{pre.quotient.grassmannian}

\theoremstyle{plain} \newtheorem{quotient.grassmannian}[grassm]{Proposition}

\begin{quotient.grassmannian} \label{quotient.grassmannian}

The maps $\mathcal{F}_N: \mathcal{H}(\mathbb{R}^N) \rightarrow \mathrm{Gr}_d(\mathbb{R}^N)_{\bullet}$ and 
$\widetilde{\mathcal{F}}_N: \mathrm{PL}(\mathbb{R}^N) \rightarrow \widetilde{\mathrm{Gr}}_d(\mathbb{R}^N)_{\bullet}$ induce isomorphisms 
\[
\mathcal{H}(\mathbb{R}^N)/\mathcal{H}(\mathbb{R}^N, \mathbb{R}^d) \stackrel{\cong}{\longrightarrow} \mathrm{Gr}_d(\mathbb{R}^N)_{\bullet} \qquad \mathrm{PL}(\mathbb{R}^N)/\mathrm{PL}(\mathbb{R}^N,\mathbb{R}^d)  \stackrel{\cong}{\longrightarrow} \widetilde{\mathrm{Gr}}_d(\mathbb{R}^N)_{\bullet}.
\]
\end{quotient.grassmannian}

\begin{proof}
To prove that the induced map $\mathcal{H}(\mathbb{R}^N)/\mathcal{H}(\mathbb{R}^N, \mathbb{R}^d)\rightarrow \mathrm{Gr}_d(\mathbb{R}^N)_{\bullet}$ is an isomorphism, it suffices to observe that two maps 
$f_1, \hspace{0.05cm} f_2: \Delta^p\times \mathbb{R}^N \rightarrow \Delta^p \times \mathbb{R}^N$ have the same image when restricted to $\Delta^p\times \mathbb{R}^d$ if and only if $f_2^{-1}\circ f_1$ is a $p$-simplex of $\mathcal{H}(\mathbb{R}^N, \mathbb{R}^d)$. Similarly, it is straightforward to verify that two $p$-simplices 
$[f_1]$, $[f_2]$ of $\mathrm{PL}(\mathbb{R}^N)$ have the same image in 
$\widetilde{\mathrm{Gr}}_d(\mathbb{R}^N)_p$ under $\widetilde{\mathcal{F}}_N$ if and only if the product $[f_2]^{-1}\cdot [f_1]$ is a $p$-simplex of 
$\mathrm{PL}(\mathbb{R}^N,\mathbb{R}^d)$. 
\end{proof}

Recall that $\mathcal{G}_N: \Psi_d^{\circ}(\mathbb{R}^N)_{\bullet} \rightarrow \widetilde{\mathrm{Gr}}_d(\mathbb{R}^N)_{\bullet}$ is the map which sends a $p$-simplex $W$ to its germ near the origin. To avoid introducing new notation, we will also denote  the restriction of this map on $\mathrm{Gr}_d(\mathbb{R}^N)_{\bullet}$ by $\mathcal{G}_N$. Now consider the diagram of simplicial sets
\begin{equation} \label{principal.map}
\xymatrix{ \mathcal{H}(\mathbb{R}^N, \mathbb{R}^d) \ar[d]  \hspace{0.1cm} \ar@{^{(}->}[r]  &  
\mathcal{H}(\mathbb{R}^N) \ar[r]^{\hspace{-0.2cm}\mathcal{F}_N} \ar[d] &
\mathrm{Gr}_d(\mathbb{R}^N)_{\bullet} \ar[d]^{\mathcal{G}_N} \\
 \mathrm{PL}(\mathbb{R}^N, \mathbb{R}^d)  \hspace{0.1cm} \ar@{^{(}->}[r] & 
 \mathrm{PL}(\mathbb{R}^N) \ar[r]^{\hspace{-0.2cm}\widetilde{\mathcal{F}}_N} &  
 \widetilde{\mathrm{Gr}}_d(\mathbb{R}^N)_{\bullet},
}
\end{equation}
where the first and second vertical maps are the obvious quotient maps. The next step is to prove that 
the map $\mathcal{G}_N: \mathrm{Gr}_d(\mathbb{R}^N)_{\bullet} \rightarrow \widetilde{\mathrm{Gr}}_d(\mathbb{R}^N)_{\bullet}$ 
is a weak homotopy equivalence. 
Note that by Proposition \ref{quotient.grassmannian}, we have that the top and bottom compositions in (\ref{principal.map}) are simplicial principal bundles. Therefore,  since (\ref{principal.map}) commutes, the next lemma implies that $\mathcal{G}_N: \mathrm{Gr}_d(\mathbb{R}^N)_{\bullet} \rightarrow \widetilde{\mathrm{Gr}}_d(\mathbb{R}^N)_{\bullet}$ is a weak homotopy equivalence.  

\theoremstyle{plain} \newtheorem{KuiperLemma}[grassm]{Lemma}

\begin{KuiperLemma} \label{KuiperLemma}

The quotient maps $\mathcal{H}(\mathbb{R}^N) \rightarrow \mathrm{PL}(\mathbb{R}^N)$ and 
$\mathcal{H}(\mathbb{R}^N, \mathbb{R}^d) \rightarrow \mathrm{PL}(\mathbb{R}^N, \mathbb{R}^d)$
are weak homotopy equivalences. 
\end{KuiperLemma}
\begin{proof}
The arguments for proving that the quotient map $\mathcal{H}(\mathbb{R}^N) \rightarrow \mathrm{PL}(\mathbb{R}^N)$ is a weak homotopy equivalence 
are given in  \cite{KL1} and \cite{KL2}. 
Let us give an overview of the main steps of this proof. From now on,  
we will denote the quotient map $\mathcal{H}(\mathbb{R}^N) \rightarrow \mathrm{PL}(\mathbb{R}^N)$ by $\gamma$.
We point out to the reader that,
in \cite{KL2}, simplicial sets and simplicial groups are called \textit{semi-simplicial sets}
and  \textit{semi-simplicial groups} respectively:  

\begin{itemize}
\item[(i)] First, we will show that the quotient map 
$\gamma: \mathcal{H}(\mathbb{R}^N) \rightarrow \mathrm{PL}(\mathbb{R}^N)$
is surjective. Fix then a $p$-simplex $[f]$ of $\mathrm{PL}(\mathbb{R}^N)$. Without loss
of generality, we can suppose that $f$ is a PL embedding over $\Delta^p$
of the form $f: \Delta^p\times B(\mathbf{0}, \epsilon') \rightarrow \Delta^p \times \mathbb{R}^N$,
where $\epsilon' > 0$ is some suitable positive value and 
$B(\mathbf{0}, \epsilon')$ is the open ball of radius $\epsilon'$ centered 
at the origin $\mathbf{0} \in \mathbb{R}^N$.
Now, fix any value $\epsilon$ such that $0 < \epsilon < \epsilon'$. 
Using Lemma 3' (PL) of \cite{KL1} (see page 12 of that article), we can find
a PL homeomorphism $g: \Delta^p \times \mathbb{R}^N  \rightarrow \Delta^p \times \mathbb{R}^N$
over $\Delta^p$ such that 
\begin{equation} \label{KLgerm}
g|_{\Delta^p\times B(\mathbf{0}, \epsilon)} = f|_{\Delta^p\times B(\mathbf{0}, \epsilon)}.
\end{equation}
It follows that the germ $[g]$ is equal to $[f]$. In other words, 
$\gamma(g) = [f]$, and we can conclude
that $\gamma$ is surjective.   
 
\item[(ii)] Next, let $\mathcal{H}^{N(\mathbf{0})}(\mathbb{R}^N)$ be the 
simplicial subgroup of $\mathcal{H}(\mathbb{R}^N)$ whose $p$-simplices 
are all PL homeomorphisms $f:\Delta^p \times \mathbb{R}^N\rightarrow \Delta^p\times \mathbb{R}^N$
over $\Delta^p$
which leave a neighborhood of the zero-section fixed (this simplicial subgroup is a special case
of a more general construction given on page 1 of \cite{KL2}). 
We can define a free right $\mathcal{H}^{N(\mathbf{0})}(\mathbb{R}^N)$-action on 
$\mathcal{H}(\mathbb{R}^N)$ by setting $f\cdot g := f\circ g$ 
for any $p$-simplices $f$ and $g$ in 
$\mathcal{H}(\mathbb{R}^N)$ and $\mathcal{H}^{N(\mathbf{0})}(\mathbb{R}^N)$ respectively.
It is not hard to show that
the quotient map  $\gamma: \mathcal{H}(\mathbb{R}^N) \rightarrow \mathrm{PL}(\mathbb{R}^N)$
factors through the simplicial set of orbits 
$\mathcal{H}(\mathbb{R}^N)/\mathcal{H}^{N(\mathbf{0})}(\mathbb{R}^N)$, and that the induced map
$\mathcal{H}(\mathbb{R}^N)/\mathcal{H}^{N(\mathbf{0})}(\mathbb{R}^N)\rightarrow \mathrm{PL}(\mathbb{R}^N)$
is an isomorphism. Therefore, the quotient map $\gamma$ is a principal 
$\mathcal{H}^{N(\mathbf{0})}(\mathbb{R}^N)$-bundle.  

\item[(iii)] Finally, in Lemma 1.6 of \cite{KL2}, it is proven that $\mathcal{H}^{N(\mathbf{0})}(\mathbb{R}^N)$
is contractible. It now follows that the quotient map $\gamma$ is a weak homotopy equivalence. 
\end{itemize}

We can repeat the argument given above to show that the quotient map 
$\mathcal{H}(\mathbb{R}^N, \mathbb{R}^d) \rightarrow \mathrm{PL}(\mathbb{R}^N, \mathbb{R}^d)$
(which we shall denote by $\widetilde{\gamma}$) is also a weak homotopy equivalence.  
First, as explained in \cite{KL1}, Lemma 3' (PL) also holds for PL automorphisms which preserve PL subspaces of $\mathbb{R}^N$
determined by equations of the form $x_i = 0$. Thus, if one starts with a PL embedding 
$f: \Delta^p\times B(\mathbf{0}, \epsilon') \rightarrow \Delta^p \times \mathbb{R}^N$ which maps 
$\Delta^p\times (B(\mathbf{0}, \epsilon')\cap \mathbb{R}^d)$ to $\Delta^p \times \mathbb{R}^d$, 
one can use Lemma 3' (PL) of \cite{KL1} to find 
a PL homeomorphism $g: \Delta^p \times \mathbb{R}^N  \rightarrow \Delta^p \times \mathbb{R}^N$
over $\Delta^p$ which preserves $\Delta^p \times \mathbb{R}^d$ and satisfies (\ref{KLgerm}). 
Therefore, the quotient map $\widetilde{\gamma}$ is surjective.
Next, we can define a simplicial subgroup 
$\mathcal{H}^{N(\mathbf{0})}(\mathbb{R}^N, \mathbb{R}^d)$ of $\mathcal{H}(\mathbb{R}^N, \mathbb{R}^d)$
analogous to $\mathcal{H}^{N(\mathbf{0})}(\mathbb{R}^N)$, and prove
that the quotient map $\widetilde{\gamma}$ is a principal 
$\mathcal{H}^{N(\mathbf{0})}(\mathbb{R}^N, \mathbb{R}^d)$-bundle
by repeating the argument given in step (ii) above. Finally, 
the proof given in Lemma 1.6 of \cite{KL2} works verbatim to 
show that $\mathcal{H}^{N(\mathbf{0})}(\mathbb{R}^N, \mathbb{R}^d)$ 
is contractible. Therefore, the quotient map 
$\widetilde{\gamma}: \mathcal{H}(\mathbb{R}^N, \mathbb{R}^d) \rightarrow \mathrm{PL}(\mathbb{R}^N, \mathbb{R}^d)$
is also a weak homotopy equivalence.
\end{proof}

By the discussion preceding this lemma, we have the following. 

\theoremstyle{plain} \newtheorem{equivalence.grass}[grassm]{Proposition}

\begin{equivalence.grass} \label{equivalence.grass}

The map $\mathcal{G}_N: \mathrm{Gr}_d(\mathbb{R}^N)_{\bullet} \rightarrow \widetilde{\mathrm{Gr}}_d(\mathbb{R}^N)_{\bullet}$ is a weak homotopy equivalence.
\end{equivalence.grass}

Finally, by combining Propositions \ref{equivalence.grass} and \ref{germ.equivalence}, we obtain the result that we promised earlier in this section. 

\theoremstyle{plain} \newtheorem{htpy.type.grass}[grassm]{Corollary}

\begin{htpy.type.grass} \label{htpy.type.grass}
The inclusion $\mathcal{J}_N: \mathrm{Gr}_d(\mathbb{R}^N)_{\bullet} \hookrightarrow \Psi_d^{\circ}(\mathbb{R}^N)_{\bullet}$ is a weak homotopy equivalence. 
\end{htpy.type.grass}

\begin{proof}
Clearly, the germ map $\mathcal{G}_N: \mathrm{Gr}_d(\mathbb{R}^N)_{\bullet} \rightarrow \widetilde{\mathrm{Gr}}_d(\mathbb{R}^N)_{\bullet}$ factors through $\Psi_d^{\circ}(\mathbb{R}^N)_{\bullet}$ as follows: 
\[
\xymatrix{ \mathrm{Gr}_d(\mathbb{R}^N)_{\bullet} \hspace{0.1cm} \ar@{^{(}->}[r]^{\mathcal{J}_N} & 
\Psi_d^{\circ}(\mathbb{R}^N)_{\bullet} \ar[r]^{\hspace{-0.15cm}\mathcal{G}_N} & 
\widetilde{\mathrm{Gr}}_d(\mathbb{R}^N)_{\bullet}.
}
\]
Propositions \ref{germ.equivalence} and \ref{htpy.type.grass} then imply that the inclusion $\mathcal{J}_N$ is a weak homotopy equivalence.    
\end{proof}

\theoremstyle{definition} \newtheorem{germ.pathconn2}[grassm]{Remark}

\begin{germ.pathconn2} \label{germ.pathconn2}
Recall that  we proved in Proposition \ref{germ.pathconn} that 
$\widetilde{\mathrm{Gr}}_d(\mathbb{R}^N)_{\bullet}$ is path-connected. Then, it follows from 
Proposition \ref{equivalence.grass} and Corollary \ref{htpy.type.grass}
that $\mathrm{Gr}_d(\mathbb{R}^N)_{\bullet}$ and $\Psi_d^{\circ}(\mathbb{R}^N)_{\bullet}$
are also path-connected.
\end{germ.pathconn2}

\subsection{Spaces of manifolds with normal data} \label{sec.normal.data}

For the proof of the main theorem, we will use the following enhancement of the quasi-PL space $\Psi_d^{\circ}(\mathbb{R}^N)$, where manifolds come equipped with \textit{normal} data near the origin.  

\theoremstyle{definition} \newtheorem{psi.normal}[grassm]{Definition}

\begin{psi.normal} \label{psi.normal}
$\Psi_d^{\hat{\circ}}(\mathbb{R}^N)$ is the quasi-PL space defined as follows:

\begin{itemize}
\item[$\cdot$] For an arbitrary PL space $P$, the elements of the set 
$\Psi_d^{\hat{\circ}}(\mathbb{R}^N)(P)$ are all possible tuples $(W,U)$, 
with $W$  an element of $\Psi_d^{\circ}(\mathbb{R}^N)(P)$ and $U$ a neighborhood of the zero-section 
$P\times \{\mathbf{0}\}$ in $P\times \mathbb{R}^N$ such that $(U, W\cap U)$ is a piecewise linear 
$(\mathbb{R}^N,\mathbb{R}^d)$-bundle.

\item[$\cdot$]  The function 
$f^*: \Psi_d^{\hat{\circ}}(\mathbb{R}^N)(P) \rightarrow \Psi_d^{\hat{\circ}}(\mathbb{R}^N)(Q)$ 
induced by a piecewise linear map $f: Q \rightarrow P$ is defined by taking pull-backs.   
\end{itemize}
\end{psi.normal}

Consider the forgetful map $\mathcal{F}:  \Psi_d^{\hat{\circ}}(\mathbb{R}^N)_{\bullet} \rightarrow \Psi_d^{\circ}(\mathbb{R}^N)_{\bullet}$ given by $(W,U) \mapsto W$. It is a direct consequence of Theorem \ref{kuiper.lashof2} 
in Appendix \ref{sec.pl.micro}
that this forgetful map is surjective. In fact, 
we can prove the following.

\theoremstyle{plain} \newtheorem{forget.normal}[grassm]{Proposition}

\begin{forget.normal} \label{forget.normal}
The forgetful map $\mathcal{F}:  \Psi_d^{\hat{\circ}}(\mathbb{R}^N)_{\bullet} \rightarrow \Psi_d^{\circ}(\mathbb{R}^N)_{\bullet}$ is a weak homotopy equivalence. Consequently, the inclusion $ \mathrm{Gr}_d(\mathbb{R}^N)_{\bullet} \hookrightarrow \Psi_d^{\hat{\circ}}(\mathbb{R}^N)_{\bullet}$ which takes any $p$-simplex $W$ to $(W, \Delta^p \times \mathbb{R}^N)$ is also a weak homotopy equivalence.  
\end{forget.normal}

\begin{proof}
Consider any commutative diagram of the form
\begin{equation} \label{forget.normal.diag}
\xymatrix{ \partial\Delta^p_{\bullet} \hspace{0.1cm} \ar@{^{(}->}[d] \ar[r] & 
\Psi_d^{\hat{\circ}}(\mathbb{R}^N)_{\bullet} \ar[d]^{\mathcal{F}}  \\
\Delta^p_{\bullet} \ar[r] & \Psi_d^{\circ}(\mathbb{R}^N)_{\bullet}.
}
\end{equation}
Let $W$ be the $p$-simplex classified by the bottom map in this diagram. Since this diagram commutes, there is an open neighborhood $U$ of the zero-section $\partial\Delta^p\times\{\mathbf{0}\}$  in $\partial\Delta^p \times \mathbb{R}^N$
so that $(U, W_{\partial\Delta^p}\cap U)$ is an element of the set $\Psi_d^{\hat{\circ}}(\mathbb{R}^N)(\partial\Delta^p)$.   
Also, we can assume that $W$ is \textit{constant near $\partial\Delta^p$} in the following sense: Fix a piecewise linear collar $c: \partial\Delta^p\times [0,1] \rightarrow \Delta^p$.  
Using $c$, we can define a piecewise linear map   $\widetilde{c}: \Delta^p \rightarrow \Delta^p$ which maps the closure of the complement of $c(\partial\Delta^p\times [0,1])$ onto $\Delta^p$, and which maps any point of the form $c(\lambda, s)$ to $\lambda$.
Then, for any value $s \in [0,1]$, the restriction of the pull-back 
$\widetilde{c}^*W$ over $c(\partial\Delta^p\times \{s\})$ is a translation of the restriction $W_{\partial\Delta^p}$. Equivalently, we can view the restriction of $\widetilde{c}^*W$ over   $c(\partial\Delta^p\times [0,1])$ as a constant concordance from $W_{\partial\Delta^p}$ to itself. Since $W$ and $\widetilde{c}^*W$ are concordant, we can replace $W$ with $\widetilde{c}^*W$ in (\ref{forget.normal.diag}), and relabel $\widetilde{c}^*W$ as $W$.  

Let $\tilde{\Delta}^p$ denote the closure of the complement $\Delta^p - c(\partial\Delta^p\times [0,1])$. By Theorem \ref{kuiper.lashof2}, we can find an open neighborhood $V$ of the section 
$\tilde{\Delta}^p\times \{\mathbf{0}\}$ in $\tilde{\Delta}^p\times\mathbb{R}^N$ such that $(V, W_{\tilde{\Delta}^p}\cap V)$ is a piecewise linear $(\mathbb{R}^N, \mathbb{R}^d)$-bundle over $\tilde{\Delta}^p$. 
Moreover, if we denote the restriction of $V$ over $\partial\tilde{\Delta}^p$ by  $V_{\partial\tilde{\Delta}^p}$, then by Remark \ref{isotopy.bund} we have that the bundle $(U, W_{\partial\Delta^p}\cap U)$
is isotopic to $(V_{\partial\tilde{\Delta}^p}, W_{\partial\tilde{\Delta}^p}\cap V_{\partial\tilde{\Delta}^p})$ once we identify 
$\partial\tilde{\Delta}^p$ with $\partial\Delta^p$.   
Using this isotopy, we can construct a piecewise linear $(\mathbb{R}^N, \mathbb{R}^d)$-bundle over the image $c(\partial\Delta^p\times [0,1])$ which is a concordance between 
$(U, W_{\partial\Delta^p}\cap U)$
and $(V_{\partial\tilde{\Delta}^p}, W_{\partial\tilde{\Delta}^p}\cap V_{\partial\tilde{\Delta}^p})$. 
Since $\Psi_d^{\hat{\circ}}(\mathbb{R}^N)$ is a quasi-PL space, 
we can glue this concordance to $(V, W_{\tilde{\Delta}^p}\cap V)$ 
to obtain a piecewise linear $(\mathbb{R}^N, \mathbb{R}^d)$-bundle $(\widetilde{U}, \widetilde{U}\cap W)$ over $\Delta^p$. Since $(\widetilde{U}, \widetilde{U}\cap W)$ clearly extends $(U, W_{\partial\Delta^p}\cap U)$, 
this bundle induces a lift $\Delta^p_{\bullet} \rightarrow \Psi_d^{\hat{\circ}}(\mathbb{R}^N)_{\bullet}$ in (\ref{forget.normal.diag}), which proves that the forgetful map $\mathcal{F}:  \Psi_d^{\hat{\circ}}(\mathbb{R}^N)_{\bullet} \rightarrow \Psi_d^{\circ}(\mathbb{R}^N)_{\bullet}$ is a weak homotopy equivalence. 

For the second claim in this proposition, just note that the inclusion
$\mathcal{J}_N: \mathrm{Gr}_d(\mathbb{R}^N)_{\bullet} \hookrightarrow \Psi_d^{\circ}(\mathbb{R}^N)_{\bullet}$
is equal to the composition of the inclusion $\mathrm{Gr}_d(\mathbb{R}^N)_{\bullet} \hookrightarrow \Psi_d^{\hat{\circ}}(\mathbb{R}^N)_{\bullet}$ and the forgetful map 
$\mathcal{F}:  \Psi_d^{\hat{\circ}}(\mathbb{R}^N)_{\bullet} \rightarrow \Psi_d^{\circ}(\mathbb{R}^N)_{\bullet}$. 
The result now follows immediately from Corollary \ref{htpy.type.grass}.
\end{proof}

Fix a PL space $P$. Recall that any element $(W,f)$ of the set 
$\widetilde{\Psi}_d(\mathbb{R}^N)(P)$ with $W \neq \varnothing$ 
defines a piecewise linear microbundle pair of the form 
\[
\xymatrix{ V_f  \hspace{0.1cm} \ar@{^{(}->}[r]^{\hspace{-0.8cm}\tilde{f}} & 
(V_f\times\mathbb{R}^N,W) \ar[r]^{\hspace{0.8cm}\pi} & V_f, 
}
\]
where $V_f$ denotes the pre-image $f^{-1}(\mathbb{R}^N)$, and $\tilde{f}: V_f \rightarrow W$ is the piecewise linear embedding defined by $\lambda \mapsto (\lambda, f(\lambda))$. 
As we did for $\Psi_d^{\circ}(\mathbb{R}^N)$, we can also define a variant of $\widetilde{\Psi}_d(\mathbb{R}^N)$ 
where manifolds have normal data near the values of the marking function.  

\theoremstyle{definition} \newtheorem{psi.marking.normal}[grassm]{Definition}

\begin{psi.marking.normal} \label{psi.marking.normal}

$\widetilde{\Psi}_d^{\hat{\circ}}(\mathbb{R}^N)$ will be the quasi-PL space such that, for any PL space $P$, 
the set $\widetilde{\Psi}_d^{\hat{\circ}}(\mathbb{R}^N)(P)$ will consist of all triples $(W,f,U)$ satisfying the following: 

\begin{itemize}
\item[(i)] The pair $(W,f)$ is an element of $\widetilde{\Psi}_d(\mathbb{R}^N)(P)$.
\vspace{0.15cm}

\item[(ii)] If $W = \varnothing$, we require that $U$ is also empty. On the other hand, if $W \neq \varnothing$, then $U$ is an open neighborhood of the image of 
$\tilde{f}$ in $V_f\times \mathbb{R}^N$ . In other words, $U$ is a neighborhood of the graph of $f|_{V_f}$.  

\vspace{0.15cm}

\item[(iii)] Furthermore, if $W\neq \varnothing$, then the diagram
\[
\xymatrix{ V_f  \hspace{0.1cm} \ar@{^{(}->}[r]^{\hspace{-0.8cm}\tilde{f}} & 
(U,W\cap U) \ar[r]^{\hspace{0.6cm}\pi} & V_f
}
\]
must be a piecewise linear $(\mathbb{R}^N,\mathbb{R}^d)$-bundle over $V_f$. 
\end{itemize} 
\end{psi.marking.normal}

As we did for the simplicial set $\Psi_d^{\hat{\circ}}(\mathbb{R}^N)_{\bullet}$, we can also prove that forgetting the normal data in $\widetilde{\Psi}_d^{\hat{\circ}}(\mathbb{R}^N)_{\bullet}$ gives a weak equivalence between 
$\widetilde{\Psi}_d^{\hat{\circ}}(\mathbb{R}^N)_{\bullet}$ and $\widetilde{\Psi}_d(\mathbb{R}^N)_{\bullet}$. 

\theoremstyle{plain} \newtheorem{forget.normal2}[grassm]{Proposition}

\begin{forget.normal2} \label{forget.normal2} 

The forgetful map $\widetilde{\mathcal{F}}: \widetilde{\Psi}_d^{\hat{\circ}}(\mathbb{R}^N)_{\bullet} \rightarrow \widetilde{\Psi}_d(\mathbb{R}^N)_{\bullet}$ defined by $(W,f,U) \mapsto (W,f)$ is a weak homotopy equivalence. 
\end{forget.normal2}

\begin{proof}
We start by observing that we can define the forgetful map $\widetilde{\mathcal{F}}$ at the level of PL sets; i.e., we can
define $\widetilde{\mathcal{F}}$ as the natural transformation 
$\widetilde{\mathcal{F}}: \widetilde{\Psi}_d^{\hat{\circ}}(\mathbb{R}^N) \Rightarrow \widetilde{\Psi}_d(\mathbb{R}^N)$ 
between PL sets such that, for any PL space $P$, the component  
$\widetilde{\mathcal{F}}_P: \widetilde{\Psi}_d^{\hat{\circ}}(\mathbb{R}^N)(P) \rightarrow \widetilde{\Psi}_d(\mathbb{R}^N)(P)$
maps any triple $(W,f,U) \in \widetilde{\Psi}_d^{\hat{\circ}}(\mathbb{R}^N)(P)$ to $(W,f)$. Then, the morphism of simplicial sets
given in the statement of this proposition is the map 
$\widetilde{\mathcal{F}}_{\bullet}: \widetilde{\Psi}_d^{\hat{\circ}}(\mathbb{R}^N)_{\bullet} \rightarrow \widetilde{\Psi}_d(\mathbb{R}^N)_{\bullet}$ 
induced by this natural transformation. 
Henceforth, 
we shall denote the forgetful map between simplicial sets by $\widetilde{\mathcal{F}}_{\bullet}$ and the one 
between PL sets by $\widetilde{\mathcal{F}}$.

Next, note that by Theorem \ref{kuiper.lashof2} we have that each component $\widetilde{\mathcal{F}}_p: \widetilde{\Psi}_d^{\hat{\circ}}(\mathbb{R}^N)_p \rightarrow \widetilde{\Psi}_d(\mathbb{R}^N)_p$
of $\widetilde{\mathcal{F}}_{\bullet}$ is surjective. To prove that $\widetilde{\mathcal{F}}_{\bullet}$
is indeed a weak homotopy equivalence,  
we will use the following PL subsets of 
$\widetilde{\Psi}_d^{\hat{\circ}}(\mathbb{R}^N)$:  

\begin{itemize}
\item[$F_0 =$] PL subset whose value at a PL space $P$ is the subset 
$F_0(P) \subset \widetilde{\Psi}_d^{\hat{\circ}}(\mathbb{R}^N)(P)$ 
of all elements  $(W,f, U)$ such that all fibers of $W$ are non-empty.  

\vspace{0.15cm}

\item[$F_1 =$] PL subset whose value at a PL space $P$ is the subset 
$F_1(P) \subset \widetilde{\Psi}_d^{\hat{\circ}}(\mathbb{R}^N)(P)$ 
of all elements  
$(W,f,U)$ such that each fiber of $W$ is 
disjoint from the origin $\mathbf{0}$ in $\mathbb{R}^N$.   

\vspace{0.15cm}

\item[$F =$] PL subset whose value at a PL space $P$ is
$F_0(P)\cup F_1(P)$. 

\end{itemize}

Moreover, let $S_0$ and $S_1$ be the PL subsets of $\widetilde{\Psi}_d(\mathbb{R}^N)$
obtained by taking the images of $F_0$ and $F_1$ respectively under 
the forgetful map $\widetilde{\mathcal{F}}$.
In other words, $S_0$ (resp. $S_1$) is the PL
subset of  $\widetilde{\Psi}_d(\mathbb{R}^N)$ 
whose value at a PL space $P$ is 
$S_0(P) = \widetilde{\mathcal{F}}_P(F_0(P))$ 
(resp. $S_1(P) = \widetilde{\mathcal{F}}_P(F_1(P))$).
Also, let $S$ be the PL subset of 
$\widetilde{\Psi}_d(\mathbb{R}^N)$ 
such that $S(P) = S_0(P)\cup S_1(P)$ 
for any PL space $P$. 
It is clear that we have that $F_{\bullet} = F_{0\bullet}\cup F_{1\bullet}$ 
and $S_{\bullet} = S_{0\bullet}\cup S_{1\bullet}$,
where $F_{0\bullet}$, $F_{1\bullet}$, $F_{\bullet}$,
$S_{0\bullet}$, $S_{1\bullet}$, and $S_{\bullet}$
are the underlying simplicial sets of
$F_0$, $F_1$, $F$, $S_0$, $S_1$, and $S$ respectively.

Now, consider the following push-out diagrams of simplicial sets: 
\begin{equation} \label{pushouts}
\xymatrix{ F_{0\bullet}\cap F_{1\bullet} \ar[r] \ar[d] & F_{1\bullet} \ar[d] & \quad & S_{0\bullet}\cap S_{1\bullet} \ar[r] \ar[d] & S_{1\bullet} \ar[d] \\
F_{0\bullet} \ar[r] & F_{\bullet} & \quad & S_{0\bullet} \ar[r] & S_{\bullet}.
}
\end{equation}
Note that the conditions used to define the PL subsets $F_0$ and $F_1$ are actually open conditions. 
In other words, if $W$ is an element of $\widetilde{\Psi}_d^{\hat{\circ}}(\mathbb{R}^N)(P)$ and $x$ is a point in $P$ such that the fiber $W_x$ is non-empty, then we can find an open neighborhood $U$ of $x$ such that each fiber of $W_U$ is non-empty.  Similarly, if we assume instead that $W_x$ does not intersect the origin, 
then we can find an open neighborhood $V$ of $x$ such that each fiber of $W_V$ does not intersect the origin either. 
It follows from Proposition \ref{subdivision.cover} that the standard inclusion 
$F_{\bullet} \hookrightarrow \widetilde{\Psi}_d^{\hat{\circ}}(\mathbb{R}^N)_{\bullet}$ is a weak homotopy equivalence. In a similar fashion, we can prove that
the standard inclusion  $S_{\bullet} \hookrightarrow \widetilde{\Psi}_d(\mathbb{R}^N)_{\bullet}$ is also a weak homotopy equivalence. 
Therefore, to finish this proof, it suffices to show that the restriction 
$\widetilde{\mathcal{F}}_{\bullet}|_{F_{\bullet}}: F_{\bullet} \rightarrow S_{\bullet}$ is a weak homotopy equivalence. Moreover, since all the maps in the diagrams 
given in (\ref{pushouts}) are inclusions, it is enough to show that the restrictions
\[
\widetilde{\mathcal{F}}_{\bullet}|_{F_{0\bullet}}: F_{0\bullet} \rightarrow S_{0\bullet} \qquad \widetilde{\mathcal{F}}_{\bullet}|_{F_{1\bullet}}: F_{1\bullet} \rightarrow S_{1\bullet} \qquad 
\widetilde{\mathcal{F}}_{\bullet}|_{F_{0\bullet}\cap F_{1\bullet}}:  F_{0\bullet}\cap F_{1\bullet} \rightarrow S_{0\bullet}\cap S_{1\bullet}
\]   
are weak homotopy equivalences. 
By doing proofs identical to the one we did for Proposition \ref{forget.normal}, we can show that the restrictions 
$\widetilde{\mathcal{F}}_{\bullet}|_{F_{0\bullet}}$ and $\widetilde{\mathcal{F}}_{\bullet}|_{F_{0\bullet}\cap F_{1\bullet}}$ are weak homotopy equivalences. On the other hand, using the scanning techniques from $\S$\ref{section.marking}, we can show that both $F_{1\bullet}$ and $S_{1\bullet}$ are contractible. This automatically implies that 
$\widetilde{\mathcal{F}}_{\bullet}|_{F_{1\bullet}}$ is also a weak homotopy equivalence, which concludes the proof of this proposition.  
\end{proof}

\subsection{Spherical fibrations over spaces of manifolds} 

To show that the map of spectra $\mathcal{I}: \mathbf{MT}PL(d)  \rightarrow \widetilde{\Psi}_d$ introduced in (\ref{map.spectra1}) is a weak equivalence, we will need the following four quasi-PL spaces. 

\theoremstyle{definition} \newtheorem{total.spherical}[grassm]{Definition}

\begin{total.spherical}  \label{total.spherical}

$\mathcal{S}, \mathcal{T}, \mathcal{S}', \mathcal{T}': \mathbf{PL}^{op} \rightarrow \mathbf{Sets}$ are the quasi-PL spaces such that, for any PL space $P$, 
the sets $\mathcal{S}(P)$, $\mathcal{T}(P)$, $\mathcal{S}'(P)$, and $\mathcal{T}'(P)$ are defined as follows: 

\begin{itemize}

\item[$\cdot$] $\mathcal{T}(P)$ is the set of tuples $(W,f)$ where $W \in \mathrm{Gr}_d(\mathbb{R}^N)(P)$ and $f:P\rightarrow \mathbb{R}^N$ is a PL map with the property that, for any $x \in P$, either $f(x) \notin W_x$ or $f(x) = \mathbf{0}$. As usual, $\mathbf{0}$ denotes the origin of $\mathbb{R}^N$. 
On the other hand, $\mathcal{S}(P)$ is the subset of
$\mathcal{T}(P)$ of all tuples $(W,f)$ such that $f(x) \notin W_x$ for all $x \in P$.

\vspace{0.15cm}

\item[$\cdot$] $\mathcal{T}'(P)$ is the set of triples $(W, U, f)$ 
such that $(W,U) \in \Psi_d^{\hat{\circ}}(\mathbb{R}^N)(P)$ and $f:P\rightarrow \mathbb{R}^N$ is a PL map 
satisfying the following two properties: 

\begin{itemize}

\item[-]  The graph of $f$ is contained in $U$.

\item[-] For any $x \in P$, we have either $f(x) \notin W_x$ or $f(x) = \mathbf{0}$.

\end{itemize}

On the other hand, $\mathcal{S}'(P)$ is the subset of $\mathcal{T}'(P)$ 
of all triples $(W, U, f)$ satisfying $f(x) \notin W_x$ for all $x \in P$.
 
\end{itemize}

Structure maps for $\mathcal{S}$, $\mathcal{T}$, $\mathcal{S}'$, and $\mathcal{T}'$ are defined by taking pull-backs.  

\end{total.spherical}

Note that there is a canonical inclusion  $\mathcal{S} \Rightarrow \mathcal{S}'$ 
of PL sets which maps any tuple 
$(W,f) \in \mathcal{S}(P)$ to $(W, P\times \mathbb{R}^N, f) \in \mathcal{S}'(P)$. We can also 
define an inclusion $\mathcal{T} \Rightarrow \mathcal{T}'$ in exactly the same way.  
The induced inclusions $\mathcal{S}_{\bullet} \hookrightarrow \mathcal{S}'_{\bullet}$
and $\mathcal{T}_{\bullet} \hookrightarrow \mathcal{T}'_{\bullet}$
between the corresponding underlying simplicial sets
have the following property.  

\theoremstyle{plain} \newtheorem{spherical.equiv}[grassm]{Proposition}

\begin{spherical.equiv} \label{spherical.equiv}

The inclusions $\mathcal{S}_{\bullet} \hookrightarrow \mathcal{S}'_{\bullet}$ 
and $\mathcal{T}_{\bullet} \hookrightarrow \mathcal{T}'_{\bullet}$
are weak homotopy equivalences. 
\end{spherical.equiv}

\begin{proof}
We will start this proof by showing that the obvious forgetful maps 
\[
F:  \mathcal{S}_{\bullet} \rightarrow \mathrm{Gr}_d(\mathbb{R}^N)_{\bullet} \qquad F': \mathcal{S}'_{\bullet} \rightarrow \Psi_d^{\hat{\circ}}(\mathbb{R}^N)_{\bullet}
\]
\[
G:  \mathcal{T}_{\bullet} \rightarrow \mathrm{Gr}_d(\mathbb{R}^N)_{\bullet} \qquad G': \mathcal{T}'_{\bullet} \rightarrow \Psi_d^{\hat{\circ}}(\mathbb{R}^N)_{\bullet}
\]
are Kan fibrations. We will only do this for the map $G'$, the proof for the other three maps being completely analogous. 
Fix then an $i$ in $\{0, \ldots, p\}$, and consider a commutative diagram of the form 
\begin{equation} \label{diagram.forget.fib}
\xymatrix{ \Lambda^p_{i\bullet} \hspace{0.1cm} \ar@{^{(}->}[d] \ar[r] & 
 \mathcal{T}'_{\bullet} \ar[d]^{G'}  \\
\Delta^p_{\bullet} \ar[r] & \Psi_d^{\hat{\circ}}(\mathbb{R}^N)_{\bullet}.
}
\end{equation}
As we did in the proof of Proposition \ref{germ.equivalence}, we will denote the $i$-th horn of the geometric simplex $\Delta^p$ by $\Lambda^p_i$. This diagram is represented by an element 
$(W,U)$ of $\Psi_d^{\hat{\circ}}(\mathbb{R}^N)(\Delta^p)$ and a piecewise linear function $f: \Lambda^p_i \rightarrow \mathbb{R}^N$ such that the triple 
$(W_{\Lambda^p_i},U_{\Lambda^p_i},f)$ 
is an element of $\mathcal{T}'(\mathbb{R}^N)(\Lambda^p_i)$. 
Since the base-space $\Delta^p$ is contractible, we can find a piecewise linear homeomorphism $H: \Delta^p \times \mathbb{R}^N \rightarrow U$ which commutes with the projection onto $\Delta^p$ and maps the subspace $\Delta^p\times \mathbb{R}^d$ onto $W\cap U$. 
Let $H_{x}$ denote the value of $H$ at a point $x \in \Delta^p$. 
Also, pick a piecewise linear homeomorphism $h: \Lambda^p_i \times [0,1] \rightarrow \Delta^p$ which maps $\Lambda^p_i \times \{ 0\}$ identically to $\Lambda^p_i$. 
Note that  the map $\tilde{f}: \Lambda^p_i \rightarrow \mathbb{R}^N$ defined by $\tilde{f}(x) = H_x^{-1}(f(x))$ can be extended to a map $\tilde{g}$ on $\Delta^p$ by setting $\tilde{g}(h(x,t)) = \tilde{f}(x)$. Then, the map $g: \Delta^p \rightarrow \mathbb{R}^N$ defined by $g(y) = H_y(\tilde{g}(y))$ extends $f$, and its graph is contained in $U$. 
Also, by the way we constructed the map $g$, it is not hard to verify that, 
for any $x \in \Delta^p$, we have either 
$g(x) \notin W_x$ or $g(x) = \mathbf{0}$.
In other words, the triple $(W, U, g)$ defines a lift 
$\Delta^p_{\bullet} \rightarrow \mathcal{T}'_{\bullet}$ in (\ref{diagram.forget.fib}). 

To conclude this proof, note that the canonical inclusions 
$\mathrm{Gr}_d(\mathbb{R}^N)_{\bullet} \hookrightarrow \Psi_d^{\hat{\circ}}(\mathbb{R}^N)_{\bullet}$,
$\mathcal{S}_{\bullet} \hookrightarrow \mathcal{S}'_{\bullet}$, 
$\mathcal{T}_{\bullet} \hookrightarrow \mathcal{T}'_{\bullet}$
and the fibrations $F$, $F'$, $G$, $G'$ fit into the following pull-back diagrams: 
\[
\xymatrix{ \mathcal{S}_{\bullet}  \hspace{0.1cm} \ar@{^{(}->}[r]  \ar[d]^{F}  & 
\mathcal{S}'_{\bullet} \ar[d]^{F'} &   & \mathcal{T}_{\bullet}  \hspace{0.1cm}  \ar@{^{(}->}[r] \ar[d]^{G} & \mathcal{T}'_{\bullet}   \ar[d]^{G'} \\
\mathrm{Gr}_d(\mathbb{R}^N)_{\bullet} \hspace{0.1cm} \ar@{^{(}->}[r] & \Psi_d^{\hat{\circ}}(\mathbb{R}^N)_{\bullet} & & 
 \mathrm{Gr}_d(\mathbb{R}^N)_{\bullet} \hspace{0.1cm} \ar@{^{(}->}[r] & \Psi_d^{\hat{\circ}}(\mathbb{R}^N)_{\bullet}.
}
\] 
Since the bottom inclusion in both diagrams is a weak homotopy equivalence by Proposition \ref{forget.normal}, we immediately have that 
$\mathcal{S}_{\bullet} \hookrightarrow \mathcal{S}'_{\bullet}$ and $\mathcal{T}_{\bullet} \hookrightarrow \mathcal{T}'_{\bullet}$
are also weak homotopy equivalences.  
\end{proof}

\theoremstyle{definition} \newtheorem{remark.sph.equi}[grassm]{Remark}

\begin{remark.sph.equi} \label{remark.sph.equi}
First, observe that Proposition \ref{forget.normal} and Remark \ref{germ.pathconn2}
imply that $\Psi_d^{\hat{\circ}}(\mathbb{R}^N)_{\bullet}$ is path-connected.  Next, 
consider again the Kan fibration $F:  \mathcal{S}_{\bullet} \rightarrow \mathrm{Gr}_d(\mathbb{R}^N)_{\bullet}$ 
defined by $(W,f) \mapsto W$. Denote by $\mathbb{R}^d$ the 0-simplex of 
$\mathrm{Gr}_d(\mathbb{R}^N)_{\bullet}$ induced by the standard inclusion 
$\mathbb{R}^d\hookrightarrow \mathbb{R}^N$ and let $\mathrm{Gr}_d(\mathbb{R}^N)_{\mathbb{R}^d}$ be the subsimplicial set consisting of all degeneracies of $\mathbb{R}^d$. The pre-image 
$F^{-1}(\mathrm{Gr}_d(\mathbb{R}^N)_{\mathbb{R}^d})$ is equal to the simplicial set $\mathcal{H}(x,\mathbb{R}^N \setminus \mathbb{R}^d)_{\bullet}$ where $p$-simplices are PL functions of the form 
$\Delta^p \rightarrow \mathbb{R}^N \setminus \mathbb{R}^d$. As proven in \cite{KL2}, the geometric realization 
$|\mathcal{H}(x,\mathbb{R}^N \setminus\mathbb{R}^d)_{\bullet}|$  is weak homotopy equivalent to 
$\mathbb{R}^N \setminus \mathbb{R}^d$. It follows that
$F:  \mathcal{S}_{\bullet} \rightarrow \mathrm{Gr}_d(\mathbb{R}^N)_{\bullet}$ is a spherical fibration whose fibers are weak homotopy equivalent to $S^{N-d-1}$. Since  $F:  \mathcal{S}_{\bullet} \rightarrow \mathrm{Gr}_d(\mathbb{R}^N)_{\bullet}$ is a pull-back of $F': \mathcal{S}'_{\bullet} \rightarrow \Psi_d^{\hat{\circ}}(\mathbb{R}^N)_{\bullet}$, we also have that $F'$ is a spherical fibration. 
\end{remark.sph.equi}

Recall that $A\mathrm{Gr}^+_d(\mathbb{R}^N)$ is the quasi-PL subspace of $\widetilde{\Psi}_d(\mathbb{R}^N)$ 
such that, for any PL space $P$, the set $A\mathrm{Gr}^+_d(\mathbb{R}^N)(P)$ consists of all tuples
$(W,f)$ with the property that the diagram
\begin{equation}  \label{bundle.diagram22}
\xymatrix{ V_f  \hspace{0.1cm} \ar@{^{(}->}[r]^{\hspace{-0.8cm}\tilde{f}} & 
(V_f\times\mathbb{R}^N,W) \ar[r]^{\hspace{0.75cm}\pi} & V_f
}
\end{equation}
is a piecewise linear $(\mathbb{R}^N,\mathbb{R}^d)$-bundle. 
Again, in the above diagram, $V_f$ denotes the pre-image $f^{-1}(\mathbb{R}^N)$.
Also, recall that the map of spectra $\mathcal{I}: \mathbf{MT}PL(d)  \rightarrow \widetilde{\Psi}_d$ 
was obtained by assembling the canonical inclusions 
$A\mathrm{Gr}^+_d(\mathbb{R}^N)_{\bullet} \hookrightarrow \widetilde{\Psi}_d(\mathbb{R}^N)_{\bullet}$.
Note that, for each $N$, the inclusion 
$A\mathrm{Gr}^+_d(\mathbb{R}^N)_{\bullet} \hookrightarrow \widetilde{\Psi}_d(\mathbb{R}^N)_{\bullet}$
can be expressed as the composition 
\begin{equation} \label{final.composition} 
\xymatrix{ A\mathrm{Gr}^+_d(\mathbb{R}^N)_{\bullet} \hspace{0.1cm} \ar@{^{(}->}[r]^{\hspace{0.1cm}\widetilde{j}} & 
\widetilde{\Psi}_d^{\hat{\circ}}(\mathbb{R}^N)_{\bullet}  \ar[r]^{\hspace{-0.1cm}\widetilde{\mathcal{F}}} &  
\widetilde{\Psi}_d(\mathbb{R}^N)_{\bullet},
}
\end{equation}
where $\widetilde{\mathcal{F}}$ is the forgetful map given by $(W,f,U) \mapsto (W,f)$, and the first map
is the inclusion 
$A\mathrm{Gr}^+_d(\mathbb{R}^N)_{\bullet} \hookrightarrow \widetilde{\Psi}_d^{\hat{\circ}}(\mathbb{R}^N)_{\bullet}$
which sends a tuple $(W,f)$ to $(W,f, V_f\times \mathbb{R}^N)$. 
Recall that we proved in Proposition \ref{forget.normal2}
that the forgetful map $\widetilde{\mathcal{F}}: \widetilde{\Psi}_d^{\hat{\circ}}(\mathbb{R}^N)_{\bullet} \rightarrow \widetilde{\Psi}_d(\mathbb{R}^N)_{\bullet}$
is a weak homotopy equivalence. 
Thus, to establish that 
$\mathcal{I}: \mathbf{MT}PL(d)  \rightarrow \widetilde{\Psi}_d$ 
is a weak equivalence of spectra, we just need to prove the following proposition.

\theoremstyle{plain} \newtheorem{last.step}[grassm]{Proposition}

\begin{last.step} \label{last.step}

The inclusion 
$\widetilde{j}: A\mathrm{Gr}^+_d(\mathbb{R}^N)_{\bullet} \hookrightarrow 
\widetilde{\Psi}_d^{\hat{\circ}}(\mathbb{R}^N)_{\bullet}$ 
defined by 
\[
(W, f)\mapsto (W, f, V_f\times \mathbb{R}^N) 
\]
is a weak homotopy equivalence.  
\end{last.step}

\begin{proof}

This proof is similar to that of Proposition \ref{forget.normal2}. 
First, we define the inclusion $\widetilde{j}$ at the level of PL sets,
i.e., we define $\widetilde{j}: A\mathrm{Gr}^+_d(\mathbb{R}^N) \Rightarrow \widetilde{\Psi}_d^{\hat{\circ}}(\mathbb{R}^N)$
to be the natural transformation of PL sets which 
maps each tuple of the form $(W,f) \in A\mathrm{Gr}^+_d(\mathbb{R}^N)(P)$ 
to $(W, f, V_f\times \mathbb{R}^N) \in \widetilde{\Psi}_d^{\hat{\circ}}(\mathbb{R}^N)(P)$, 
where $V_f$ is the pre-image $f^{-1}(\mathbb{R}^N)$. 
Then, the map given in the statement of this proposition is the map of 
simplicial sets $A\mathrm{Gr}^+_d(\mathbb{R}^N)_{\bullet} \rightarrow 
\widetilde{\Psi}_d^{\hat{\circ}}(\mathbb{R}^N)_{\bullet}$ induced by this morphism $\widetilde{j}$. 
From now on, as we did in the proof of Proposition \ref{forget.normal2},
we shall denote by $\widetilde{j}_{\bullet}$
 the map of simplicial sets $A\mathrm{Gr}^+_d(\mathbb{R}^N)_{\bullet} \rightarrow 
\widetilde{\Psi}_d^{\hat{\circ}}(\mathbb{R}^N)_{\bullet}$ induced by 
$\widetilde{j}: A\mathrm{Gr}^+_d(\mathbb{R}^N) \Rightarrow \widetilde{\Psi}_d^{\hat{\circ}}(\mathbb{R}^N)$.

Now, let $F_0$, $F_1$, and $F$ be the PL subsets of 
$\widetilde{\Psi}_d^{\hat{\circ}}(\mathbb{R}^N)$ whose values at a PL space $P$ are the following:  

\begin{itemize}
\item[$F_0(P) =$] subset of $\widetilde{\Psi}_d^{\hat{\circ}}(\mathbb{R}^N)(P)$
of all triples $(W,f, U)$ such that each fiber of $W$ is non-empty and the product $P \times \{\mathbf{0}\}$ is contained in $U$. 

\item[$F_1(P) =$] subset of $\widetilde{\Psi}_d^{\hat{\circ}}(\mathbb{R}^N)(P)$
of all triples $(W,f,U)$ such that each fiber of $W$ is disjoint from the origin $\mathbf{0}\in \mathbb{R}^N$.  

\item[$F(P) =$] $F_0(P)\cup F_1(P)$.  
\end{itemize}

Additionally, let $V_0$, $V_1$, and $V$ be the PL subsets of 
$A\mathrm{Gr}^+_d(\mathbb{R}^N)$ whose values at a PL space $P$ are the following:  

\begin{itemize}
\item[$V_0(P) =$] subset of $A\mathrm{Gr}^+_d(\mathbb{R}^N)(P)$
of all tuples $(W,f)$ such that each fiber of $W$ is non-empty.  

\item[$V_1(P) =$] subset of $A\mathrm{Gr}^+_d(\mathbb{R}^N)(P)$
of all tuples $(W,f)$ such that each fiber of $W$ is disjoint from the origin $\mathbf{0} \in \mathbb{R}^N$.  

\item[$V(P) =$] $V_0(P)\cup V_1(P)$. 
\end{itemize}

Evidently, the underlying simplicial sets $F_{\bullet}$ and 
$V_{\bullet}$ of $F$ and $V$ are equal to the unions 
$F_{0\bullet}\cup F_{1\bullet}$ and  $V_{0\bullet}\cup V_{1\bullet}$ respectively.
Also, note that $V_{0\bullet}$, $V_{1\bullet}$, and $V_{\bullet}$ are the pre-images of $F_{0\bullet}$, $F_{1\bullet}$, and $F_{\bullet}$ respectively under the  inclusion 
$\widetilde{j}_{\bullet}: A\mathrm{Gr}^+_d(\mathbb{R}^N)_{\bullet} \hookrightarrow \widetilde{\Psi}_d^{\hat{\circ}}(\mathbb{R}^N)_{\bullet}$.  

Now consider the following push-out diagrams of simplicial sets, where all the maps are the obvious inclusions:   
\begin{equation} \label{pushouts2}
\xymatrix{ V_{0\bullet}\cap V_{1\bullet} \ar[r] \ar[d] & V_{1\bullet} \ar[d] & \quad & F_{0\bullet}\cap F_{1\bullet} \ar[r] \ar[d] & F_{1\bullet} \ar[d] \\
V_{0\bullet} \ar[r] & V_{\bullet} & \quad & F_{0\bullet} \ar[r] & F_{\bullet}.
}
\end{equation}
By arguments similar to those used in the proof of Proposition \ref{forget.normal2}, the inclusions
\[
V_{\bullet} \hookrightarrow A\mathrm{Gr}^+_d(\mathbb{R}^N)_{\bullet} \qquad F_{\bullet} 
\hookrightarrow  \widetilde{\Psi}_d^{\hat{\circ}}(\mathbb{R}^N)_{\bullet}
\]  
are both weak homotopy equivalences. 
Then, since all the maps in (\ref{pushouts2}) are inclusions, it suffices to prove that
\[
V_{1\bullet} \hookrightarrow F_{1\bullet} \qquad V_{0\bullet} \hookrightarrow F_{0\bullet} \qquad V_{0\bullet}\cap V_{1\bullet} \hookrightarrow F_{0\bullet}\cap F_{1\bullet}
\]
are all weak equivalences. We discuss each one of these maps separately:

$\underline{V_{1\bullet} \hookrightarrow F_{1\bullet}:}$ Again, since both $V_{1\bullet}$ and $F_{1\bullet}$ are contractible, 
we immediately have that  this inclusion is a weak homotopy equivalence.

$\underline{V_{0\bullet} \hookrightarrow F_{0\bullet}:}$  In this step of the proof (and the next one), 
we shall use the following notation: 
Given any PL function $f: \Delta^p \rightarrow \mathbb{R}^N$, let $H_{-f}: \Delta^p \times \mathbb{R}^N \rightarrow \Delta^p \times \mathbb{R}^N$ be the PL homeomorphism defined by 
$H_{-f}(\lambda, \vec{x}) = (\lambda, \vec{x} - f(\lambda))$. 
For any PL subspace $Q$ of $\Delta^p\times \mathbb{R}^N$, 
we will denote the image $H_{-f}(Q)$ by $Q -f$.

Now, let $\mathcal{T}_{\bullet}$ and $\mathcal{T}'_{\bullet}$
be the underlying simplicial sets of the quasi-PL spaces 
$\mathcal{T}$ and $\mathcal{T}'$ introduced in Definition \ref{total.spherical}, and note that
the inclusions $\mathcal{T}_{\bullet} \hookrightarrow \mathcal{T}'_{\bullet}$ and
$V_{0\bullet} \hookrightarrow F_{0\bullet}$ fit in the commutative diagram
\[
\xymatrix{ V_{0\bullet}  \hspace{0.1cm} \ar@{^{(}->}[r] \ar[d]  & 
 F_{0\bullet}  \ar[d]  \\
\mathcal{T}_{\bullet} \hspace{0.1cm} \ar@{^{(}->}[r] & \mathcal{T}'_{\bullet},
}
\]
where the 
left and right vertical maps are defined respectively by $(W,f) \mapsto (W - f, -f)$ and $(W,f,U)\mapsto (W - f, -f, U - f)$. 
Both of these maps are clearly isomorphisms. Therefore, by Proposition \ref{spherical.equiv}, 
the map $V_{0\bullet} \hookrightarrow F_{0\bullet}$
is a weak homotopy equivalence. 

$\underline{V_{0\bullet}\cap V_{1\bullet} \hookrightarrow F_{0\bullet}\cap F_{1\bullet}:}$  
This step is practically identical to the previous one. First, 
let $\mathcal{S}_{\bullet}$ and $\mathcal{S}'_{\bullet}$
be the underlying simplicial sets of the quasi-PL spaces 
$\mathcal{S}$ and $\mathcal{S}'$ we introduced in Definition \ref{total.spherical}.
We can fit the inclusions 
$\mathcal{S}_{\bullet} \hookrightarrow \mathcal{S}'_{\bullet}$ and 
$V_{0\bullet}\cap V_{1\bullet} \hookrightarrow F_{0\bullet}\cap F_{1\bullet}$ in the commutative diagram
\[
\xymatrix{ V_{0\bullet}\cap V_{1\bullet} \hspace{0.1cm} \ar@{^{(}->}[r] \ar[d]  & 
 F_{0\bullet} \cap F_{1\bullet} \ar[d]  \\
\mathcal{S}_{\bullet} \hspace{0.1cm} \ar@{^{(}->}[r] & \mathcal{S}'_{\bullet},
}
\]
where, once again, 
the left and right vertical maps are defined by $(W,f) \mapsto (W - f, -f)$ and $(W,f,U)\mapsto (W - f, -f, U - f)$ respectively. 
Since both of these vertical maps are isomorphisms, Proposition \ref{spherical.equiv} implies that
$V_{0\bullet}\cap V_{1\bullet} \hookrightarrow F_{0\bullet}\cap F_{1\bullet}$ is also a weak homotopy equivalence. 
\end{proof}

We can now prove the following. 

\theoremstyle{plain} \newtheorem{hatcher}[grassm]{Theorem}

\begin{hatcher} \label{hatcher}
The map of spectra $\mathcal{I}: \mathbf{MT}PL(d)  \rightarrow \widetilde{\Psi}_d$ is a weak equivalence.
\end{hatcher}

\begin{proof} 
As mentioned earlier, 
the map $\mathcal{I}_N: A\mathrm{Gr}^+_d(\mathbb{R}^N)_{\bullet} \rightarrow \widetilde{\Psi}_d(\mathbb{R}^N)_{\bullet}$ between the corresponding $N$-th levels is equal to the composition 
\[
\xymatrix{ A\mathrm{Gr}^+_d(\mathbb{R}^N)_{\bullet} \hspace{0.1cm} \ar@{^{(}->}[r]^{\hspace{0.1cm}\widetilde{j}} & 
\widetilde{\Psi}_d^{\hat{\circ}}(\mathbb{R}^N)_{\bullet}  \ar[r]^{\hspace{-0.1cm}\widetilde{\mathcal{F}}} &  
\widetilde{\Psi}_d(\mathbb{R}^N)_{\bullet}.
}
\]
Then, by Propositions \ref{forget.normal2} and \ref{last.step}, it follows that $\mathcal{I}_N$ 
is a weak homotopy equivalence. 
\end{proof}

By combining this last result with Theorem \ref{marking.equivalence}, we obtain the main theorem of this paper,
which was presented as Theorem \ref{thmb} in the introduction. 

\theoremstyle{plain} \newtheorem{main.theorem.g}[grassm]{Theorem}

\begin{main.theorem.g} \label{main.theorem.g}

There is a weak equivalence $\mathbf{MT}PL(d)  \stackrel{\simeq}{\longrightarrow} \Psi_d$.  

\end{main.theorem.g}

\begin{proof}
Define the map of spectra $\mathbf{MT}PL(d) \rightarrow \Psi_d$ to be the composition 
\[
\xymatrix{ \mathbf{MT}PL(d)  \ar[r]^{\hspace{0.5cm}\mathcal{I}}  & \widetilde{\Psi}_ d \ar[r]^{\mathcal{F}} & \Psi_d,
}
\]
where $\mathcal{F}$ is the map defined in (\ref{forgetful.spectra.map}). 
By Theorems \ref{hatcher} and \ref{marking.equivalence}, this composition  is a weak equivalence.
\end{proof}

\subsection{The homotopy type of the Madsen-Tillmann spectrum}

In this subsection, we justify why it is appropriate to call our spectrum $\mathbf{MT}PL(d)$ the \textit{PL Madsen-Tillmann spectrum}. Let 
$\mathrm{Gr}_d^{\mathrm{Diff}}(\mathbb{R}^N)$ be the 
\textit{smooth Grassmannian} of $d$-planes in $\mathbb{R}^N$. In other words,  
$\mathrm{Gr}_d^{\mathrm{Diff}}(\mathbb{R}^N)$ is the set of all $d$-dimensional linear subspaces of $\mathbb{R}^N$
(see \cite{MS} for a discussion of the topology of this space).
Also, let $S\mathrm{Gr}_d^{\mathrm{Diff}}(\mathbb{R}^N)$ be the space of tuples $(x,P)$ with $P \in \mathrm{Gr}_d^{\mathrm{Diff}}(\mathbb{R}^N)$ and $x \in \mathbb{R}^N \setminus P$. Finally, recall that the $N$-th level of the Madsen-Tillmann spectrum $\mathbf{MT}O(d)$ is the Thom space $\mathrm{Th}(\gamma^{\perp}_{d,N})$, where $\gamma^{\perp}_{d,N}$ is the standard orthogonal vector bundle over $\mathrm{Gr}_d^{\mathrm{Diff}}(\mathbb{R}^N)$. One can easily show that $\mathrm{Th}(\gamma^{\perp}_{d,N})$   is weak homotopy equivalent to the mapping cone of the forgetful map 
$S\mathrm{Gr}_d^{\mathrm{Diff}}(\mathbb{R}^N)\rightarrow \mathrm{Gr}_d^{\mathrm{Diff}}(\mathbb{R}^N)$, which is a spherical fibration over $\mathrm{Gr}_d^{\mathrm{Diff}}(\mathbb{R}^N)$. In the next proposition, we give a similar result for the levels of $\mathbf{MT}PL(d)$.   

\theoremstyle{plain} \newtheorem{levels.mtpld}[grassm]{Proposition}

\begin{levels.mtpld} \label{levels.mtpld}

Let $\mathcal{S}: \mathbf{PL}^{op} \rightarrow \mathbf{Sets}$ be the quasi-PL space introduced in Definition \ref{total.spherical} and let 
$F: \mathcal{S}_{\bullet} \rightarrow \mathrm{Gr}_d(\mathbb{R}^N)_{\bullet}$ be the obvious forgetful map. Then, the simplicial set $A\mathrm{Gr}^+_d(\mathbb{R}^N)_{\bullet}$ has the same weak homotopy type as the mapping cone of $F$. 
\end{levels.mtpld}

\begin{proof} 

We will use again the following two PL subsets of $A\mathrm{Gr}^+_d(\mathbb{R}^N)$ that we used in the proof of Proposition \ref{last.step}: 

\begin{itemize}
\item[$V_0 =$] PL subset whose value at a PL space $P$ is the subset $V_0(P) \subset A\mathrm{Gr}^+_d(\mathbb{R}^N)(P)$ 
of all tuples $(W,f)$ such that each fiber of $W$ is non-empty.  

\item[$V_1 =$]  PL subset whose value at a PL space $P$ is the subset $V_1(P) \subset A\mathrm{Gr}^+_d(\mathbb{R}^N)(P)$ 
of all tuples $(W,f)$ such that each fiber of $W$ is disjoint from the origin $\mathbf{0} \in \mathbb{R}^N$. 
\end{itemize}

Once again, let us denote the underlying simplicial sets of $V_0$ and $V_1$ by $V_{0\bullet}$ and $V_{1\bullet}$ respectively. 
As explained in the proof of Proposition \ref{last.step}, the inclusion 
$V_{0\bullet}\cup V_{1\bullet} \hookrightarrow A\mathrm{Gr}^+_d(\mathbb{R}^N)_{\bullet}$ is a weak homotopy equivalence. Moreover, if $\mathcal{V}$ denotes the diagram
\begin{equation} \label{d.one}
\xymatrix{
V_{1\bullet} & \ar@{_{(}->}[l]  \hspace{0.1cm} V_{0\bullet}\cap V_ {1\bullet} \hspace{0.1cm} \ar@{^{(}->}[r] & V_{0\bullet},
}
\end{equation}
then we also have that the natural map $\mathrm{hocolim}(\mathcal{V}) \rightarrow V_{0\bullet}\cup V_{1\bullet}$ 
is a weak homotopy equivalence.   
Therefore, it suffices to show that $\mathrm{hocolim}(\mathcal{V})$ and 
the mapping cone of $F: \mathcal{S}_{\bullet} \rightarrow \mathrm{Gr}_d(\mathbb{R}^N)_{\bullet}$, which we will denote by $\mathcal{C}(F)$, 
are weak homotopy equivalent. 
Recall that $\mathcal{C}(F)$ is defined as the push-out of
\begin{equation} \label{d.two}
\xymatrix{
\mathcal{C}(\mathcal{S}_{\bullet}) & \ar@{_{(}->}[l]  \hspace{0.1cm} \mathcal{S}_{\bullet} \hspace{0.1cm} \ar@{^{(}->}[r] & \mathcal{M}(F).
}
\end{equation}
In this diagram, $\mathcal{M}(F)$ and $\mathcal{C}(\mathcal{S}_{\bullet})$ denote the mapping cylinder of 
$F$ and the cone of $\mathcal{S}_{\bullet}$ respectively. 
The right-hand map is the inclusion of $\mathcal{S}_{\bullet}$ into the \textit{top face} of $\mathcal{M}(F)$, and the left-hand map is the inclusion into the \textit{bottom face} of the cone $\mathcal{C}(\mathcal{S}_{\bullet})$. 

Consider now the diagram 
\begin{equation} \label{d.three}
\xymatrix{
V_{1\bullet} \ar[d] & \ar@{_{(}->}[l]  \hspace{0.1cm} V_{0\bullet}\cap V_{1\bullet} \hspace{0.1cm} \ar@{^{(}->}[r] \ar[d] & V_{0\bullet} \ar[d] \\
\mathcal{C}(\mathcal{S}_{\bullet}) & \ar@{_{(}->}[l]  \hspace{0.1cm} \mathcal{S}_{\bullet} \hspace{0.1cm} \ar@{^{(}->}[r] & \mathcal{M}(F),
}
\end{equation}
where the top row is $\mathcal{V}$, the bottom row is (\ref{d.two}), the left-vertical map is the  constant map which sends all simplices to the base-point of 
$\mathcal{C}(\mathcal{S}_{\bullet})$, the middle map is defined by $(W,f) \mapsto (W - f, -f)$, and the right-vertical map is the composition of the map
$V_{0\bullet} \rightarrow \mathrm{Gr}_d(\mathbb{R}^N)_{\bullet}$ defined by $(W,f) \mapsto W - f$ and the canonical inclusion 
$\mathrm{Gr}_d(\mathbb{R}^N)_{\bullet} \hookrightarrow \mathcal{M}(F)$. For the definition of the middle and right-vertical maps, 
we are using the same notation we introduced in the proof of Proposition \ref{last.step}.
Before we continue, let us discuss why each vertical map in (\ref{d.three}) is a weak homotopy equivalence. 
This is obvious for the left-vertical map, since both its domain and target are contractible. This is also evident for the middle-vertical map, since this
map is an isomorphism of simplicial sets (a fact that we also pointed out in the proof of Proposition \ref{last.step}).  
To see that the right-vertical map is also a weak homotopy equivalence, consider again the 
PL set  $\mathcal{T}$ we introduced in Definition \ref{total.spherical}, and let $\mathcal{T}_{\bullet}$ be its underlying simplicial set.  
The canonical inclusion $\mathrm{Gr}_d(\mathbb{R}^N)_{\bullet} \hookrightarrow \mathcal{M}(F)$ 
is evidently a weak homotopy equivalence. Thus, to prove that the right-vertical map in (\ref{d.three}) is a weak homotopy equivalence, 
we just need to show that the map $V_{0\bullet} \rightarrow \mathrm{Gr}_d(\mathbb{R}^N)_{\bullet}$ 
given by $(W,f) \mapsto W - f$ is also a weak homotopy equivalence. 
We can express this map 
$V_{0\bullet} \rightarrow \mathrm{Gr}_d(\mathbb{R}^N)_{\bullet}$ as the composition
\begin{equation} \label{d.four}
V_{0\bullet} \longrightarrow \mathcal{T}_{\bullet} \stackrel{G}{\longrightarrow} \mathrm{Gr}_d(\mathbb{R}^N)_{\bullet},
\end{equation}
where the left-hand map is defined by $(W,f) \mapsto (W-f, -f)$ and $G$ is the obvious forgetful map. As pointed out
in the proof of Proposition \ref{last.step}, the left-hand map in (\ref{d.four}) is an isomorphism of simplicial sets. 
On the other hand, we proved in Proposition \ref{spherical.equiv} that the map $G: \mathcal{T}_{\bullet} \rightarrow \mathrm{Gr}_d(\mathbb{R}^N)_{\bullet}$
is a Kan fibration. Moreover, using the Simplicial Approximation Theorem, we can show that the fibers of $G$ are contractible. 
Therefore, $G$ is a weak homotopy equivalence, and it follows that the right-vertical map in (\ref{d.three}) is also a weak homotopy equivalence. 

Now, diagram (\ref{d.three}) is not commutative, but the right-hand and left-hand squares are commutative up to homotopy. This allows us to define a map 
$\mathrm{hocolim}(\mathcal{V}) \rightarrow \mathcal{C}(\mathcal{F})$. Since all the vertical maps in (\ref{d.three}) are weak homotopy equivalences, it follows that 
$\mathrm{hocolim}(\mathcal{V}) \rightarrow \mathcal{C}(\mathcal{F})$ is also a weak homotopy equivalence, and we can therefore conclude that
$A\mathrm{Gr}^+_d(\mathbb{R}^N)_{\bullet}$ has the same weak homotopy type as $\mathcal{C}(\mathcal{F})$. 
\end{proof}

\subsection{Spaces of topological manifolds}

 We close this section by pointing out that it is possible to adapt the
 proofs of this article to prove an analogue of Theorem \ref{main.theorem.g} 
 for spaces of topological manifolds. To do this,  
one needs to work with the simplicial sets $\Psi_d^{\mathrm{Top}}(\mathbb{R}^N)_{\bullet}$ 
introduced in \cite{GLK}. 
A $p$-simplex of $\Psi_d^{\mathrm{Top}}(\mathbb{R}^N)_{\bullet}$ is a closed subspace $W$ of $\Delta^p\times \mathbb{R}^N$ with the property that the standard projection $\pi: \Delta^p\times \mathbb{R}^N \rightarrow \Delta^p$ and the restriction $\pi|_W$ form a relative submersion in the sense of Definition \ref{pl.sub.relative}. The relative dimensions of the pair $(\pi,\pi|_W)$ are $N$ and $d$. In particular, we can view any $p$-simplex of $\Psi_d^{\mathrm{Top}}(\mathbb{R}^N)_{\bullet}$ as a family of \textit{locally flat} topological $d$-dimensional submanifolds of $\mathbb{R}^N$, closed as subspaces, parameterized by $\Delta^p$. 
As also explained in \cite{GLK}, $\Psi_d^{\mathrm{Top}}(\mathbb{R}^N)_{\bullet}$ can be extended to a \textit{quasi-topological space} 
$\Psi_d^{\mathrm{Top}}(\mathbb{R}^N): \mathbf{Top}^{op}\rightarrow \mathbf{Sets}$, 
in the sense of \cite{Gr}. 

To define the topological analogue of $\widetilde{\Psi}_d(\mathbb{R}^N)$, one needs to use \textit{continuous} marking functions. That is, continuous functions which take values in $\mathbb{R}^N \cup\{\infty\}$ and satisfy conditions (i), (ii), and (iii) from Definition \ref{marking}.  Then, to prove Theorem \ref{main.theorem.g} in the topological category, one just needs to repeat the argument given in this paper, replacing all piecewise linear definitions and results with their topological counterparts. 
In particular, we need to make the following changes: 

\begin{itemize}
\item[$\bullet$] Replace piecewise linear bundles and microbundles with \textit{topological} bundles and microbundles respectively. 

\item[$\bullet$] Use the topological version of the Isotopy Extension Theorem whenever we use the piecewise linear version. 

\item[$\bullet$] Apply \textit{the Kister-Mazur Theorem} instead of the theorem of Kuiper and Lashof in each step where the latter is used. 
\end{itemize}

By making these adjustments and following the same argument we gave here, we obtain an alternate proof for the following result proven in \cite{GLK}. 

\theoremstyle{plain} \newtheorem{kupers}[grassm]{Theorem}

\begin{kupers} \label{kupers}

There is a weak equivalence of spectra
\[
\Psi^{\mathrm{Top}} \simeq \mathbf{MT}Top(d),
\]
where $\mathbf{MT}Top(d)$ denotes the topological Madsen-Tillmann spectrum. 
\end{kupers}

The $N$-th level of $\mathbf{MT}Top(d)$ 
is the mapping cone of the forgetful map 
$F: \mathcal{S}^{\mathrm{Top}}_{\bullet} \rightarrow \mathrm{Gr}_d^{\mathrm{Top}}(\mathbb{R}^N)_{\bullet}$, where
$\mathrm{Gr}_d^{\mathrm{Top}}(\mathbb{R}^N)_{\bullet}$ is the \textit{topological Grassmannian} and 
$\mathcal{S}^{\mathrm{Top}}_{\bullet}$ is the topological analogue of the underlying simplicial set $\mathcal{S}_{\bullet}$
of the quasi-PL space $\mathcal{S}$ introduced in Definition \ref{total.spherical}; i.e., 
a $p$-simplex of $\mathrm{Gr}_d^{\mathrm{Top}}(\mathbb{R}^N)_{\bullet}$  is an element 
$W \in \Psi_d^{\mathrm{Top}}(\mathbb{R}^N)_p$ containing the zero-section 
$\Delta^p \times \{ \mathbf{0}\}$ such that the diagram
\[
\xymatrix{
\Delta^p \hspace{0.1cm} \ar@{^{(}->}[r]^{\hspace{-1cm}s_{\mathbf{0}}} & 
(\Delta^p\times\mathbb{R}^N,W) \ar[r]^{\hspace{0.7cm}\pi} & \Delta^p
}
\]
is a topological $(\mathbb{R}^N, \mathbb{R}^d)$-bundle, and a $p$-simplex
of $\mathcal{S}^{\mathrm{Top}}_{\bullet}$ is a pair $(W, f)$ where 
$W \in \mathrm{Gr}_d^{\mathrm{Top}}(\mathbb{R}^N)_p$ and 
$f: \Delta^p \rightarrow \mathbb{R}^N$ is a continuous function 
with the property that $f(x) \notin W_x$ for all $x \in \Delta^p$ 
(see also \S 7.2 of \cite{GLK}).

\appendix 

\section{Basic notions from piecewise linear topology}  \label{appendix.pl} 
In this appendix, we will collect some of the basic definitions and results from piecewise linear topology that we used in this paper. 

\subsection{PL submersions} \label{sec.pl.sub} In the definition of the space of manifolds $\Psi_d(U)$, we used the following notion. 

\theoremstyle{definition}  \newtheorem{pl.sub}{Definition}[section]

\begin{pl.sub} \label{pl.sub}
Let $f:P \rightarrow Q$ be a piecewise linear map. We say that $f$ is a \textit{piecewise linear submersion of  codimension $d$} if, for any $x \in P$, we can find an open neighborhood $V \subseteq Q$ of $f(x)$ and a piecewise linear map $h: V \times \mathbb{R}^d \rightarrow P$ satisfying the following:  

\begin{itemize}
\item[(i)] $h$ is an open piecewise linear embedding and 
the image of $h$ is contained in $f^{-1}(V)$.

\item[(ii)] $h(f(x), \mathbf{0}) = x$, where $\mathbf{0}$ is the origin in $\mathbb{R}^d$. 

\item[(iii)] $h$ makes the following diagram commutative:

\[
 \xymatrix{ V \times \mathbb{R}^d \ar[rr]^{h} \ar[rd] & & f^{-1}(V) \ar[dl]^{f} \\
  & V & 
 }
\]

where the left diagonal map is the standard projection onto $V$.  
\end{itemize}

A map $h: V \times \mathbb{R}^d \rightarrow P$ satisfying conditions (i), (ii) and (iii) is called a \textit{piecewise linear submersion chart around x.} It follows immediately from this definition that the image of any PL submersion is open. 
\end{pl.sub}

In this paper, we also worked with pairs of submersions which are compatible in the following sense. 

\theoremstyle{definition} \newtheorem{pl.sub.relative}[pl.sub]{Definition}

\begin{pl.sub.relative} \label{pl.sub.relative}
(see \cite{Stern}) Let $P$ and $Q$ be PL spaces, and let $P'$ be a PL subspace of $P$. 
A piecewise linear map of pairs $f: (P,P') \rightarrow (Q,Q)$ is a 
\textit{relative PL submersion of codimensions N and d} if it satisfies the following:

\begin{itemize}
\item[(i)] $f: P \rightarrow Q$ is a PL submersion of codimension $N$. 
 
\item[(ii)]  For each $x$ in $P'$, there is an open neighborhood $V$ of $f(x)$ in $Q$ and an open piecewise linear embedding $h:V \times (\mathbb{R}^N, \mathbb{R}^d) \rightarrow (P,P')$ such that $h(f(x), \mathbf{0}) = x$ 
and the composition $f\circ h$ is the standard projection onto $V$. 
\end{itemize}  
\end{pl.sub.relative} 

One important property of PL submersions used in Section \S \ref{section.marking} is the following. 

\theoremstyle{plain} \newtheorem{prop.microfib}[pl.sub]{Proposition}

\begin{prop.microfib} \label{prop.microfib}
Any PL submersion $f: W \rightarrow Q$ of codimension $d$  is a microfibration. 
Moreover, any microlift with respect to $f: W \rightarrow Q$ can be assumed to be piecewise linear. 
\end{prop.microfib}

Recall that a microfibration is a map $f:W \rightarrow Q$ which satisfies the \textit{microlifting property}. That is, given any compact PL space $P$ and a commutative diagram of the form
\[
\xymatrix{ P \hspace{0.1cm} \ar@{^{(}->}[d] \ar[r] & W \ar[d]^{f} \\
P\times[0,1] \ar[r] & Q,
}
\]
we can find a positive value $0 < \epsilon < 1$ and a map $\tilde{g}:P\times [0,\epsilon]\rightarrow W$ 
which lifts the restriction of the bottom map on $P\times [0,\epsilon]$. We illustrate this in the following diagram: 
\[
\xymatrix{ & P \hspace{0.1cm} \ar@{^{(}->}[d] \ar[r] & W \ar[d]^{f} \\
P\times[0,\epsilon] \ar[r] \ar@{-->}[urr]^{\tilde{g}}  & P\times[0,1] \ar[r] & Q.
}
\]
The method to prove Proposition \ref{prop.microfib} is similar to the one used in \cite{Hatcher} to show that a fiber bundle has 
the homotopy lifting property (see Proposition 4.48 in \cite{Hatcher}). 
Moreover, by using the Simplicial Approximation Theorem, 
the microlift $\tilde{g}:P\times [0,\epsilon]\rightarrow W$ 
can be assumed to be piecewise linear if $f:W \rightarrow Q$ is a PL submersion. 

\theoremstyle{definition} \newtheorem{prop.microfib.rem}[pl.sub]{Remark}

\begin{prop.microfib.rem} \label{prop.microfib.rem}
The second statement in Proposition \ref{prop.microfib} can actually be made more general. Namely, 
if $f: W \rightarrow Q$ is a PL submersion and $g: Q' \rightarrow Q$ is a PL map for which there exists
a lift $\tilde{g}: Q' \rightarrow W$ of the form
\[
\xymatrix{ & W \ar[d]^{f} \\
Q' \ar[r]^{g} \ar@{-->}[ur]^{\tilde{g}} & Q, 
}
\]
then we can guarantee that $\tilde{g}$ is also PL via the 
Simplicial Approximation Theorem. 
\end{prop.microfib.rem}

\subsection{PL microbundles} \label{sec.pl.micro}

In the piecewise linear category, there are two analogues of the notion of smooth vector bundle. 
We introduce both of these in the following definition. 

\theoremstyle{definition} \newtheorem{pl.microbundle}[pl.sub]{Definition}

\begin{pl.microbundle} \label{pl.microbundle}

A diagram of PL maps
\begin{equation} \label{microbundle}
\xymatrix{ \xi^d: P \hspace{0.1cm} \ar@{^{(}->}[r]^{\iota} & E(\xi) \ar[r]^{\pi} & P
}
\end{equation}
is a \textit{d-dimensional PL microbundle} if it satisfies the following: 

\begin{itemize}
\item[(i)] The image of $\iota$ is closed in $E(\xi)$ and $\pi\circ\iota = \mathrm{Id}_P$. 

\item[(ii)] For any $x\in P$, we can find open neighborhoods $U \subseteq P$ and $V \subseteq E(\xi)$ of $x$ and $\iota(x)$ respectively for which there exists a PL homeomorphism 
$h: U\times \mathbb{R}^d \rightarrow V$ which makes the following diagram commute: 
\begin{equation} \label{diag.microbundle}
\xymatrix{ & U\times\mathbb{R}^d \ar[dd]^{h} \ar[dr]^{p_2} & \\
U \hspace{0.1cm} \ar@{^{(}->}[ur]^{i_{\mathbf{0}}} \ar@{^{(}->}[dr]^{\iota|_U} & & U\\
& V \ar[ur]^{\pi|_{V}} &
}
\end{equation}
In this diagram, $i_{\mathbf{0}}: U \hookrightarrow U\times\mathbb{R}^d$ and $p_2: U\times\mathbb{R}^d \rightarrow U$ are defined respectively by $i_{\mathbf{0}}(x) = (x, \mathbf{0})$ and $p_2(x,y) = x$. As usual, $\mathbf{0}$ denotes the origin of $\mathbb{R}^d$. 
\end{itemize}

A map $h: U\times \mathbb{R}^d \rightarrow V$ which makes (\ref{diag.microbundle}) commute is called a \textit{microbundle chart}. We say that $\xi$ is a \textit{PL $\mathbb{R}^d$-bundle} if in (ii) 
we can take $V = \pi^{-1}(U)$. 
\end{pl.microbundle}

Two PL microbundles $\xi: P \stackrel{\iota_{\xi}}{\longrightarrow} E(\xi) \stackrel{\pi_{\xi}}{\longrightarrow} P$
and $\xi': P \stackrel{\iota_{\xi'}}{\longrightarrow} E(\xi') \stackrel{\pi_{\xi'}}{\longrightarrow}P$ of dimension $d$
over $P$ are said to be \textit{isomorphic} if we can find open neighborhoods $V_0 \subseteq E(\xi)$ and $V_1 \subseteq E(\xi')$ of $\iota_{\xi}(P)$ and $\iota_{\xi'}(P)$ respectively for which there is a  piecewise linear homeomorphism $h: V_0 \rightarrow V_1$ commuting with all the relevant maps. The following fundamental theorem of N. Kuiper and R. Lashof (see \cite{KL1}) says that any $d$-dimensional PL microbundle is isomorphic to a PL $\mathbb{R}^d$-bundle. 
 
 \theoremstyle{plain} \newtheorem{kuiper.lashof}[pl.sub]{Theorem}
 
\begin{kuiper.lashof} \label{kuiper.lashof}
Given any PL microbundle $\xi^d: P \stackrel{\iota}{\longrightarrow} E(\xi) \stackrel{\pi}{\longrightarrow}P$, there is a neighborhood $V$ of $\iota(P)$ in $E(\xi)$ such that 
$\eta: P \stackrel{\iota}{\longrightarrow} V \stackrel{\pi|_V}{\longrightarrow}P$ is a PL $\mathbb{R}^d$-bundle. 
Moreover, any two 
PL $\mathbb{R}^d$-bundles $\eta_0$ and $\eta_1$ contained in $\xi$ are isomorphic, i.e., if $V_0 \subseteq E(\xi)$ and $V_1 \subseteq E(\xi)$ are the total
spaces of $\eta_0$ and $\eta_1$ respectively, then there exists a PL homeomorphism $h: V_0 \rightarrow V_1$ which commutes with all the structure maps. 
\end{kuiper.lashof}

In this article, we also found it necessary to work with pairs of microbundles. More precisely, a \textit{PL (n,d)-microbundle pair} is a diagram
of pairs
\begin{equation} \label{diag.micro.pair}
\xymatrix{ (\xi^n, \eta^d): (P,P) \hspace{0.1cm} \ar@{^{(}->}[r]^{\hspace{0.5cm}\iota} & (E,E') \ar[r]^{\pi} & (P,P) 
}
\end{equation}
such that
 $\xi^n: P \stackrel{\iota}{\longrightarrow} E \stackrel{\pi}{\longrightarrow} P$
and $\eta^d: P \stackrel{\iota}{\longrightarrow} E' \stackrel{\pi|_{E'}}{\longrightarrow}P$  are PL microbundles over the PL space $P$ and, for each $x$ in $P$, there is a microbundle chart  $h: U\times \mathbb{R}^n \rightarrow V$ for $\xi^n$ whose restriction on $U\times \mathbb{R}^d$ is a microbundle chart for $\eta^d$. 
Similarly, we say that $(\xi^n, \eta^d)$ is a \textit{PL $(\mathbb{R}^n, \mathbb{R}^d)$-bundle} if for the chart $h: U\times \mathbb{R}^n \rightarrow V$ we can take $V = \pi^{-1}(U)$. Typically, in a diagram of the form (\ref{diag.micro.pair}), we shall write $P$ instead of 
$(P,P)$. 

The analogue of Theorem \ref{kuiper.lashof} for microbundle pairs is also proven in \cite{KL1}. 

\theoremstyle{plain} \newtheorem{kuiper.lashof2}[pl.sub]{Theorem}

\begin{kuiper.lashof2} \label{kuiper.lashof2}
Given a PL $(n,d)$-microbundle pair  $(\xi, \xi'): P \stackrel{\iota}{\longrightarrow} (E,E') \stackrel{\pi}{\longrightarrow} P$, there is a neighborhood $V$ of $\iota(P)$ in $E$ such that $(\eta, \eta'): P \stackrel{\iota}{\longrightarrow} (V,E'\cap V) \stackrel{\pi|_V}{\longrightarrow}P$ is a PL 
$(\mathbb{R}^n, \mathbb{R}^d)$-bundle.
Also, any two PL $(\mathbb{R}^n, \mathbb{R}^d)$-bundles $(\eta_0, \eta_0')$ and $(\eta_1, \eta_1')$ contained in $(\xi,\xi')$ are isomorphic, i.e., 
if $V_0 \subseteq E(\xi)$ and $V_1 \subseteq E(\xi)$ are the total
spaces of $\eta_0$ and $\eta_1$ respectively, then there exists a PL homeomorphism $h: V_0 \rightarrow V_1$ which commutes with all the structure maps
and which maps $E'\cap V_0$ to $E'\cap V_1$.  
\end{kuiper.lashof2}

\theoremstyle{definition} \newtheorem{isotopy.bund}[pl.sub]{Remark}

\begin{isotopy.bund} \label{isotopy.bund}
Let $\mathcal{H}(\mathbb{R}^d)$ and $\mathrm{PL}(\mathbb{R}^d)$ be the simplicial sets introduced in Definitions \ref{automorphisms} and \ref{germ.automorphisms} respectively. Also, let $\mathcal{E}_{\mathbb{R}^d}(\mathbb{R}^d)$ be the simplicial set of bundle monomorphisms $\Delta^p \times \mathbb{R}^d \rightarrow \Delta^p \times \mathbb{R}^d$
which preserve the zero-section (see Convention \ref{kl.convention}). Note that the quotient maps 
$\mathcal{H}(\mathbb{R}^d) \rightarrow \mathrm{PL}(\mathbb{R}^d)$ and $\mathcal{E}_{\mathbb{R}^d}(\mathbb{R}^d) \rightarrow \mathrm{PL}(\mathbb{R}^d)$  fit in a commutative diagram of the following form:
\begin{equation}
\xymatrix{ \mathcal{H}(\mathbb{R}^d) \ar@{^{(}->}[dd] \ar[dr]^{\simeq} & \\
 & \mathrm{PL}(\mathbb{R}^d) \\
 \mathcal{E}_{\mathbb{R}^d}(\mathbb{R}^d). \ar[ur] & 
}
\end{equation}
Using standard scaling arguments, we can prove that the quotient map  $\mathcal{E}_{\mathbb{R}^d}(\mathbb{R}^d) \rightarrow \mathrm{PL}(\mathbb{R}^d)$ is a weak homotopy equivalence. Therefore, by Lemma \ref{KuiperLemma}, we also have that the inclusion $\mathcal{H}(\mathbb{R}^d) \hookrightarrow \mathcal{E}_{\mathbb{R}^d}(\mathbb{R}^d)$ is a weak homotopy equivalence. Using this fact, one can prove that any two $\mathbb{R}^d$-bundles $\eta_0$ and $\eta_1$ contained in a  PL microbundle 
$\xi^d: P \stackrel{\iota}{\longrightarrow} E(\xi) \stackrel{\pi}{\longrightarrow}$ are \textit{isotopic}. That is, if $V_0, V_1 \subseteq E(\xi)$ are the total spaces of $\eta_0$ and $\eta_1$ respectively, then we can find an open piecewise linear embedding $H: V_0\times [0,1] \rightarrow E(\xi)\times [0,1]$ which commutes with the projection onto $P \times [0,1]$, preserves the zero-section,  maps $V_0$ identically to itself at time $t = 0$, and
maps   $V_0$ PL homeomorphically onto $V_1$ at time $t=1$. A similar discussion also works to show that any two 
PL $(\mathbb{R}^N,\mathbb{R}^d)$-bundles contained in a microbundle pair 
$(\xi, \xi')$ are isotopic.  
\end{isotopy.bund}

\subsection{Regular neighborhoods} \label{sec.regular}

In this appendix,  we collect the essential facts about regular neighborhoods that we used in this paper. To state the definition of regular neighborhood, we first need to introduce some terminology. Throughout this appendix, $I^n$ will denote the $n$-dimensional PL ball $[0,1]^n$. Also, in the next definition, we shall identify the 
PL ball $I^{n-1}$ with the face $[0,1]^{n-1}\times \{0\}$ of $I^n$, and thus regard $I^{n-1}$ as a PL subspace of $I^n$ contained in the boundary $\partial I^n$.

\theoremstyle{definition} \newtheorem{elem.collapse}[pl.sub]{Definition}

\begin{elem.collapse} \label{elem.collapse}
Fix a PL space $X$. If $Y$ is a PL subspace of $X$ such that the pair 
$(\mathrm{cl}(X -  Y), \mathrm{cl}(X -  Y)\cap Y)$ is PL homeomorphic to the pair 
$(I^n, I^{n-1})$ for some $n$, then we say that there is an \textit{elementary collapse} from $X$ to $Y$,
and we write  $X  \searrow_e Y$. More generally, we say that $X$ \textit{collapses to} $Y$, and write $X \searrow Y$,
if there is a sequence of elementary collapses
\[
X = X_0 \searrow_e X_1 \searrow_e X_2 \searrow_e \ldots\ldots \searrow_e X_p = Y.
\] 
If $X \searrow Y$, then we can also say that $Y$ \textit{expands to} $X$, and write $Y \nearrow X$. Similarly, an elementary 
collapse $X \searrow_e Y$ can also be called an \textit{elementary expansion} from $Y$ to $X$, which we denote by 
$Y \nearrow^e X$.  
\end{elem.collapse}

Typically, regular neighborhoods are defined in the literature using the notion of derived subdivision, which involves taking triangulations of PL spaces. However, in this appendix, we opt for a triangulation-free characterization of regular neighborhoods, which we will give in the next definition. 

\theoremstyle{definition} \newtheorem{reg.neighborhood}[pl.sub]{Definition}

\begin{reg.neighborhood} \label{reg.neighborhood} 
Consider a PL manifold $M$ of dimension $m$, possibly with boundary, and let $X$ be a compact PL subspace of $M$ contained in 
$M - \partial M$. A PL subspace $R$ of $M$ is called a \textit{regular neighborhood} of $X$ in $M$ if: 

\begin{itemize}
\item[(i)] $R \subseteq M - \partial M$ and $R$ is a compact neighborhood of $X$ in $M$. 

\item[(ii)] $R$ is an $m$-dimensional PL manifold with boundary. 

\item[(iii)] $R \searrow X$ (or equivalently, $X \nearrow R$). 
\end{itemize}
\end{reg.neighborhood}

Take a PL manifold $M$. It is a standard fact of PL topology that any compact PL subspace 
$X \subseteq M - \partial M$ admits a regular neighborhood in $M$
(see Theorem 3.8 in \cite{plhud} for a proof of this result).  
In fact, as also proven in Theorem 3.8 of \cite{plhud},
any two regular neighborhoods $R_0$ and $R_1$ of $X$ in $M$ are \textit{ambient isotopic,} i.e., there is an 
isotopy $h_t: M \rightarrow M$ of PL homeomorphisms such that $h_0 = \mathrm{Id}_M$, $h_1(R_0) = R_1$, and
$h_t(x) = x$ for all $x \in X$ and $t \in [0,1]$.  

\theoremstyle{definition} \newtheorem{reg.thin}[pl.sub]{Remark}

\begin{reg.thin} \label{reg.thin}
Consider again a PL manifold $M$, possibly with $\partial M \neq \varnothing$. 
Using standard subdivision techniques, it can be proven that any compact PL subspace $X \subseteq M - \partial M$ 
admits regular neighborhoods that are arbitrarily \textit{thin,} i.e.,
 if $U$ is an arbitrary open set in $M$ containing $X$,
then it is always possible to find a regular neighborhood $R$ of $X$ such that $R \subset U$. 
\end{reg.thin}

The following proposition (whose proof can be found in \cite{Br}) provides a useful link
between regular neighborhoods and mapping cylinders.

\theoremstyle{plain} \newtheorem{map.cylinder}[pl.sub]{Proposition}

\begin{map.cylinder} \label{map.cylinder}
Let $M$ be a PL manifold (possibly with $\partial M \neq \varnothing$) and 
let $X$ be a compact PL subspace contained in $M - \partial M$.
If $R$ is a regular neighborhood of $X$ in $M$, then $R$ is homeomorphic to the mapping cylinder $\mathrm{Cyl}(f)$ 
of a piecewise linear map $f: \partial R \rightarrow X$. 
\end{map.cylinder}

The previous result is proven in \cite{Br} assuming that $M$ is an arbitrary PL space, not necessarily a PL manifold. However, in this paper (e.g., the proof of Lemma \ref{marking.germ}), we only needed this result in the case when $M$
was assumed to be a PL manifold.  We close this section of the appendix with the following result regarding nested regular neighborhoods, which we used in
Example \ref{nested.regular}. A proof of this proposition can be found in Chapter 3 of \cite{RS}. 

\theoremstyle{plain} \newtheorem{nested.regular.thm}[pl.sub]{Proposition}

\begin{nested.regular.thm} \label{nested.regular.thm}
Fix  a PL manifold $M$ (possibly with $\partial M \neq \varnothing$)
and let $X$ be a compact PL subspace in $M - \partial M$. If $R_0$ and $R_1$ are two 
regular neighborhoods of $X$ in $M$ such that $R_0 \subset \mathrm{Int}\hspace{0.02cm}R_1$, then there exists 
a PL homeomorphism $h: \partial R_0 \times [0,1] \rightarrow \mathrm{cl}(R_1 - R_0)$ such that $h_0(x) = x$ for all
$x \in \partial R_0$ and $h_1(\partial R_0) = \partial R_1$.  
\end{nested.regular.thm}

\section{Subdivision methods} \label{subdivision.gl}

The purpose of this appendix is to outline the subdivision methods developed in \cite{GL}. To do so, we first need to fix the following notation:   

\begin{itemize}
\item[(i)]  $\Delta_{\mathrm{inj}}$ will denote the subcategory of $\Delta$ whose set of objects is equal to $\mathrm{Obj}(\Delta)$ and whose morphisms are all the injective functions in $\Delta$. 

\item[(ii)]  Let $\mathcal{I}: \Delta \hookrightarrow \mathbf{PL}$ be the
inclusion functor defined in Convention \ref{first.conv}, and let 
$\mathcal{I}_{\mathrm{inj}}$ denote the restriction of $\mathcal{I}$ on $\Delta_{\mathrm{inj}}$. 
For any PL set $\mathcal{F}: \mathbf{PL}^{op} \rightarrow \mathbf{Sets}$, we will denote by $\mathcal{F}_*$ the \textit{semi-simplicial set} obtained by 
pre-composing $\mathcal{F}$ with the inclusion $\mathcal{I}_{\mathrm{inj}}^{op}: \Delta^{op}_{\mathrm{inj}} \hookrightarrow \mathbf{PL}^{op}$.  

\item[(iii)]  $\Delta^p_*$ will denote the 
semi-simplicial set induced by the usual order on the set of vertices of 
the geometric simplex $\Delta^p$.    
\end{itemize}

For any positive integer $r$, let $\mathrm{sd}^r\Delta^p$ be the $r$-th barycentric subdivision of $\Delta^p$, and denote by $v_{F}$ the barycentric point of a face $F$ of $\mathrm{sd}^{r-1}\Delta^p$. 
We can define an order relation on the set of vertices of $\mathrm{sd}^r\Delta^p$ by setting $v_{F_0} \leq v_{F_1}$ if and only if 
$F_0$ is a face of $F_1$. 
We will denote by $\mathrm{sd}^r\Delta^p_*$ the semi-simplicial set induced by this order. Note that any ascending chain
\[
v_{F_0} < v_{F_1} < \ldots < v_{F_k}
\]
of $k+1$ vertices will span a $k$-simplex of $\mathrm{sd}^r\Delta^p$.  
Also, it is straightforward to verify that $\mathrm{sd}^r\Delta^p_*$ is isomorphic to the usual $r$-th barycentric subdivision of $\Delta^p_*$ (see \cite{RW} for a treatment of barycentric subdivisions of semi-simplicial sets). 
 
In Proposition \ref{subdivision.map} below, we will collect a couple of results proven in Section \S 2.4 of \cite{GL}. 
Before stating this proposition, we need to introduce some notation: 
Fix a PL set $\mathcal{F}$ and suppose that $f: \Delta^p_* \rightarrow \mathcal{F}_*$ 
is the classifying map of a $p$-simplex $W \in \mathcal{F}_p$ of the semi-simplicial set $\mathcal{F}_{*}$.
Then, given any positive integer $r $, we define $\widetilde{f}_r: \mathrm{sd}^r\Delta^p_* \rightarrow \mathcal{F}_*$
to be the morphism of semi-simplicial sets which sends  a $k$-simplex $v_{F_0} < v_{F_1} < \ldots < v_{F_k}$
of $\mathrm{sd}^r\Delta^p_*$ to the pull-back of $W$ along the composition  
\[
\xymatrix{ 
\Delta^k \ar[r]^{\hspace{-1.7cm}\cong} & \mathbf{conv}(v_{F_0}, v_{F_1} , \ldots, v_{F_k}) \hspace{0.1cm} \ar@{^{(}->}[r] & \Delta^p,
}
\]
where $\mathbf{conv}(v_{F_0}, v_{F_1} , \ldots, v_{F_k})$ is the convex hull of the points $v_{F_0}, v_{F_1} , \ldots, v_{F_k}$,
and the first map is the linear isomorphism 
which maps the vertices of $\Delta^k$ to those
of  $\mathbf{conv}(v_{F_0}, v_{F_1} , \ldots, v_{F_k})$ in an order-preserving fashion. 
Borrowing the terminology from \cite{GL}, we call this map $\widetilde{f}_r$ \textit{the classifying map of $W$ with respect to the triangulation 
$(\mathrm{sd}^r\Delta^p, \mathrm{Id}_{\Delta^p})$
of} $\Delta^p$ (see Definition 2.19 in \cite{GL}).

\theoremstyle{plain} \newtheorem{subdivision.map}[pl.sub]{Proposition}

\begin{subdivision.map} \label{subdivision.map}
For any PL set $\mathcal{F}: \mathbf{PL}^{op} \rightarrow \mathbf{Sets}$, there 
exists an explicit map $\rho: |\mathcal{F}_*| \rightarrow  |\mathcal{F}_*|$ satisfying the following properties: 

\begin{itemize}
\item[(i)] There exists a homotopy $\mathcal{H}$ between the identity map $\mathrm{Id}_{|\mathcal{F}_{*}|}$
and $\rho$. 

\item[(ii)]  For any $p$-simplex $W$ of $\mathcal{F}_*$ and any positive integer $r$, 
the following diagram is commutative: 
\begin{equation}
\xymatrix{
|\Delta^p_*| \ar[r]^{|f|} \ar[d]_{\cong} & |\mathcal{F}_{*}| \ar[d]^{\rho^r} \\
|\mathrm{sd}^r\Delta^p_*| \ar[r]^{\hspace{0.1cm}|\widetilde{f}_r|} &  |\mathcal{F}_{*}|.}
\end{equation}
In this diagram, the top map is the geometric realization of the classifying map
$f: \Delta^p_* \rightarrow \mathcal{F}_*$ of $W$,
the right-vertical map is the $r$-fold composition of $\rho$ with itself, the left-vertical map 
is the inverse of the canonical homeomorphism from $|\mathrm{sd}^r\Delta^p_*|$ to $|\Delta^p_*|$, and
the bottom map is the geometric realization of the map  $\widetilde{f}_r$ 
defined before the statement of this proposition. 
\end{itemize}
\end{subdivision.map}

In \cite{GL}, we call the map $\rho: |\mathcal{F}_*| \rightarrow  |\mathcal{F}_*|$ appearing in this proposition
\textit{the subdivision map of} $\mathcal{F}$.
An explicit construction of the homotopy $\mathcal{H}$ from part (i) is given in the proof of Proposition 2.25 in \cite{GL}. 
On the other hand, part (ii) is a special case of Proposition 2.27 of \cite{GL}.
Namely, it is the special case obtained by applying that proposition to the PL space $\Delta^p$ and the canonical triangulation of $\Delta^p$.

\theoremstyle{definition} \newtheorem{subdivision.remark}[pl.sub]{Remark}

\begin{subdivision.remark} \label{subdivision.remark}
Consider an arbitrary PL set $\mathcal{F}$. The construction of the subdivision map $\rho: |\mathcal{F}_*| \rightarrow  |\mathcal{F}_*|$
of $\mathcal{F}$ is given in Note 2.23 of \cite{GL}. 
Also, as mentioned before, the 
homotopy $\mathcal{H}$
appearing in part (i) of Proposition \ref{subdivision.map}
was constructed in the proof of Proposition 2.25 of \cite{GL}. 
Now, suppose
that $\mathcal{F}'$ is a PL subset of $\mathcal{F}$, and let 
$\rho': |\mathcal{F}'_*| \rightarrow  |\mathcal{F}'_*|$
be the subdivision map of $\mathcal{F}'$.
Moreover, let $\mathcal{H}'$ be the homotopy 
between $\mathrm{Id}_{|\mathcal{F}'_{*}|}$ and $\rho'$ obtained by following the construction given in the proof of Proposition 2.25 of \cite{GL}.  
Even though it is not stated as a result in \cite{GL}, 
we claim that the constructions 
given in Note 2.23 and the proof of Proposition 2.25 of \cite{GL}
are respectful with respect to PL subsets. In other words, 
the maps $\rho$, $\mathcal{H}$ corresponding to $\mathcal{F}$ and the maps 
$\rho'$, $\mathcal{H}'$ corresponding to $\mathcal{F}'$ are related in the following manner:  
\begin{equation} \label{plset.pair}
\rho|_{|\mathcal{F}'_*|} = \rho'
\qquad \mathcal{H} |_{[0,1] \times |\mathcal{F}'_*|} = \mathcal{H}'.
\end{equation}
We reiterate that this fact was not presented as a result in \cite{GL}. 
Nevertheless, a careful inspection of the constructions 
given in Note 2.23 and the proof of Proposition 2.25 in \cite{GL}
would reveal to the reader that the identities 
in (\ref{plset.pair}) indeed hold for any pair $(\mathcal{F}, \mathcal{F}')$ of PL sets. 
\end{subdivision.remark}

We are now going to use Proposition \ref{subdivision.map} 
(as well as Remark \ref{subdivision.remark}) to prove Proposition \ref{subdivision.cover}, which we stated without proof 
back in Section \S \ref{first.section}.

\begin{proof}[Proof of Proposition \ref{subdivision.cover}]

Let $\mathcal{F}$ be a PL set and suppose that $\mathcal{F}'$
is a PL subset of $\mathcal{F}$ satisfying the following property: 
\begin{itemize}

\item[$\cdot$] For any PL space $P$ and any element $W$ of $\mathcal{F}(P)$, there exists an open cover $\mathcal{U}$ of $P$ such that, for each open set $U \in \mathcal{U}$, the restriction $W_U$ is an element of $\mathcal{F}'(U)$.    

\end{itemize} 
 Using this assumption, we will prove that the map 
 $\mathcal{F}'_{\bullet} \hookrightarrow \mathcal{F}_{\bullet}$ between
 underlying simplicial sets is a weak homotopy equivalence.  
We will do this by first showing that the corresponding inclusion of semi-simplicial sets $\mathcal{F}'_* \hookrightarrow \mathcal{F}_*$ 
is a weak homotopy equivalence. 
Consider then a continuous map $g: \Delta^q \rightarrow |\mathcal{F}_*|$ such that 
$g(\partial\Delta^q) \subseteq |\mathcal{F}'_*|$. 
By the compactness of $\Delta^q$, there is a 
finite semi-simplicial set $L_* \subseteq \mathcal{F}_*$ such that $g(\Delta^q) \subseteq |L_*|$ and
$g(\partial\Delta^q) \subseteq |L_*\cap \mathcal{F}'_*|$. 
By our assumptions, and using the fact that $L_*$ is finite, we can find a positive integer $r$ such that, for any element
$W \in L_p \subseteq \mathcal{F}(\Delta^p)$ and any simplex $\sigma$ of $\mathrm{sd}^r\hspace{0.05cm}\Delta^p$, 
the restriction of $W$ over $\sigma$ 
is in $\mathcal{F}'(\sigma)$.   
It then follows from part (ii) of Proposition \ref{subdivision.map} that the image of the map
$\widetilde{g} := \rho^r \circ g$ is contained in $|\mathcal{F}'_*|$, 
where $\rho^r$ is the $r$-fold composition of the 
subdivision map $\rho: |\mathcal{F}_*| \rightarrow  |\mathcal{F}_*|$ 
with itself. 
But, if $\mathcal{H}$ is the homotopy 
from $\mathrm{Id}_{|\mathcal{F}_*|}$ to $\rho$
provided by part (i) of Proposition \ref{subdivision.map}, we obtain
a homotopy $\widetilde{\mathcal{H}}$ from $g$ to $\widetilde{g}$ by setting
\[
\widetilde{\mathcal{H}}_t := (\mathcal{H}_t)^r \circ g,
\]
where, for each $t \in [0,1]$, the map
$(\mathcal{H}_t)^r$ is the $r$-fold composition of $\mathcal{H}_t$ with itself. 
Furthermore, by the observations we made
in Remark \ref{subdivision.remark}, the homotopy
$\widetilde{\mathcal{H}}$ has the property that $\widetilde{\mathcal{H}}_t(\partial \Delta^q) \subseteq |\mathcal{F}'_*|$
for all $t \in [0,1]$. We can therefore conclude
that the inclusion 
$\mathcal{F}'_* \hookrightarrow \mathcal{F}_*$ is a weak homotopy equivalence. 

To finish this proof, note that the natural maps 
$|\mathcal{F}_{*}| \rightarrow |\mathcal{F}_{\bullet}|$ and $|\mathcal{F}'_{*}| \rightarrow |\mathcal{F}'_{\bullet}|$ 
obtained by collapsing degeneracies fit into the following commutative diagram:
\[
\xymatrix{ 
|\mathcal{F}'_{*}| \hspace{0.1cm} \ar@{^{(}->}[r] \ar[d]_{\simeq}  & |\mathcal{F}_{*}| \ar[d]^{\simeq}  \\ 
|\mathcal{F}'_{\bullet}| \hspace{0.1cm} \ar@{^{(}->}[r] & |\mathcal{F}_{\bullet}|.}
\]
Since both vertical maps are weak homotopy equivalences (see \cite{RSD}), it follows that 
$\mathcal{F}'_{\bullet} \hookrightarrow \mathcal{F}_{\bullet}$ is also a weak homotopy equivalence. 
\end{proof}

\end{document}